\documentclass[a4paper,11pt,twoside,english]{article}
\usepackage{marginnote}
\usepackage[utf8x]{inputenc}
\usepackage[T1]{fontenc}
\usepackage{color}
\usepackage{lmodern}
\usepackage{graphicx}
\usepackage{amsmath,amssymb,amsthm}
\usepackage{mathrsfs}  
\usepackage{dsfont} 

\usepackage{caption}
\usepackage{subcaption}
\usepackage{enumitem}
\usepackage{setspace}
\usepackage{xspace} 
\usepackage{blkarray}
\usepackage[subnum]{cases}
\usepackage{stmaryrd}
\usepackage{bbm} 
\usepackage[margin=2cm]{geometry}

\definecolor{darkred}{rgb}{0.8,0,0.1}
\definecolor{darkgreen}{rgb}{0,0.8,0.2}
\usepackage[colorlinks,linkcolor=darkred,citecolor=darkgreen,unicode=true]{hyperref}
\usepackage[all]{hypcap}

\usepackage{cleveref} 

\newcommand{\titel}{Brown measures of deformed $L^\infty$-valued circular elements}

\hypersetup{
	pdftitle={\titel},
	pdfauthor={Johannes Alt, Torben Krüger}
}

\numberwithin{equation}{section}

\newtheorem{theorem}{Theorem}[section]
\newtheorem{lemma}[theorem]{Lemma}
\newtheorem{proposition}[theorem]{Proposition}
\newtheorem{corollary}[theorem]{Corollary}
\theoremstyle{definition}
\newtheorem{remark}[theorem]{Remark}
\newtheorem{definition}[theorem]{Definition}
\newtheorem{example}[theorem]{Example}
\numberwithin{equation}{section}

\newcommand{\D}{\mathbb{D}}
\newcommand{\R}{\mathbb{R}}  
\ifdefined\C\renewcommand{\C}{\mathbb{C}}\else\newcommand{\C}{\mathbb{C}}\fi 
\renewcommand{\Im}{\mathrm{Im}\,} 
\renewcommand{\Re}{\mathrm{Re}\,} 
\renewcommand{\i}{\mathrm{i}} 
\newcommand{\N}{\mathbb{N}}  
\newcommand{\E}{\mathbb{E}}  
\newcommand{\eps}{\varepsilon} 
\newcommand*{\defeq}{\mathrel{\vcenter{\baselineskip0.5ex \lineskiplimit0pt\hbox{\scriptsize.}\hbox{\scriptsize.}}}=}

\newcommand{\pt}{\partial}

\DeclareMathOperator{\sign}{sign} 

\DeclareMathOperator{\supp}{supp}

\DeclareMathOperator{\linspan}{span}

\newcommand{\DD}{\mathbb{D}}

\renewcommand{\rm}{\mathrm} 

\newcommand{\bels}[2] {
        \begin{equation} \label{#1} \begin{split} 
                #2 
        \end{split} \end{equation}
        }
\newcommand{\bes}[1]{
        \begin{equation*}  \begin{split} 
                #1 
        \end{split} \end{equation*}
        }

\newcommand{\bbm}{\mathbbm} 
\renewcommand{\cal}{\mathcal} 
\newcommand{\scr}{\mathscr}  
\renewcommand{\frak}{\mathfrak} 
\newcommand{\ol}[1]{\overline{#1} \!\,} 
\newcommand{\wh}{\widehat}
\newcommand{\wt}{\widetilde}

\newcommand{\ee}{\mathrm{e}} 
\newcommand{\ii}{\mathrm{i}} 
\newcommand{\dd}{\mathrm{d}}

\newcommand{\p}[1]{({#1})}
\newcommand{\pb}[1]{\bigl({#1}\bigr)}
\newcommand{\pB}[1]{\Bigl({#1}\Bigr)} 
\newcommand{\pbb}[1]{\biggl({#1}\biggr)}

\newcommand{\sbb}[1]{\biggl[{#1}\biggr]}

\newcommand{\cb}[1]{\bigl\{{#1}\bigr\}}
\newcommand{\cB}[1]{\Bigl\{{#1}\Bigr\}}
\newcommand{\cbb}[1]{\biggl\{{#1}\biggr\}}

\newcommand{\abs}[1]{\lvert #1 \rvert}

\newcommand{\absbb}[1]{\bigg\lvert #1 \bigg\rvert}

\newcommand{\norm}[1]{\lVert #1 \rVert}

\newcommand{\avg}[1]{\langle #1 \rangle}
\newcommand{\avgb}[1]{\big\langle #1 \big\rangle}

\newcommand{\avgbb}[1]{\bigg\langle #1 \bigg\rangle}

\newcommand{\db}[1]{\llbracket #1 \rrbracket}

\newcommand{\scalar}[2]{\langle{#1} \mspace{2mu}, {#2}\rangle}

\newcommand{\scalarB}[2]{\Big\langle{#1} \,\mspace{2mu},\, {#2}\Big\rangle}
\newcommand{\scalarbb}[2]{\bigg\langle{#1} \,\mspace{2mu},\, {#2}\bigg\rangle}

\DeclareMathOperator{\diag}{diag}

\DeclareMathOperator{\Tr}{Tr}

\DeclareMathOperator{\re}{Re}
\DeclareMathOperator{\im}{Im}

\DeclareMathOperator*{\essinf}{ess\,inf}	
\DeclareMathOperator{\dist} {dist}                
\DeclareMathOperator{\spec}{Spec}						

\newcommand{\1} {\mspace{1 mu}}
\newcommand{\2} {\mspace{2 mu}}

\newcommand{\qq}[1]{\llbracket #1 \rrbracket} 

\newcommand{\mtwo}[2]
{
\left(
\begin{array}{cc}
#1 
\\
#2
\end{array}
\right)
}

\newcommand{\vtwo}[2]
{
\left(
\begin{array}{c}
#1 
\\
#2
\end{array}
\right)
}

\makeatletter
 \newcommand{\linkdest}[1]{\Hy@raisedlink{\hypertarget{#1}{}}}
\makeatother

\makeatletter
\def\blfootnote{\xdef\@thefnmark{}\@footnotetext}
\makeatother

\begin{document}
\blfootnote{Date: \today}
\blfootnote{Keywords:  Brown measure, matrix Dyson equation, non-Hermitian random matrix } 
\blfootnote{MSC2020 Subject Classifications: 46L54, 60B20, 15B52.}

\title{\vspace{-0.8cm}{\textbf{\titel}}} 
\author{Johannes Alt$^{\text{a,} }$\thanks{Funding from the Deutsche Forschungsgemeinschaft (DFG, German Research Foundation) under Germany's Excellence Strategy - GZ 2047/1, project-id 390685813 is gratefully acknowledged. \newline \hspace*{0.5cm} Email: \href{mailto:johannes.alt@iam.uni-bonn.de}{johannes.alt@iam.uni-bonn.de}} 
 \hspace*{0.5cm} \and \hspace*{0.5cm} Torben Krüger$^\text{b,}$\thanks{
Financial support 
from VILLUM FONDEN Young Investigator Award  (Grant No. 29369) is gratefully acknowledged.  
\newline \hspace*{0.5cm} Email: \href{mailto:torben.krueger@fau.de}{torben.krueger@fau.de}}}
\date{}

\maketitle

\vspace*{-0.2cm} 

\noindent \hspace*{3.54cm} $^\text{a}$Institute for Applied Mathematics, University of Bonn \\ 
\noindent \hspace*{3.54cm} $^\text{b}$Department of Mathematics, FAU Erlangen-Nürnberg \\

\begin{abstract} 
{We consider the Brown measure of $a+\frak{c}$, where $a$ lies in a commutative tracial von Neumann algebra $\cal{B}$ and $\frak{c}$ is a $\cal{B}$-valued circular element. Under certain regularity conditions on $a$ and the covariance of $\frak{c}$ this Brown measure has a density with respect to the Lebesgue measure on the complex plane which is real analytic apart from jump discontinuities at the boundary of its support. With the exception of finitely many singularities this one-dimensional spectral edge is real analytic.  We provide a full description of all possible edge singularities as well as all points in the interior, where the density vanishes. 
The edge singularities are classified in terms of their local edge shape while internal zeros of the density are classified in terms of the shape of the density locally around these points. We also show that each of these countably infinitely many  singularity types occurs for an  appropriate choice of $a$ when $\frak{c}$  is a standard circular element. The Brown measure of $a+\frak{c}$ arises as the  empirical spectral distribution of a diagonally deformed non-Hermitian random matrix with independent entries when its dimension tends to infinity.
}
\end{abstract}

\tableofcontents

\section{Introduction} 

The empirical spectral distribution of a non-Hermitian random matrix typically converges to a non-random  probability distribution $\sigma$, the limiting spectral measure, on the complex plane as its dimension tends to infinity. The most prominent instance of this phenomenon is the circular law, stating that the empirical eigenvalue distributions of a suitably   normalised  sequence of matrices $X_n \in \C^{n \times n}$ with centered i.i.d.\ entries converge to the uniform distribution on the complex unit disk $\D_1$ in the limit $n \to \infty$ \cite{Girko1984,bai1997} (see \cite{tao2010} for optimal moment conditions and \cite{BordenaveChafai2012} for a review). 

For the purpose of identifying the limiting spectral measure and computing mixed moments of $X_n$ and $X_n^*$ as $n \to \infty$, according to free probability theory the limit of the sequence $X_n$ is appropriately described    by a circular  element $\frak{c}$. This $\frak{c}$ is infinite dimensional operator  in a $W^*$-probability space $\cal{A}$ with faithful, tracial state  $\avg{\2\cdot\2}$ and satisfies 
$\lim_{n \to \infty} \frac{1}{n} \Tr p(X_n, X_n^*) = \avg{p(\frak{c},\frak{c}^*)}$ for all polynomials $p$ in two non-commutative variables. In agreement with the circular law, the uniform distribution on $\D_1$ is the Brown measure of $\frak{c}$, where the Brown measure is a generalisation of the spectral measure to non-normal operators \cite{Brown1986,HaagerupLarsen2000}. When a deterministic deformation matrix $A=A_n$ is added to $X=X_n$, the associated spectral distribution of the random matrix $A + X$ is asymptotically close to the Brown measure $\sigma_{A+ \frak{c}}$  of  the sum of an embedding of $A$ into the $W^*$-probability space and a circular element  $\frak{c}$ that is $\ast$-free from $A$ \cite{Sniady2002}. 
In this case, the support of  $\sigma_{A+ \frak{c}}$ coincides with the closure $\ol{\mathbb{S}}_{A}$, where $\mathbb{S}_a 
:= \{ \zeta \in \C \, : \, \avg{(a-\zeta)^{-1}(a^*-\ol{\zeta})^{-1})} > 1\}$.  This observation goes back to \cite{Khoruzhenko1996}. 

That $\ol{\mathbb{S}}_a$ is the support of the Brown measure of $a + \frak{c}$ for a deformation $a\in \cal{A}$, which is $\ast$-free from $\frak{c}$, was shown  in  \cite{BordenaveCaputoChafai2014}  for normal deformations  
and extended to general deformations in \cite{BordenaveCapitaine2016,Zhong2021}.  
Subsequently, the regularity of $\sigma_{a+\mathfrak c}$ has been analysed. The measure is absolutely continuous with respect to the Lebesgue measure on $\C$ 
\cite{BelinschiYinZhong2024} and the density is strictly positive and real analytic on $\mathbb{S}$ \cite{Zhong2021,HoZhong2023}.  
In  \cite{ErdosJi2023} it has been shown that at the edge of $\mathbb{S}_a$, the density typically possesses a jump discontinuity and that around its zeros within its support $\ol{\mathbb{S}}_a$  
it grows at most quadratically with a corresponding quadratic lower bound  either on a  two-sided cone or in a whole neighbourhood. 

In the present work, we refine the distinction between these three possibilities by  establishing a full classification of the  points of vanishing density  within $\ol{\mathbb{S}}_a$. The edge singularities, i.e.\ the zeros of the density in $\partial \ol{\mathbb S}_a$ are classified in terms of the local  shape of $\partial \ol{\mathbb S}_a$, while  zeros in the interior of $\ol{ \mathbb S}_a$ are classified in terms of the local growth of the density. For both cases countably infinite singularity types are possible. 
  Finally, we prove that each case  of these possible   singularity types does exist in the Brown measure 
$\sigma_{a + \mathfrak c}$ for a suitable choices of $a$. 

In addition to providing a clear picture of the regularity of the Brown measure  this full classification of the singularity types is important because of their potential to distinguish 
the universality classes of the local eigenvalue point processes near these singularities. 
To that end, it is insightful to review some known results in the analogous Hermitian model, the deformed Wigner matrices $A+H$, where $A$ is a deterministic real diagonal deformation to a sequence of $n \times n$-Wigner matrices $H$ with i.i.d. entries, up to  Hermitian symmetry constraints.   The spectral density for such matrices is well approximated for large dimensions by the spectral measure of $A +\frak{s}$, where $\frak{s}$ is a  semicircular element, $\ast$-free of $A$. 
 This follows from \cite{Voiculescu1986} and \cite{Voiculescu1991} if $H$ has Gaussian entries and for a general Wigner matrix from \cite{Dykema1993}.  
 The spectral measure of $A+ \mathfrak s$  has a density $\rho$ with respect to the Lebesgue measure on the real line which is real analytic wherever  it is positive. Under some regularity assumptions on the deformation, the  singularities of $\rho$, that form when the density vanishes, can only be of algebraic degree two or three \cite{Ajanki_QVE_Memoirs}. 
More precisely, they exhibit a one-sided square root growth at the spectral edge and a two-sided cubic root growth for singularities in the interior of the support of $\rho$.  
These singularity types determine the corresponding  universality classes of the local spectral statistics that are associated with the Airy \cite{MR3405746} and Pearcey process \cite{Cusp1EKS,Cusp2EKS}, respectively. 
For Gaussian entries this correspondence between the singularity type of the density and the local statistics had  already been established   for regular square root edges  without deformation in \cite{MR1257246,MR1385083} and subsequently  with deformation in \cite{MR3500269}, as well as for the cubic root cusp in \cite{MR1662382, BH,AM,MR3500269}.

Allowing the entries in the randomness $H =(h_{ij})_{i,j=1}^n$ to have differing distributions, introduces an additional structure designed to model systems with non-trivial underlying geometry. 
For the deformed  Wigner type matrices $A+H$ the asymptotic spectral density $\rho$ depends on the variance profile $R=(\E |h_{ij}|^2)_{i,j=1}^n$.  To define the limiting object that replaces the semicircular element from the i.i.d. case in this setting, the  additional structure of a conditional expectation $E:\cal{A} \to \cal{B}$ with values in the commutative subalgebra $\cal{B}=\C^n \subset \cal{A}$ with entrywise multiplication is introduced on the $W^*$-probability space. This is the setup of $\cal{B}$-valued free probability theory   \cite{Voiculescu1995}  and the operator  $\frak{s}$, for which $\rho$ is now the spectral density of $A + \frak{s}$,   is a $\cal{B}$-valued circular element whose covariance encodes the variance profile $R$  through the identity $Rb = E[ \frak{s}b \frak{s}]$ for $b \in \cal{B}$. 
The classification of singularities of $\rho$ in terms of the edge - cusp - dichotomy, as well as the associated appearance of the Airy -  and Pearcey - universality classes, persists in this more general model \cite{AEKS2018,Cusp1EKS,Cusp2EKS}. 

Introducing variance profiles in the non-Hermitian setup, require the same extension to a $\cal{B}$-valued probability space $(\cal{A}, \cal{B},E)$  with $\cal{B}=\C^n$. The empirical spectral measure of a random matrix  $X =(x_{ij})_{i,j=1}^n$ with centred independent entries of variances $s_{ij}:=\E\1 \abs{x_{ij}}^2$ is well approximated by the Brown measure of  a non-normal operator $\frak{c}$ that  is a $\C^n$-valued circular element and covariance $S=(s_{ij})_{i,j=1}^n$ \cite{AK_pseudo}.  Many entries may have variance zero, which is of particular interest when studying non-Hermitian band matrices \cite{Jain2021_non_Hermitian_random_band,Jana2022_nonHermitian_random_band}.
 Here the Brown measure is supported on a disk, radially symmetric and its density has  a jump at the edge, but is non-constant  on its support in general \cite{Cook2018,Altcirc}.  Adding  a diagonal  deformation $A=(a_i\delta_{ij})_{i,j=1}^n$   to  $X_n$ breaks the asymptotic radial symmetry of the Brown measure of $A + \frak{c}$ and the spectrum  concentrates  on an area in the complex plane that is contained in the asymptotic pseudospectrum, whose outer boundary was identified in   \cite{AEKN_Kronecker}. In the companion paper \cite{AK_pseudo} we show that the spectrum asymptotically fills the entire area enclosed by this boundary, i.e.\ for deformed random matrices with variance profile the support of the limiting spectral measure and the asymptotic pseudospectrum coincide. 

Recently, it has been shown that for non-Hermitian matrices $A + X$, where $X \in \C^{n \times n}$ has i.i.d. entries and $A$ is deterministic, the local universality classes correspond to the local shape of the limiting spectral density $\sigma$ as well in two cases. In the bulk, i.e.\  when $\sigma$ is positive, the local eigenvalue statistics coincides with the Ginibre bulk statistics. For Gaussian $X$ this was shown in \cite{zhang2024bulkuniversalitydeformedginues} and for non-Gaussian entries and 
$A=0$  under the four moment matching assumption in \cite{Tao_2015}  and without this assumption in \cite{maltsev2023bulkuniversalitycomplexnonhermitian,Osman_real,dubova2024bulkuniversalitycomplexeigenvalues}.  
At spectral edge points with a jump discontinuity of $\sigma$, the Ginibre edge   statistic emerges, which has been proved first for Gaussian matrices in \cite{liu2024repeatederfcstatisticsdeformed} and then for non-Gaussian entries in \cite{Cipolloni_2020} for $A=0$ and in \cite{CampbellCipolloniErdosJi2024} for nonzero deformations. 
This phenomenon is expected to generalise to matrices  $X$ with variance profile. In contrast to the Hermitian setting the class of  singularities  that emerge as points where the  limiting spectral density vanishes, or equivalently the zeros of the density of the Brown measure of $A + \frak{c}$ with a $\C^n$-valued circular element $\frak{c}$ inside its support, is much more rich. As we show in this work, it contains an infinite family of such possible singularity types both at the spectral edge as well as in the interior of the support. We conjecture that each local singularity type corresponds to a specific universality class for the local statistics of the corresponding random matrix ensemble. 
 
 Generalising from the setting of finite dimensional $\cal{B}=\C^n$, in this work we
  consider the general framework of deformed $\cal{B}$-valued circular elements $a+\frak{c}$, where $a \in \cal{B}$ is the deformation, $\frak{c}$ the $\cal{B}$-valued circular element,  $\cal{B} \subset \cal{A}$ is a commutative subalgebra of a  $W^*$-probability space $\mathcal A$ and $E \colon \mathcal A \to \mathcal B$ is a conditional expectation.   
 The covariance $S \colon \cal{B} \to \cal{B}$ of $\mathfrak c$ is defined through $S[b] = E[\mathfrak c b \mathfrak c^*]$ for all $b \in \mathcal B$ and $E[\mathfrak c b \mathfrak c]=0$. 
 Higher order $\cal{B}$-valued free cumulants of $\frak{c}$ and $\frak{c}^*$ vanish, making it a generalized circular element as introduced in \cite{Sniday2003_Multinomial}. $\cal{B}$-valued circular elements were previously analyzed in  \cite{MR2198797}   under the name $\cal{B}$-circular elements and were introduced in  \cite{MR2044226}.
 For the Brown measure $\sigma$ of $\frak{c}+a$ we provide a classification in terms of the solution to a non-linear equation on $\cal{B}$. Furthermore, we determine the regularity properties of this measure and describe its  singularities.  
We show that  under some regularity assumptions on the variance profile  the Brown measure  $\sigma$ admits a bounded probability density on the complex plane, which is  real analytic and strictly positive on an open domain $\mathbb{S}:= \{  \beta  < 0 \} \subset \C$ with boundary $\partial \mathbb{S}= \{  \beta  = 0 \}$, where $\beta : \C \to \R$ is a continuous function that is real analytic in a neighbourhood of $\partial \mathbb{S}$. From this we obtain that $\partial \mathbb{S}$ is a real analytic variety of dimension at most $1$.  The density of the Brown measure, also denoted by  $\sigma$, vanishes outside the closure of $\mathbb{S}$ and typically has a jump discontinuity at the spectral edge $\partial \mathbb{S}$, except  at critical points of $\beta$, where it vanishes continuously. 
 Consequently, we have generalised the previously known results \cite{Zhong2021,BelinschiYinZhong2024,ErdosJi2023} about the density $\sigma_{a + \mathfrak c}$, 
when  $\mathfrak c$ is a standard circular element, i.e.\ $S = \avg{\2\cdot\2}$ is the trace of the 
argument multiplied by the unit in $\mathcal B$.  
In this case our choice of $\beta$ simplifies to  $\beta(\zeta) = \avg{ \abs{a-\zeta}^{-2}}$, which coincides with the analogous quantity in  \cite{Khoruzhenko1996}.

Near the zeros or singularities of the density, we prove the following behaviour. 
A zero $\zeta$ lies either at the boundary of the support $\ol{\mathbb S}_a$ of the Brown measure, i.e.\ at the spectral edge, or in the interior of $\ol{\mathbb{S}}_a$. In the first case we analyse the local behavior of the level sets of $\beta$  around $\zeta$ to identify  the shape of the edge $\partial \mathbb{S}$ in a neighbourhood of  $\zeta$. In appropriately chosen coordinates $(x,y)$ of the plane this shape is of the form $x^2=y^{k+1}$,  forming an edge singularity of order $k \in \N$.
  In the latter case we classify the singularities in terms of the local shape of the density around its zero $\zeta$.  At these internal singularities the density has  the form $x^2 + y^{2k}$ with order $k \in \N$ in appropriately chosen coordinates~$(x,y)$.

\paragraph{Proof ideas:}  All properties of the Brown measure $\sigma$ and its relationship to $\beta$ established in this paper 
in the general setup are ultimately obtained by our analysis of a $\zeta$-dependent system of two coupled $\mathcal B$-valued equations, called \emph{Dyson equation}.  
It describes the diagonal entries $v_1(\zeta)$ and $v_2(\zeta)$ of the matrix in $\cal B^{2\times 2}$ obtained 
by applying the conditional expectation $E$ entrywise to the resolvent of the Hermitization of $a + \mathfrak c$.  
In the random matrix 
setup, the Hermitization idea goes back to \cite{Girko1984}. See e.g.\ \cite{BelinschiSniadySpeicher2018} 
for its use in the analysis of Brown measures.  
More explicitly, there are two unique\footnote{The existence and uniqueness of $v_1$ and $v_2$ in \eqref{eq:V_equations at eta=0} are proved in Corollary~\ref{cor:uniqueness_v_equations_eta=0} below.} positive functions $v_1,v_2: \mathbb{S} \to \cal{B}$, which satisfy  
\begin{equation} \label{eq:V_equations at eta=0} 
\frac{1}{ v_1(\zeta)}  =   Sv_2(\zeta)  +  \frac{\abs{\zeta-a}^2}{S^*v_1(\zeta) } \,,  \qquad 
\frac{1}{ v_2(\zeta)}  =   S^*v_1(\zeta) +  \frac{\abs{\zeta-a}^2}{ Sv_2(\zeta)}\,, 
\end{equation}  
for each $\zeta \in \mathbb S$. As $\zeta$ approaches  $ \partial \mathbb{S}$ the solution $(v_1(\zeta),v_2(\zeta))$ vanishes. 
If $\mathfrak c$ is a standard circular element, then the $\cal B$-valued system \eqref{eq:V_equations at eta=0} 
simplifies to a scalar-valued equation.
In either case, from $v_1$ the  probability density $\sigma$ inside $\mathbb{S}$ is derived through
\begin{equation} \label{eq:sigma_intro} 
\sigma(\zeta):=-\partial_{\zeta}\avgbb{\frac{ v_1(\zeta)(a-\zeta)}{\pi\2S^*v_1(\zeta)}}\,.
\end{equation} 
Taking the derivative in \eqref{eq:sigma_intro} yields a quadratic form of a non-symmetric operator. 
The main idea for the proof of strict positivity of $\sigma$ in the bulk regime, i.e.\ on $\mathbb{S}$, is to transform the formula for $\sigma$ 
into the quadratic form of a strictly positive operator. 
Near the spectral edge $\partial \mathbb{S}$, the behaviour of $\sigma$ is governed by the quantity $\beta$ from the definition of $\mathbb{S}$. 
In fact, 
 $\beta(\zeta)$ coincides locally around the spectral edge with the isolated eigenvalue of the non-symmetric operator $B_\zeta$ that is closest to zero, 
where $ B_\zeta : \cal{B}\to \cal{B}$ is defined through $B_\zeta f:= \abs{a-\zeta}^2\2f  - Sf$.
A consequence that we derive from this insight is that the jump height of the edge discontinuity of $\sigma$ at the spectral edge is proportional to $\abs{\partial_\zeta \beta}^2$. This requires a careful singular expansion of $v_1, v_2$ at the spectral edge, where the Dyson equation \eqref{eq:V_equations at eta=0} is unstable. A signature of this instability is that $B_\zeta $ is singular for $\zeta \in \partial \mathbb{S}$ and the main contributions to $v_1$ and $v_2$ near $\partial \mathbb{S}$ point into the singular eigendirections of $B_\zeta $. 
Owing to the dependence of $v_1$, $v_2$ and $\beta$ on $S$, 
treating non-constant $S$ and $a$ is a recurring challenge for the analysis in both regimes.  

Finally, we employ an iterative scheme to show that for each $k \in \N$ a finite-dimensional commutative subalgebra $\cal B \subset \cal A$ and a deformation $a \in \cal B$ exist, such that the Brown measure of the deformed standard circular element $a + \frak{c}$  exhibits an edge or internal singularity of  order $k$. 
To perform an iteration step within this scheme, 
we increase the dimension of $\cal{B}$ and add a small carefully chosen perturbation $h$ to a deformation $a$ within the additional dimensions, where $a$ generated a singularity in the Brown measure of $a + \frak{c}$ at the origin. Then by  tuning the parameters that determine $a$ and changing it to $\wh{a}\approx a$ we ensure that $\wh{a}+h$ now generates  a singularity at the origin with a higher order than the  one generated by $a$.

\section{Main results}

We consider an operator-valued probability space $(\mathcal A, E, \mathcal B)$. 
That is, $\mathcal A$ is a unital von Neumann algebra, $\mathcal B \subset \mathcal A$ is a von Neumann subalgebra  
with the same unit and $E \colon \mathcal A \to \mathcal B$ is a positive conditional expectation, i.e.\ 
$E[b] = b$ for all $b \in \mathcal B$, $E[b_1 \mathfrak a b_2] = b_1 E[\mathfrak a] b_2$ for all $\mathfrak a \in \mathcal A$ and 
$b_1$, $b_2 \in \mathcal B$ as well as $E[\mathfrak a]$ is positive for all positive $\mathfrak a \in \mathcal A$. 
We also assume that $\mathcal B$ is equipped with a faithful tracial state $\avg{\2\cdot\2}$ such that  $(\mathcal A, \avg{E[\,\cdot\,]})$ is a tracial $W^*$-probability space. 
We call an element $ \mathfrak{c} \in \cal{A}$
a \emph{$\cal{B}$-valued circular element} if it is centered, i.e. $E[\mathfrak{c}]=0$, all mixed free $\cal{B}$-cumulants\footnote{For the definition of the mixed free $\mathcal B$-cumulants, we refer to \cite[Definition~7, Chapter~9]{MingoSpeicherBook}.} 
of $\mathfrak{c}$ and $\mathfrak{c}^*$ of any order larger than $2$ vanish and if the second order cumulants
satisfy 
\bels{second order cumulant of c}{
E[\mathfrak{c}b\mathfrak{c}] =0\,, \qquad E[\mathfrak{c}^*b\mathfrak{c}^*] =0
}
for all $b \in \cal{B}$. 
Define the operators $S \colon \mathcal B \to \mathcal B$ and $S^* \colon \mathcal B \to \mathcal B$ through 
\begin{equation} \label{eq:def_S_S_star} 
S b := E[\mathfrak c b \mathfrak c^*] , \qquad S^* b:= E[\mathfrak c^* b \mathfrak c] 
\end{equation} 
for all $b \in \mathcal B$. Note that $S^*b = (S(b^*))^*$ for all $b \in \mathcal B$. 
We call $a+\frak{c}$ with some $a \in \cal{B}$ a \emph{deformed $\cal{B}$-valued circular element} with deformation $a$ 
 and covariance $S$. 

For $\mathfrak a \in \mathcal A$ the \emph{Brown measure} $\sigma_{\mathfrak a}$ of $\mathfrak a$ is 
the unique compactly supported probability measure on $\C$ with 
\begin{equation} \label{eq:Brown_measure} 
\int_{\C} \log \abs{\zeta - \xi}\sigma_{\mathfrak a}(\dd \xi) = \log D(\mathfrak a - \zeta) 
\end{equation} 
for all $\zeta \in \C$. Here, $D(\mathfrak a - \zeta)$ denotes the \emph{Fuglede-Kadison determinant} of 
$\mathfrak a - \zeta$. The Fuglede-Kadison determinant of an arbitrary $\mathfrak b \in \mathcal A$ is defined by 
\begin{equation} \label{eq:def_determinant} 
 D(\mathfrak b) := \lim_{\eps \downarrow 0} \exp( \avg{E[\log (\mathfrak b^* \mathfrak b + \eps)^{1/2}]} ) 
\in [0,\infty). 
\end{equation} 
Originally introduced in \cite{Brown1986}, the Brown measure was revived in \cite{HaagerupLarsen2000}. 
An introduction to the Brown measure and the Fuglede-Kadison determinant can be found for example in \cite[Chapter~11]{MingoSpeicherBook}. 

In the following we will consider the case when $\mathcal B$ is commutative. Therefore we may assume that $\mathcal B= L^\infty(\mathfrak X, \Sigma, \mu)$ is the space of $\mu$-almost everywhere bounded functions with respect to a probability measure $\mu$ on $(\mathfrak X, \Sigma)$ and $\avg{b} = \int_{\mathfrak X} b \2\dd\mu$ is the trace of  $b \in\cal{B}$ \cite[Chapter III, Theorem~1.18]{TakesakiBook}. 

We study the Brown measure of a deformed $\cal{B}$-valued circular element. To that end, fix $a \in \cal{B}$ and $\frak{c} \in \cal{A}$ be a  $\cal{B}$-valued circular element that satisfies \eqref{second order cumulant of c} and \eqref{eq:def_S_S_star}. 
Throughout this work we will assume that $S$ and $S^*$ from \eqref{eq:def_S_S_star} are represented by an integral kernel, i.e.\ that there is a measurable function $s \colon \mathfrak X \times \mathfrak X \to [0,\infty)$ such that 
\begin{equation} \label{eq:def_operators_S_and_S_star} 
(S u)(x) = \int_{\mathfrak X} s(x,y) u(y) \mu(\dd y), \qquad (S^* u)(x) = \int_{\mathfrak X} s(y,x) u(y) \mu(\dd y) 
\end{equation} 
for  $x \in \mathfrak X$ and $u \in \cal{B}$. 
Since $S$ and $S^*$ are positive operators and therefore, bounded, after possibly changing $s$ on a $\mu \otimes \mu$-nullset, we assume 
\bels{boundedness of S}{ \sup_{x \in \mathfrak X} \int_{\mathfrak X} s(x,y) \mu(\dd y) < \infty, 
\qquad \qquad \sup_{y \in \mathfrak X} \int_{\mathfrak X} s(x,y) \mu(\dd x) < \infty. }

Under the following regularity assumptions on $a$ and $s$, we prove detailed information about the Brown measure $\sigma = \sigma_{a + \mathfrak c}$ of the deformed $\mathcal B$-valued circular operator $a + \mathfrak c$. 
From here on and throughout the paper, we will impose some of the following assumptions. 
\newcounter{assumptions} 
\begin{enumerate}[label=\textbf{A\arabic*}]
\setcounter{enumi}{\value{assumptions}} 
\item \stepcounter{assumptions} \label{assum:flatness} 
\emph{Block-primitivity of $S$:} There is a constant $C>0$, a primitive matrix $Z=(z_{ij})_{i,j=1}^K \in \{0,1\}^{K 
 \times K}$ with $z_{ii}=1$ for each $i$ and a measurable partition $(I_i)_{i=1}^K$  of $\mathfrak X$   with $\min_i\mu(I_i)\ge \frac{1}{C}$, such that 
\[
\frac{1}{C}  s(x,y) \le  \sum_{i,j=1}^K z_{ij} \bbm{1}_{I_i \times I_j}(x,y) \le Cs(x,y)\,
\]
for all $x$, $y \in \mathfrak X$.  
\end{enumerate} 

We recall that a matrix $Z \in [0,\infty)^{K\times K}$ with nonnegative entries is called \emph{primitive} 
if there is an $L \in\N$ such that all entries of the power $Z^L$ are strictly positive. 
For $a$ and $s$ as above, we define the function 
$\Gamma_{a,s} \colon (0,\infty) \to (0,\infty)$ similarly to  \cite[Section~9]{AEK_Shape}  as 
\begin{equation} \label{eq:def_Gamma}  
\Gamma_{a,s}(\tau) := \pbb{\essinf_{x \in \mathfrak X} \int_{\mathfrak X} \frac{1}{(\tau^{-1} + \abs{a(x) - a(y)} + d_s(x,y))^2} \mu(\dd y)}^{1/2}, 
\end{equation}  
where $d_s(x,y) := \pb{\int_{\mathfrak X} ( \abs{s(x,q) - s(y,q)}^2 + \abs{s(q,x) - s(q,y)}^2 )\mu(\dd q)}^{1/2}$. 
Note that $\Gamma_{a,s}$ is strictly monotonically increasing.

\begin{enumerate}[label=\textbf{A\arabic*}]
\setcounter{enumi}{\value{assumptions}} 
\item 
\label{assum:Gamma} 
\emph{Data regularity:} The data $a$ and $s$ satisfy the regularity assumption 
\[\lim_{\tau \to \infty} \Gamma_{a,s}(\tau) = \infty  \,.\] 
\end{enumerate}

\begin{remark}\label{rmk:Holder-cont implies Gamma-condition}
The following list provides simple conditions that imply our assumptions \ref{assum:flatness} or \ref{assum:Gamma}. 
\begin{enumerate}[label=(\roman*)]
\item 
The condition \ref{assum:flatness} is for example satisfied if $S$ is bounded from below, i.e.\ when 
\[ 
\frac{1}{C} \leq s(x,y) \leq C 
\] 
for some constant $C>0$ and for all $x$, $y \in \mathfrak X$. 
\item If $\mathfrak X$ is a finite set then \ref{assum:Gamma} holds for any $a$ and $s$. 
\item 
In the  case $\mathfrak X=[0,1]$ and $\mu$ the Lebesgue-measure on $[0,1]$, Assumption~\ref{assum:Gamma} holds e.g.\ when $s$ and $a$ are Hölder-continuous with Hölder-exponent $\frac{1}{2}$. 
\item 
In fact it suffices that Hölder-continuity of $a$ and $s$ holds piecewise.
Let $I_1$, \ldots, $I_K$ be disjoint intervals in $[0,1]$ such that $I_1 \cup \ldots \cup I_K = [0,1]$. 
If $s \colon [0,1] \times [0,1] \to [0,\infty)$, $a\colon [0,1] \to \C$ are such that 
 $s|_{I_l \times I_k}$ and $a|_{I_l}$ are $\frac{1}{2}$-Hölder-continuous for every $l$, $k \in \{1, 2, \ldots, K\}$ then 
\ref{assum:Gamma} is satisfied.   
In particular, if $s$ satisfies \ref{assum:flatness} in addition, then \ref{assum:flatness} and \ref{assum:Gamma} both hold. 
\end{enumerate} 
\end{remark}

To classify the support of $\sigma$ 
we introduce the operator 
$B\equiv B_\zeta: \cal{B} \to \cal{B}$  given by 
\begin{equation} \label{eq:def_B} 
B_\zeta:= D_{\abs{a-\zeta}^2}-S\,,
\end{equation} 
where $D_u \colon \cal{B} \to \cal{B}, x \mapsto ux$ denotes the multiplication operator on $\cal{B}$ 
with $u \in \mathcal B$. 
Since $B$ maps real-valued functions to real-valued functions,  
we obtain a function $\beta \colon \C \to \R$ defined through 
\begin{equation} \label{eq:def_beta} 
\beta(\zeta) := \inf_{x \in \cal B_+} \sup_{y \in \cal B_+} \frac{\scalar{x}{B_\zeta \2y}}{\scalar{x}{y}}\end{equation} 
for $\zeta \in \C$, where the infimum and supremum are taken over $\cal B_+:=\{x \in \cal{B}: x>0\}$ 
 and $ \scalar{u_1}{u_2} := \avg{\ol{u}_1 u_2}$ for all $u_1$, $u_2 \in \mathcal B$. 
The definition of $\beta$ is motivated by the Birkhoff-Varga formula for the spectral radius of a matrix with positive entries \cite{Birkhoff_Varga1958}. 
In terms of $\beta$ we define the set 
\begin{equation} \label{eq:def_mathbb_S} 
 \mathbb S := \{ \zeta \in \C \colon \beta(\zeta ) < 0 \}\,, 
\end{equation} 
whose closure coincides with $\supp \sigma$, as stated  in the next theorem. 
We will see in Proposition~\ref{prp:S characterisation}~\ref{continuity of beta} that $\beta$ is a continuous function and therefore $\mathbb{S}$ is an open set.  
Moreover, $\beta$ is real analytic in a neighbourhood of $\partial \mathbb{S}$, see Corollary~\ref{crl:beta as eigenvalue of B} below.

\begin{theorem}[Properties of $\sigma$] \label{thr:properties_sigma_general} 
Let  $a \in \mathcal B=L^\infty(\mathfrak X, \Sigma, \mu)$ and  $\frak c$ be a $\cal B$-valued circular element such that $S$ and $S^*$ from \eqref{eq:def_S_S_star} satisfy \eqref{eq:def_operators_S_and_S_star}  with  $s$ fulfilling   \ref{assum:flatness} and \ref{assum:Gamma}. 
\begin{enumerate}[label=(\roman*)] 
\item \label{item:prop_sigma_ii} With respect to the Lebesgue measure, the Brown measure $\sigma$  of $a + \frak c$   has a bounded density on $\C$, which we also denote by $\sigma$, i.e.\ 
$\sigma(\dd \zeta) = \sigma(\zeta) \dd^2 \zeta$. 
\item \label{item:prop_sigma_iii}  The density $\zeta \mapsto \sigma(\zeta)$ is strictly positive on $\mathbb{S}$ and admits a real analytic extension to an open  neighbourhood of $\ol{\mathbb{S}}$. 
\item \label{item:prop_sigma_iv} $\supp \sigma = \overline{\mathbb{S}}$  and this set is bounded. Furthermore $\spec(D_a) \subset \mathbb{S}$. 
\item \label{item:prop_sigma_v} $\partial \mathbb{S}=\{ \zeta \in\C \colon \beta(\zeta) = 0\}$ and it is a real analytic variety of (real) dimension at most $1$. 
\item \label{item:prop_sigma_vi} The unique continuous extension $\sigma \colon \overline{\mathbb{S}} \to [0,\infty)$ of the density $\sigma|_{\mathbb S}$ to $\overline{\mathbb{S}}$ satisfies  $\sigma(\zeta) = g(\zeta)\abs{\partial_{\zeta} \beta(\zeta)}^2$  
for all $\zeta \in \partial \mathbb{S}$, where $g :\partial \mathbb{S} \to (0,\infty)$ is a strictly positive function that can be extended to a real analytic function on a neighbourhood of $\partial \mathbb{S}$.  
\end{enumerate} 
\end{theorem} 

The \hyperlink{proof_thr:properties_sigma_general}{proof of Theorem~\ref{thr:properties_sigma_general}} 
is given at the end of Section~\ref{sec:Properties of the Brown measure} below.

Our next results describe in detail the edge behaviour of $\sigma$ at points $\zeta_0 \in \partial\mathbb{S}$. Typically we observe a sharp drop of the density $\sigma$ at the edge $\partial\ol{\mathbb{S}}$. In fact, by Theorem~\ref{thr:properties_sigma_general} \ref{item:prop_sigma_vi} this is the case if and only if $\partial_\zeta \beta(\zeta_0)$ does not vanish. 

\begin{definition}[Regular edge and singular points]
Let $\zeta_0 \in \partial \mathbb{S}$ such that $\partial_\zeta \beta(\zeta_0) \ne 0$. Then we say that $\zeta_0$ is a \emph{regular edge point}.  We denote the set of these points by $\rm{Reg}$ and its complement by $\rm{Sing}:= \partial \mathbb{S} \setminus \rm{Reg}$, which is called the set of \emph{singular points} for $\sigma$. 
\end{definition}

To give a short statement of the  classification of the singular  points for $\sigma$ we introduce the following notion. 

\begin{definition}[Singularity types] \label{def:singularity type}
Let $\alpha: U \to \C$ be a real analytic function on an open set $U \subset \C$ and $\zeta_0 \in U$. We say that $\alpha$ is of singularity type $p(x,y)$ at $\zeta_0$ for a polynomial $p: \R^2 \to \C$ if there is a real analytic diffeomorphism $\Phi: V \to U_0$ from an open neighbourhood $V$ of $0 \in \C$ to an open neighbourhood $U_0$ of $\zeta_0 
=\Phi(0)$, such that 
\[
\alpha(\Phi(x + \ii y)) = p(x,y) \,, \qquad x + \ii y \in V\,.
\]
\end{definition}

In the sense above $\beta$ is of singularity type $x$ at any regular edge point. 
By Theorem~\ref{thr:properties_sigma_general} the edge $\partial \supp \sigma$ is locally analytically diffeomorphic  to a line and $\sigma$ has a jump discontinuity along this line.

\begin{theorem}[Classification of edge and internal  singularities]\label{thr:Sing classification} Let  $a \in \mathcal B$ and  $s$ satisfy \ref{assum:flatness} and \ref{assum:Gamma}. 
The set of singular  points for $\sigma$ admits a partition 
\bels{Sing partition}{
\rm{Sing}= \bigcup_{k=1}^\infty\pb{\rm{Sing}_{k}^{\rm{int}}\cup \rm{Sing}_{k}^{\rm{edge}}}\cup \rm{Sing}^{\rm{int}}_{\infty}\,.
}
Here $\rm{Sing}\setminus \rm{Sing}_{\infty}$ is a finite set, while $\rm{Sing}_{\infty}$ is either empty or 
a finite disjoint union of closed real analytic paths without self-intersections. 
The different singularity types are characterised by the following properties: 
\begin{enumerate}[label=\arabic*)] 
\item Internal singularities of order $k$: Any $\zeta \in \rm{Sing}_{k}^{\rm{int}}$ is an isolated point of $\C \setminus \mathbb S$ and $\sigma$ is of singularity type $x^2 + y^{2k}$ at $\zeta$. 
\item Edge singularities of order $k$:  Any $\zeta \in \rm{Sing}_{k}^{\rm{edge}}$ lies in  $\partial \ol{\mathbb S}$ and  $\beta$ is of singularity type $ y^{1+k}-x^2$ at $\zeta$. 
\item Internal singularities of order $\infty$:  Any $\zeta \in \rm{Sing}^{\rm{int}}_{\infty}$  lies in the interior of $\ol{\mathbb S}$ and $\sigma$ is of singularity type $x^2$ at $\zeta$. 
\end{enumerate}
\end{theorem}

The \hyperlink{proof_thr:Sing classification}{proof of Theorem~\ref{thr:Sing classification}} is presented 
at the end of Section~\ref{sec:Properties of the Brown measure} below. 

\begin{remark}
The internal singularities $\rm{Sing}_{k}^{\rm{int}}\cup  \rm{Sing}^{\rm{int}}_{\infty}$  are classified by their local behaviour of the density $\sigma$. 
The edge singularities $\rm{Sing}_{k}^{\rm{edge}} $ are equivalently characterised by the local shape of the spectral edge $\partial \supp \sigma$, namely for $\zeta \in \rm{Sing}_{k}^{\rm{edge}}$ there is a local  analytic diffeomorphism $\Psi\colon U \to V$ from an open neighbourhood $V$ of $\zeta_0$ to an open  $U \subset \C$ with $\Psi(\zeta_0)= 0$ such that 
\[
\Psi( \supp \sigma \cap U) = \{x+\ii y \in \C: x^2 -y^{1+k}\ge 0\}\cap V
\]
\end{remark}

Finally we show that all singularities allowed by the classification in Theorem~\ref{thr:Sing classification} do occur in the simple case when $\frak{c}$ is a standard circular element and $\frak{X}=[0,1]$ with the Lebesgue measure. A standard circular element satisfies $E[\frak{c}^* b \frak{c}]=E[\frak{c} b \frak{c}^*] = \avg{b}$ for all $b \in \mathcal B$. 

\begin{theorem}[Existence of all singularity types]
\label{thr:all_singularities_appear}
For each of the sets $\rm{Sing}_{k}^{\rm{int}}$, $\rm{Sing}_{k}^{\rm{edge}}$  with $k \in \N$ and  $\rm{Sing}^{\rm{int}}_{\infty}$,  there is a choice $a \in L^\infty[0,1]$ such that for the Brown measure of $a + \frak{c}$, with $\frak{c}$ a standard circular element,  this set is not empty. With the exception of $\rm{Sing}^{\rm{int}}_{\infty}$, the deformation $a$ can be chosen to have finite image. 
\end{theorem}

The \hyperlink{proof_thr:all_singularities_appear}{proof of Theorem~\ref{thr:all_singularities_appear}} can be 
found just before Section~\ref{subsec:Singularity types of even order}.

\begin{remark}[Circular element with general deformation] \label{rmk:general_deformation}
Although not directly covered by our results stated above, the classification of singularities from Theorem~\ref{thr:Sing classification} also hold for the
Brown measure of $a + \mathfrak c$, where $a$ is a general operator in a tracial von Neumann algebra 
and $\mathfrak c$ is a standard circular element, which is $\ast$-free from $a$. 
We note that $a$ can be non-normal in general. 
We define $\mathbb{S} = \{ \zeta \in \C \colon f(\zeta) > 1\}$ with  
\[
f(\zeta):= \lim_{\eta \downarrow 0} \avgbb{\frac{1}{(a-\zeta)(a-\zeta)^* + \eta^2}}\,
\]
for $\zeta \in \C$.  
Then the support  of the Brown measure $\sigma$ is $\ol{\mathbb S}$ \cite{Zhong2021} and the measure has a density on $\C$ \cite{BelinschiYinZhong2024}, which is real analytic and strictly positive inside $\mathbb{S}$ \cite{Zhong2021}.  
Under the additional assumption 
\begin{equation} \label{eq:assumption_spec_a_subset_mathbb_S} 
 \spec(a) \subset \mathbb{S}, 
\end{equation}
the dichotomy between regular edge points, at which the density has a jump discontinuity, and singular points of $\sigma$ was established in \cite{ErdosJi2023}. The singular  points $\zeta \in \rm{Sing}$ are called quadratic  edges in \cite{ErdosJi2023} and are classified by satisfying $f(\zeta)=1$ and $\partial_\zeta f(\zeta)=0$, while regular edge points $\zeta \in \rm{Reg}$ satisfy $f(\zeta)=1$ and $\partial_\zeta f(\zeta)\ne 0$. Here we provide a classification of the singular points $\zeta \in \rm{Sing}$ by showing that $\rm{Sing}$ is a disjoint union of the sets $ \rm{Sing}_{k}^{\rm{int}}$, $ \rm{Sing}_{k}^{\rm{edge}}$ with $k \in \N$ and $ \rm{Sing}^{\rm{int}}_{\infty}$ as in \eqref{Sing partition}, where these sets are defined as in Theorem~\ref{thr:Sing classification} with $\beta := 1 - f$. 
 We provide the proof for this classification in  Section~\ref{sec:circular_element_general_deformation} below. 
\end{remark}

\subsection{Notations}

We now introduce some notations used throughout. 
We write $\qq{n} \defeq \{ 1, \ldots, n\}$ for $n \in \N$. 
For $r >0$, we denote by $\mathbb{D}_r \defeq \{ z \in \C \colon \abs{z} < r \}$ the disk of radius $r$ around the origin in $\C$ and 
by $\dist(x,A) :=\inf\{ \abs{x-y} \colon y \in A \}$ the Euclidean distance of a point $x \in \C$ from a set $A \subset \C$. 

We use the convention that $c$ and $C$ denote  generic constants that may depend on the model parameters, but are otherwise     
 uniform in all other parameters, e.g.\ $n$, $\zeta$, etc.. 
For two real scalars $f$ and $g$ 
we write $f \lesssim g$ and $g \gtrsim f$ if  $f \leq C g$ for such a constant $C>0$. 
In case $f \lesssim g$ and $f \gtrsim g$ both hold,  we write $f \sim g$. 
If the constant $C$ depends on a parameter $\delta$ that is not a model parameter, we write $\lesssim_\delta$, $\gtrsim_\delta$ and 
$\sim_\delta$, respectively. 
The notation for inequality up to constant is also used for self-adjoint matrices/operators $f$ and $g$, where $f \leq C g$ is interpreted in the sense of  quadratic forms. 
For complex $f$  and $g \geq 0$  we write $f = O(g)$ in case $\abs{f} \lesssim g$. 
Analogously $f = O_\delta(g)$ expresses the fact  $\abs{f}\lesssim_\delta g$. 

In addition to $L^\infty$, we introduce the usual $L^p$ spaces on $(\mathfrak X, \cal A, \mu)$. 
We denote them by $L^p:=L^p(\mathfrak X, \cal{A}, \mu)$ and the corresponding norms by $\norm{\2\cdot \2}_p$. 
For functions  $u_1$, $u_2 \in L^2$, we define their  scalar product as 
\[
 \scalar{u_1}{u_2} := \avg{\ol{u}_1 u_2}\,.
\]

\section{Examples} \label{sec:examples}

In this section, we present several examples that illustrate some of the different singularity types appearing 
in Theorem~\ref{thr:Sing classification}.  
Throughout this section, we choose $\mathfrak X = [0,1]$, $\mu$ the Lebesgue-measure on $[0,1]$ and $s$ constant on 
$[0,1]^2$. 
If $s \equiv t$ is constant on $[0,1]^2$ then it is easy to see that $\beta(\zeta) = \frac{1}{t} - \int_0^1 \frac{\dd x}{\abs{a(x) - \zeta}^2}$ for all $\zeta \in \C$ and therefore,  
as has already been derived in \cite{Khoruzhenko1996,tao2010,BordenaveCaputoChafai2014}, 
\begin{equation} \label{eq:formula_S_iid} 
 \mathbb{S} = \bigg\{ \zeta \in \C \colon \int_0^1 \frac{\dd x}{\abs{a(x)-\zeta}^2} > \frac{1}{t}  \bigg\}. 
\end{equation} 

\begin{example}[Simplest edge singularity of $\partial \mathbb{S}$] \label{example:x2_minus_y2} 
We choose $s \equiv 1$ and 
$a \colon [0,1] \to \C$ with $a(x_1) =1$ and $a(x_2) = -1$ for all $x_1 \in [0,1/2)$ and $x_2 \in [1/2,1]$. 
In that case, we obtain from \eqref{eq:formula_S_iid} that 
\[ \mathbb{S} = \{ \zeta \in \C \colon \abs{1-\zeta}^{-2} + \abs{1+ \zeta}^{-2} >1 \}. \] 
As we will see in Lemma~\ref{lmm:Examples exist} \ref{lmm:Examples 1} below, this yields an edge singularity of order $1$, i.e.\ in 
$\mathrm{Sing}_1^{\mathrm{edge}}$ at zero. 
The boundary of $\mathbb{S}$ and the eigenvalues of a sampled corresponding random matrix model are depicted 
in Figure~\ref{subfig:x2_minus_y2}.  This example  also appeared in \cite[Example~5.2]{BianeLehner2001}.  
\end{example} 

\begin{example}[Simplest internal singularity of $\sigma$] \label{example:x2_plus_y2}
With the choices $s \equiv 1/4$ and 
$a \colon [0,1] \to \C$ with $a(x_1) = (1 + \ii)/\sqrt{2}$ , $a(x_2) = (1 - \ii)/\sqrt{2}$, $a(x_3) = (-1 + \ii)/\sqrt{2}$ and $a(x_4) = -(1 + \ii)/\sqrt{2}$ for all $x_1 \in [0,1/4)$, $x_2 \in [1/4,1/2)$, $x_3 \in [1/2,3/4)$ and 
$x_4 \in [3/4,1]$, \eqref{eq:formula_S_iid} yields 
\[ \mathbb{S} = \{ \zeta \in \C \colon  \abs{\zeta -(1 + \ii)/\sqrt{2}}^{-2} + \abs{\zeta -(1 - \ii)/\sqrt{2}}^{-2} 
+ \abs{\zeta +(1 - \ii)/\sqrt{2}}^{-2} + \abs{\zeta +(1 + \ii)/\sqrt{2}}^{-2}  > 1 \}.\] 
Lemma~\ref{lmm:Examples exist} \ref{lmm:Examples 1} below will demonstrate that this gives rise to an internal singularity of order $1$, i.e.\ in  
$\mathrm{Sing}_2^{\mathrm{int}}$ at zero. 
Figure~\ref{subfig:x2_plus_y2} shows the boundary of $\mathbb{S}$ and the eigenvalues of a 
sampled corresponding random matrix model.  We note that a  similar example was given in 
\cite[Example~3.1(d)]{ErdosJi2023}.  
\end{example} 

\begin{figure}[h!]  
\subfloat[Example~\ref{example:x2_minus_y2} \label{subfig:x2_minus_y2}]{ 
\includegraphics[width=.49\textwidth]{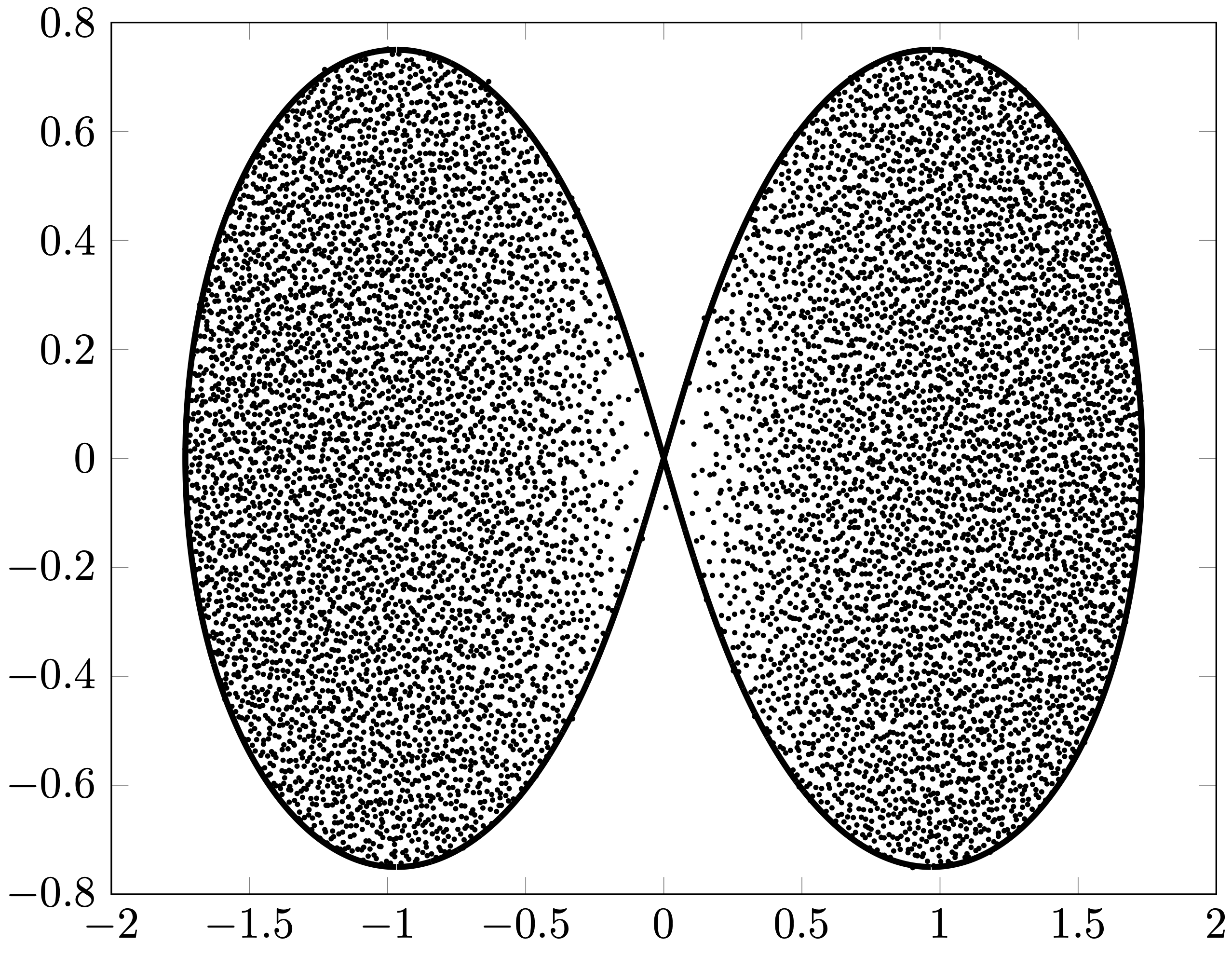}
} 
\hspace*{.02\textwidth}  
\subfloat[Example~\ref{example:x2_plus_y2} \label{subfig:x2_plus_y2}]{
\includegraphics[width=.49\textwidth]{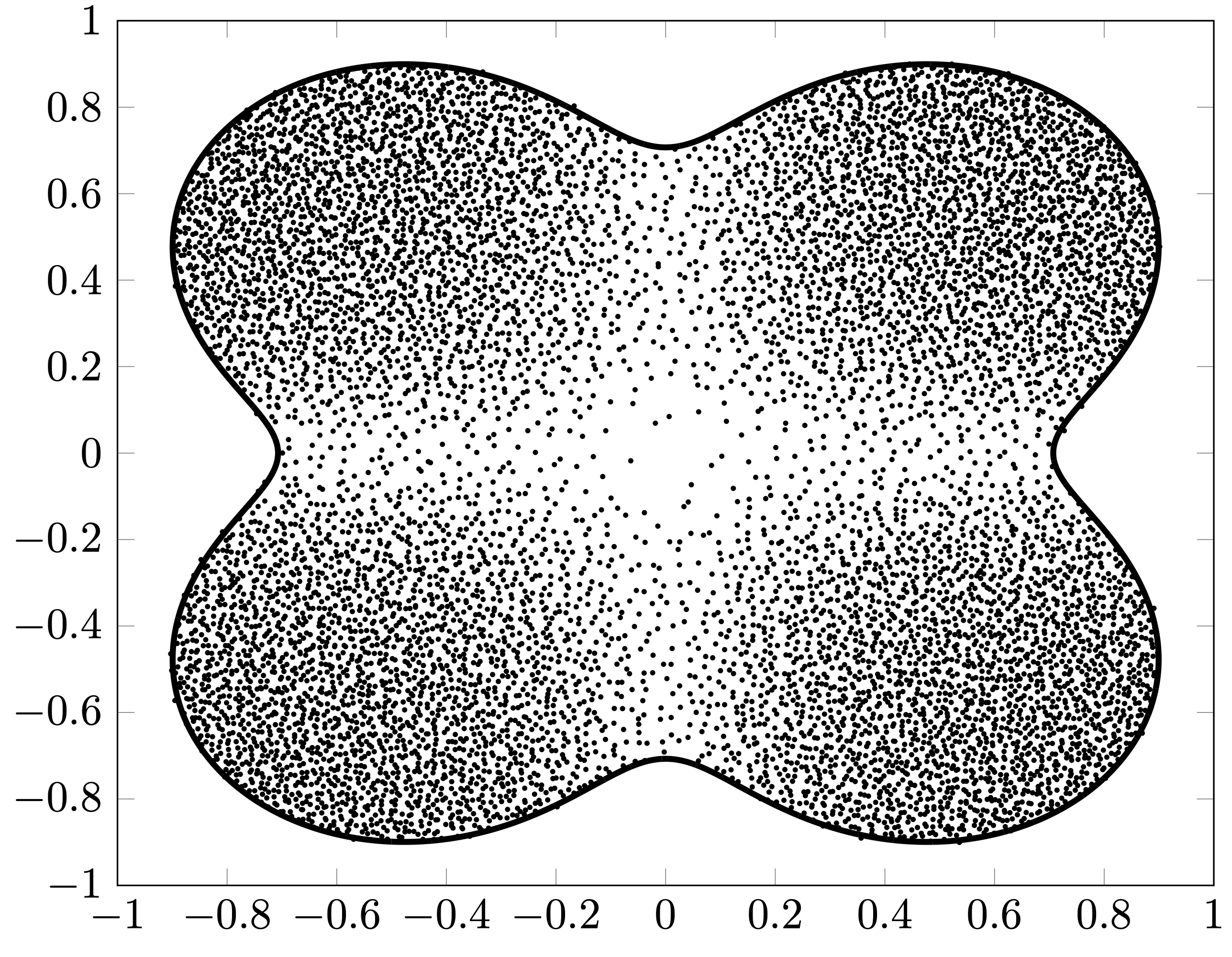}
} 
\caption{The solid black lines in subfigures (\protect\subref{subfig:x2_minus_y2}) and (\protect\subref{subfig:x2_plus_y2}) show the 
boundary of $\mathbb{S}$ from Examples~\ref{example:x2_minus_y2} and \ref{example:x2_plus_y2}, respectively. 
The black dots are the eigenvalues of a sample of $A + X/\sqrt{n}$, where $X$ is an $n\times n$ matrix with i.i.d.\ $N(0,1)$ standard real normal distributed entries, $n=10000$ and $A=(\diag(a(i/n))\delta_{ij})_{i,j=1}^n$ is 
a diagonal matrix and $a$ is chosen as in Examples~\ref{example:x2_minus_y2} and \ref{example:x2_plus_y2}, respectively.} 
 \label{fig:simple} 
\end{figure} 

\begin{example}[One-sided edge cusp of $\partial \mathbb{S}$]  \label{example:x2_y3} 
Let $s \equiv t:= \frac{2}{3}(20-7\sqrt{7})$ be constant on $[0,1]^2$, $\delta = (-17 + 7\sqrt{7})/8$ and 
$a \colon [0,1] \to \C$ with $a(x_1) = 1$, $a(x_2) = -1$ and $a(x_3) = \ii$ for all $x_1 \in [0,1/(2+ \delta))$,  
$x_2 \in [1/(2 + \delta),2/(2+\delta))$ and $x_3 \in [2/(2+\delta),1]$.
We note that \eqref{eq:formula_S_iid} yields 
\[ \mathbb{S} = \bigg\{ \zeta \in \C \colon \frac{1}{\abs{1-\zeta}^2} + \frac{1}{\abs{1+ \zeta}^2} + 
\frac{\delta}{\abs{\zeta- \ii}^2} > \frac{2 + \delta}{t} \bigg\}. \]  
With $y_0 = (\sqrt{7} - 2)/3$, a simple calculation shows that $\beta$ is of singularity type $x^2 + y^3$ at $\ii y_0$ and, thus, $\ii y_0$ is an edge singularity of order $2$. 
For this example, a plot analogous to the previous examples is presented in Figure~\ref{subfig:x2_y3}. 
\end{example}

\begin{example}[Two-sided edge cusp]  \label{example:x2_y4} 
We set $s \equiv t := 4$ be constant on $[0,1]^2$ and 
 $a\colon [0,1] \to \C$,  with $a(x_1) = \sqrt{3} + \ii$, $a(x_2) = \sqrt{3} - \ii$, $a(x_3) = - \sqrt{3} + \ii$ 
and $a(x_4) = - \sqrt{3} - \ii$ for all $x_1 \in [0,1/4)$, $x_2 \in [1/4,1/2)$, $x_3 \in [1/2,3/4)$ 
and $x_4 \in [3/4,1]$.  We conclude from \eqref{eq:formula_S_iid} that 
\[ \mathbb{S} = \bigg\{ \zeta \in \C \colon \frac{1}{\abs{\sqrt{3} + \ii-\zeta}^2} + \frac{1}{\abs{\sqrt{3} - \ii-\zeta}^2} + \frac{1}{\abs{- \sqrt{3} + \ii-\zeta}^2}  + \frac{1}{\abs{- \sqrt{3} - \ii-\zeta}^2}  > 1 \bigg\}. 
\] 
Using Lemma~\ref{lmm:singularity types sigma f and beta} below,  $-\beta$ and hence the boundary 
of $ \mathbb{S}$ has a singularity of type $x^2 - y^4$ at zero, making this an edge singularity of order $2$. 
Figure~\ref{subfig:x2_y4} shows a plot of this example analogous to the previous figures. 
\end{example} 

We refer to \cite[Example~2.6 and Figure~1]{AEKN_Kronecker} for more examples in the spirit of the previous examples.

\begin{figure}[ht]
\subfloat[Example~\ref{example:x2_y3} \label{subfig:x2_y3}]{
\includegraphics[width=.49\textwidth]{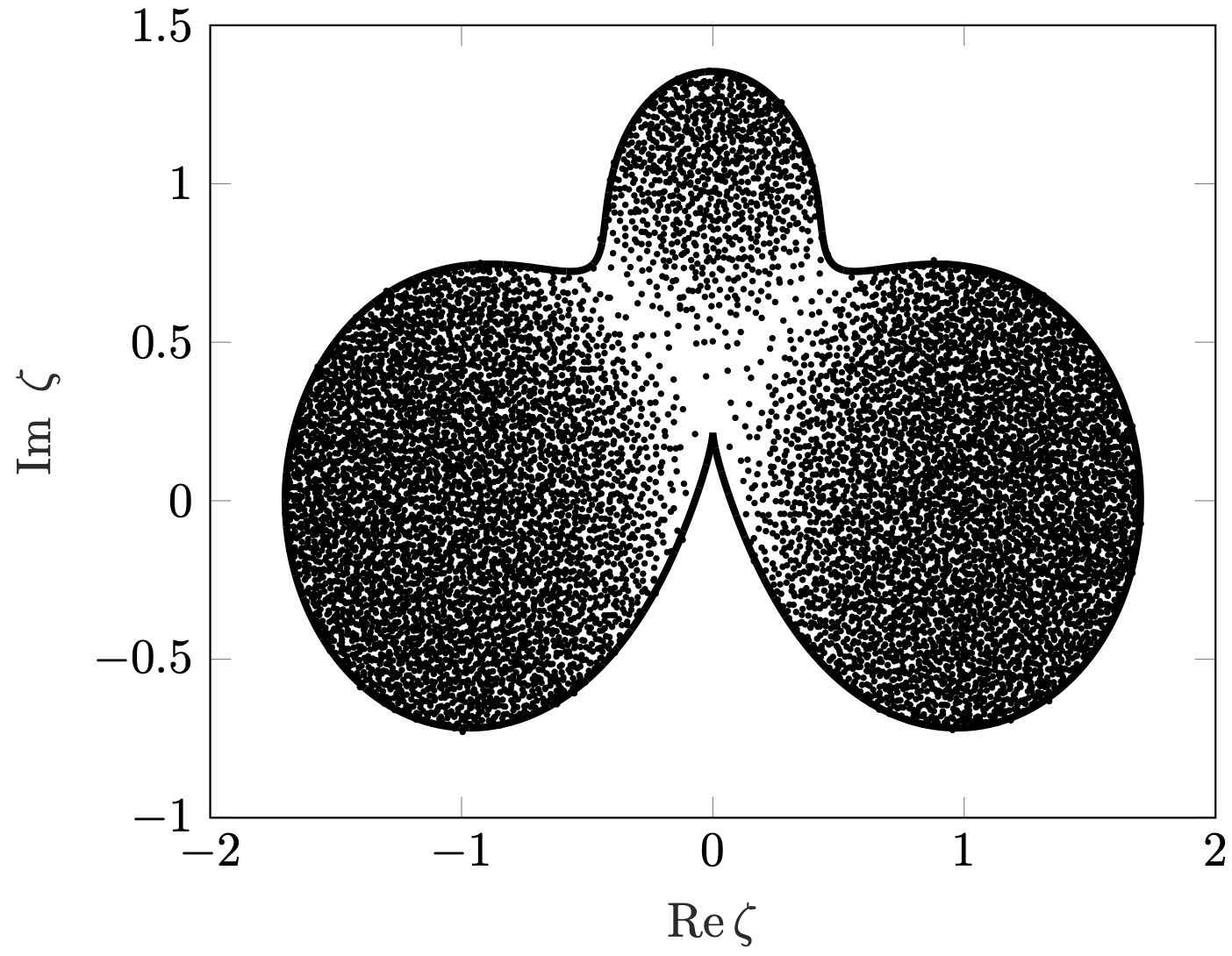}
} 
\hspace*{.02\textwidth}  
\subfloat[Example~\ref{example:x2_y4} \label{subfig:x2_y4}]{
\includegraphics[width=.49\textwidth]{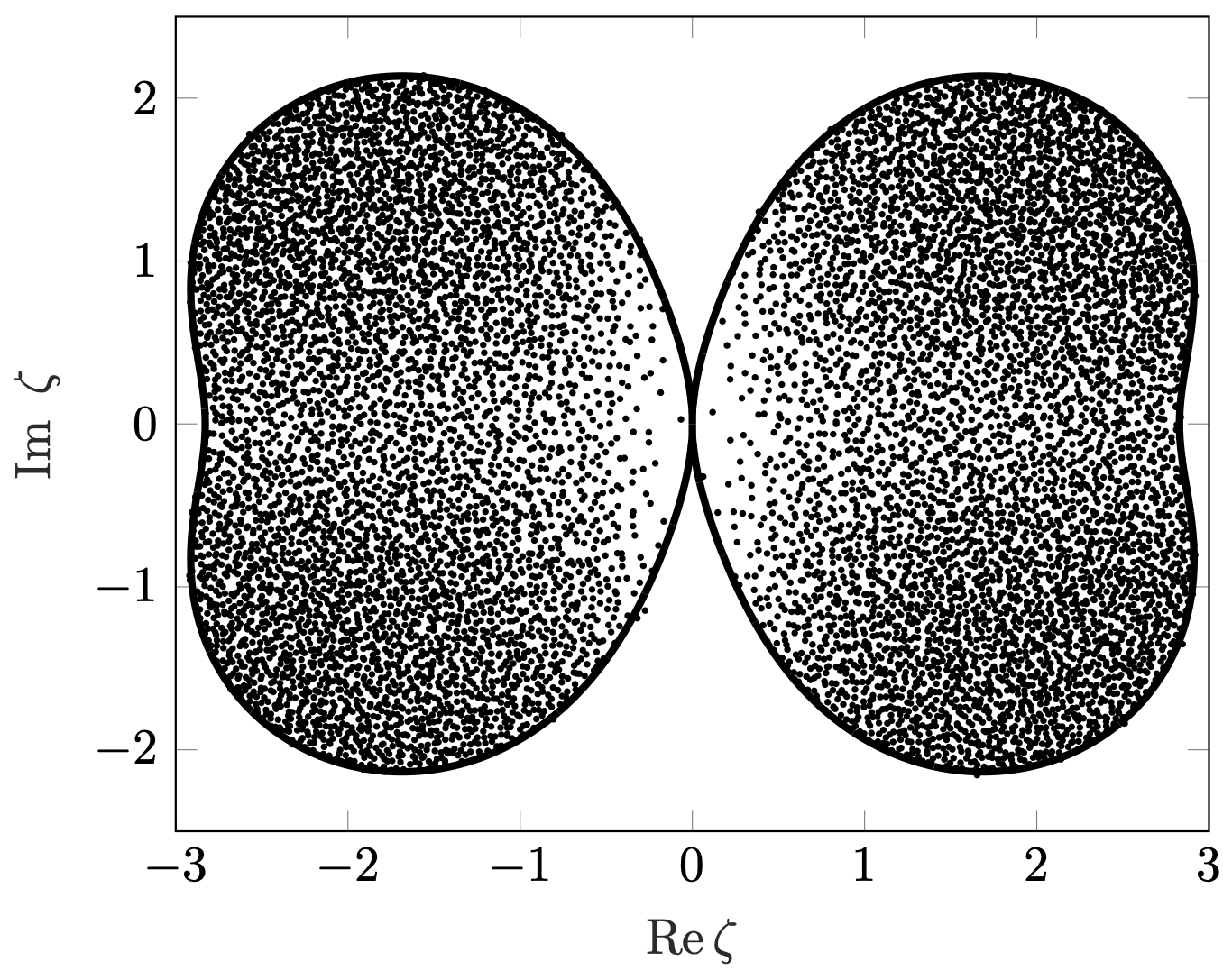}
} 
\caption{The solid black lines in subfigures (\protect\subref{subfig:x2_y3}) and (\protect\subref{subfig:x2_y4}) show the 
boundary of $\mathbb{S}$ from Examples~\ref{example:x2_y3} and \ref{example:x2_y4}, respectively. 
The black dots are the eigenvalues of a sample of $A + X/\sqrt{n}$, where $X$ is an $n\times n$ matrix with i.i.d.\ $N(0,1)$ standard real normal distributed entries, $n=10000$ and $A=(\diag(a(i/n))\delta_{ij})_{i,j=1}^n$ is 
a diagonal matrix and $a$ is chosen as in Examples~\ref{example:x2_y3} and \ref{example:x2_y4}, respectively.} 
\label{fig:one} 
\end{figure}

\begin{example}[Internal singularity of type $\rm{Sing}^{\rm{int}}_{\infty}$] \label{example:x2} 
We choose $s \equiv 1$ on $[0,1]^2$ and  
\[ a \colon [0,1] \to \C, \qquad x \mapsto \begin{cases} \sqrt{2} \ee^{4\pi \ii x} & \text{ if } x \in [0,1/2], \\ 0 & \text{ if } x \in (1/2,1]. \end{cases} \] 
As will be proved in Section~\ref{subsec:Singularity types of infinite order} below,  in this setup 
\begin{equation} \label{eq:example_x2_mathbb_S} 
\mathbb S = \big\{ \zeta \in \C \colon 0 \leq \abs{\zeta}^2 < 1 \text{ or } 1 < \abs{\zeta}^2 < {(3 + \sqrt{5})/2} \big\}  
\end{equation} 
and each point of $\partial \mathbb{D}_1$ belongs to $\rm{Sing}_\infty^{\rm{int}}$. 
Moreover, it is shown in Section~\ref{subsec:Singularity types of infinite order} that, for each $\zeta \in \C$, 
\begin{equation} \label{eq:sigma_for_x2_example} 
 \sigma(\zeta) = \frac{1}{\pi} \bigg( 1 - \frac{2}{2+x + \frac{x}{(2x-1)^3}} \bigg( 1 + \frac{1}{x^2} \bigg) \bigg) \mathbf{1} (\zeta \in \mathbb{S}), 
\end{equation} 
where $x =x(\abs{\zeta}^2) \in (0,\infty)$ is the unique positive solution of $\frac{1}{x} + \frac{1}{\sqrt{1 + 4x+x^2-8\abs{\zeta}^2 + 3}}=2$.  
The boundary of $\mathbb{S}$ and the sampled eigenvalues are depicted in Figure~\ref{fig:two}.  
\end{example}

\begin{figure}[h!]
\begin{center} 
\includegraphics[width=.49\textwidth]{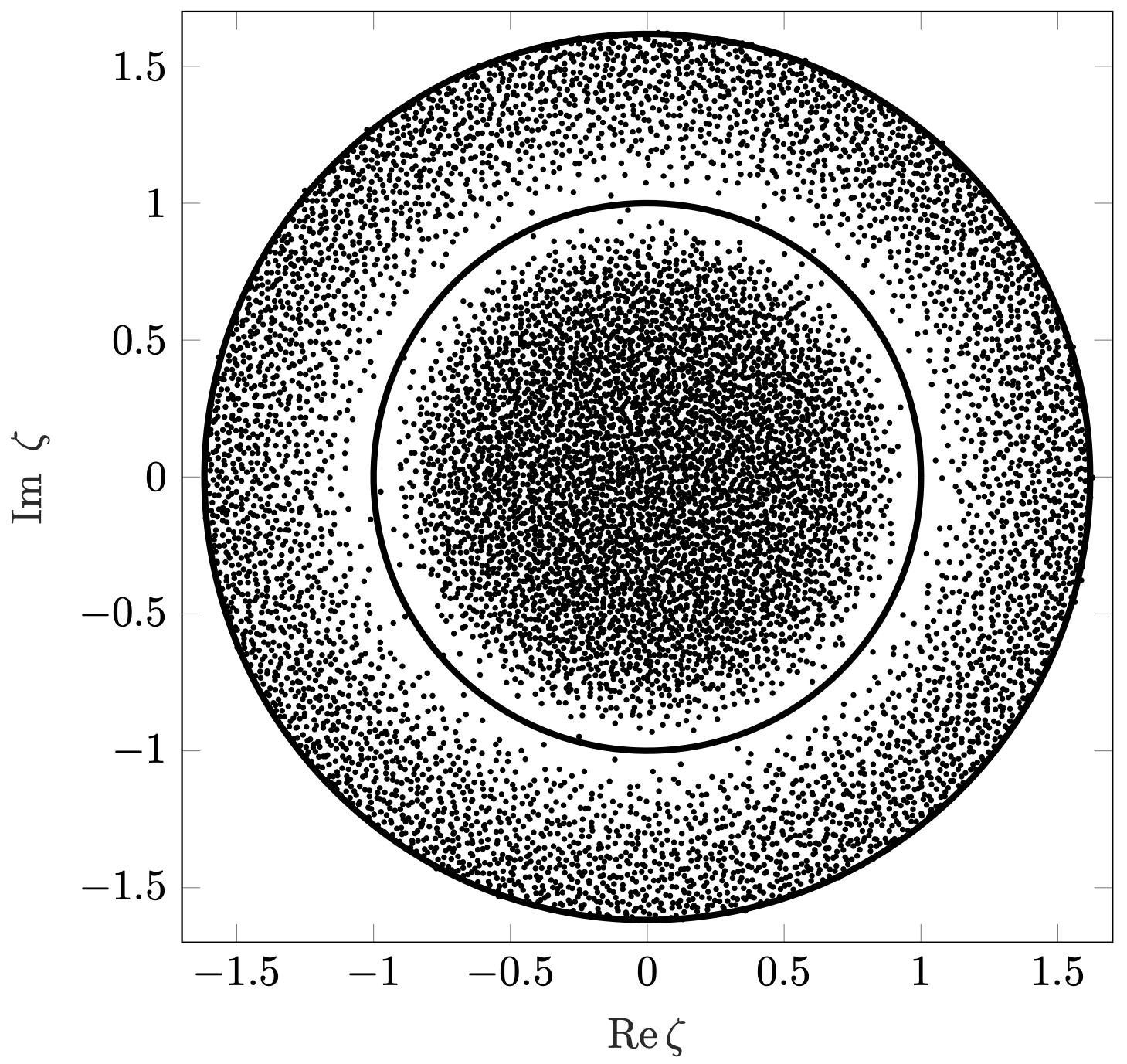} 
\end{center} 
\caption{The solid lines in this figure show the 
boundary of $\mathbb{S}$ from Example~\ref{example:x2}. 
The black dots show the eigenvalues of a sample of $A + X/\sqrt{n}$, where $X$ is an $n\times n$ matrix with i.i.d.\ $N(0,1)$ standard real normal distributed entries, $n=10000$ and $A=(\diag(a(i/n))\delta_{ij})_{i,j=1}^n$ is 
a diagonal matrix and $a$ is chosen as in Example~\ref{example:x2}.} 
\label{fig:two} 
\end{figure}

\section{Brown measure, Hermitization and Dyson equation}

In this section, we derive the Dyson equation for the Cauchy transform of the Hermitization of the deformed operator-valued circular element  $a + \mathfrak c$. This Hermitization itself is a family of operator-valued semicircular 
elements indexed by the spectral parameter $\zeta \in \C$ and the associated Dyson equation is our main tool for analyzing  the Brown measure of  $a + \mathfrak c$. 

Throughout the following, let $(\mathcal A, E, \mathcal B)$ be an operator-valued probability space with positive conditional expectation  $E \colon \mathcal A \to \mathcal B$ such that $(\mathcal A, \avg{E[\,\cdot\,]})$ is a tracial $W^*$-probability space. 
We assume that $\mathcal B \subset \mathcal A$ is a commutative von Neumann subalgebra with trace $\avg{\2\cdot\2}$ containing the unit of $\mathcal A$. 
We will also work on the operator-valued probability space $(\mathcal A^{2\times 2}, \mathrm{id} \otimes E, \mathcal B^{2\times 2})$, for which we employ the identifications $\mathcal A^{2\times 2}= \C^{2 \times 2} \otimes \mathcal A$ and  $\mathcal B^{2\times 2}= \C^{2 \times 2} \otimes \mathcal B$.

For any $\mathfrak c \in \mathcal A$, $a \in \mathcal B$ and $\zeta \in \C$, we  \emph{hermitise} $\mathfrak c + a - \zeta$ 
and  apply $\mathrm{id} \otimes E$ to the resolvent of the Hermitization at $w \in \C$ with $\Im w>0$. 
We define 
\begin{equation} \label{eq:def_M_semicircular} 
M(\zeta,w ) = \mathrm{id} \otimes E \bigg[\begin{pmatrix} - w & \mathfrak c + a - \zeta \\ (\mathfrak c +a - \zeta)^* & - w \end{pmatrix}^{-1} \bigg]. 
\end{equation} 
Here, $w$ and $\zeta$ are interpreted as the constant functions on $\mathfrak X$ with the respective value. 
Note that the inverse is well defined as $\Im w>0$ and the operator whose inverse is taken in \eqref{eq:def_M_semicircular} is self-adjoint for $w = 0$. 
Furthermore, the imaginary part $\Im M(\zeta,w) =\frac{1}{2\ii} ( M(\zeta,w) - M(\zeta,w)^*)$ is positive definite.

For the next lemma, we define $\Sigma \colon \mathcal B^{2\times 2} \to \mathcal B^{2\times 2}$ through 
\begin{equation} \label{eq:def_Sigma} 
 \Sigma \bigg[ \begin{pmatrix} r_{11} & r_{12} \\ r_{21} & r_{22} \end{pmatrix} \bigg] 
= \begin{pmatrix} S r_{22} & 0 \\ 0 & S^* r_{11} \end{pmatrix} 
\end{equation} 
for all $r_{11}$, $r_{12}$, $r_{21}$, $r_{22} \in \mathcal B$, where $S$ and $S^*$ are as in \eqref{eq:def_S_S_star}.

\begin{lemma} \label{lmm:circular_element_Dyson}
Let $\mathfrak c \in \mathcal A$ be a $\cal{B}$-valued circular element satisfying \eqref{second order cumulant of c} and \eqref{eq:def_S_S_star}. 
Then 
 $\begin{pmatrix} 0 & \mathfrak{c} \\ \mathfrak{c}^* &0 \end{pmatrix}$ is an operator-valued semicircular 
element in $(\mathcal A^{2\times 2}, \mathrm{id} \otimes E, \mathcal B^{2\times 2})$ 
with covariance $\Sigma$ from \eqref{eq:def_Sigma}. 
Moreover, $M(\zeta,w)$ from \eqref{eq:def_M_semicircular} with such $\mathfrak c$
 satisfies 
\begin{equation} \label{general MDE} 
- M(\zeta,w)^{-1} = \begin{pmatrix} w & \zeta - a \\  \ol{\zeta - a} & w \end{pmatrix} + \Sigma[M(\zeta,w)] 
\end{equation} 
for any $a \in \mathcal B$, $w \in \C$ with $\Im w >0$ and $\zeta \in \C$. 
\end{lemma}

The identity \eqref{general MDE} is called \emph{Dyson equation}. 
Under the constraint that $\Im M(\zeta, w) $ is positive 
definite, the Dyson equation has a unique solution \cite[Theorem~2.1]{HeltonRashidiFarSpeicher2007}. 
In particular, $M(\zeta,w)$ from \eqref{eq:def_M_semicircular} coincides with this solution. For constant function $s$ the identity \eqref{general MDE} has appeared in \cite{BordenaveCaputoChafai2014} and its analog for elliptic operators in \cite{Zhong2021}.

\begin{proof}
An easy computation using the properties of the operator-valued circular element $\mathfrak c$, in particular, 
\eqref{second order cumulant of c} and \eqref{eq:def_S_S_star}, 
yields  
the required $\mathcal B^{2\times 2}$-valued free cumulants of $\mathfrak s = \begin{pmatrix} 0 & \mathfrak c \\ \mathfrak c^* &0 \end{pmatrix}$ and shows that $\mathfrak s$ is an $\mathcal B^{2\times 2}$-valued semicircular element, see e.g.\ \cite[Definition~10~2) in Chapter 9]{MingoSpeicherBook}. 

Owing to \eqref{second order cumulant of c} and \eqref{eq:def_S_S_star}, 
the covariance of $\mathfrak s$ satisfies $E[\mathfrak s b \mathfrak s] = \Sigma[b]$ for all $b \in \mathcal B^{2\times 2}$ 
with $\Sigma$ from \eqref{eq:def_Sigma}. 
Since the $R$-transform of an operator-valued semicircular element coincides with its covariance, 
 see e.g.\ \cite[Theorem~11.~4) in Chapter~9]{MingoSpeicherBook}, $M(\zeta,w)$ as defined in \eqref{eq:def_M_semicircular} satisfies \eqref{general MDE} with $\Sigma$ from \eqref{eq:def_Sigma} by \cite[Theorem~11.~1) in Chapter~9]{MingoSpeicherBook}. 
\end{proof}

We  now express the Brown measure of $\mathfrak c + a$ through $M$ from \eqref{eq:def_M_semicircular}
or, equivalently, the solution to \eqref{general MDE}. 
To that end, it suffices to consider $w = \ii \eta$ with $\eta \in (0,\infty)$ and we  see that 
the Dyson equation can be simplified in this case. 
We consider two coupled equations for functions $v_1, v_2 \in \cal{B}$ with $v_1>0$ and $v_2>0$, namely  
\begin{subequations} \label{eq:V_equations} 
\begin{align} 
\frac{1}{ v_1} & =  \eta +Sv_2+  \frac{\abs{\zeta-a}^2}{\eta +S^*v_1} \,, \label{eq:v1} \\ 
\frac{1}{ v_2} & =  \eta +S^*v_1+  \frac{\abs{\zeta-a}^2}{\eta +Sv_2}\,, \label{eq:v2} 
\end{align}
\end{subequations} 
for all $\eta >0$ and $\zeta \in \C$. We also call \eqref{eq:V_equations} \emph{Dyson equation}. 
Given $v_1$ and $v_2$ from \eqref{eq:V_equations}, we introduce $y$ defined by 
\begin{equation} \label{eq:B_relation} 
y:=\frac{v_1\2( \bar a-\bar \zeta)}{\eta + S^*v_1}=\frac{v_2\2(\bar a-\bar \zeta)}{\eta + Sv_2}\,, 
\end{equation}
where the second step follows from 
\begin{equation} \label{eq:v2_St_v1_equals_v1_S_v2} 
v_2(\eta + S^*v_1)=v_1(\eta + Sv_2)\,.
\end{equation}
Indeed, for the proof of the last identity, 
we multiply \eqref{eq:v1} by $v_1 v_2 (\eta + S^*v_1)$ and \eqref{eq:v2} by $v_1 v_2(\eta + Sv_2)$ and conclude \eqref{eq:v2_St_v1_equals_v1_S_v2} from the resulting relations.

The next lemma relates the solutions to \eqref{eq:V_equations} and \eqref{general MDE} and implies existence and uniqueness 
to the one of \eqref{eq:V_equations}.

\begin{lemma} \label{lem:existence_uniqueness}
For any $\zeta \in \C$ and $\eta >0$, the following holds. 
\begin{enumerate}[label=(\roman*)] 
\item \label{item:sol_V_implies_sol_M} 
Let $(v_1,v_2) \in \mathcal B^2$ be a solution of \eqref{eq:V_equations} with $v_1>0$ and $v_2 >0$. 
Set  
\begin{equation} \label{eq:structure_M} 
 M := \begin{pmatrix} \ii v_1 & \ol{y} \\ y & \ii v_2 \end{pmatrix}\in \mathcal B^{2\times 2}
\end{equation} 
with $y$ as in \eqref{eq:B_relation}. 
Then $M$ solves \eqref{general MDE} with $w = \ii \eta$ and $\Im M$ is positive definite.  
\item \label{item:sol_M_implies_sol_V} 
Let $M = M(\zeta,\ii \eta)$ be the solution of \eqref{general MDE} such that $\Im M$ is positive definite. 
Then there are $v_1$, $v_2 \in \mathcal B$ such that $v_1>0$, $v_2>0$, $M$ satisfies \eqref{eq:structure_M}  
with $y$ as in \eqref{eq:B_relation} and $(v_1,v_2)$ satisfy \eqref{eq:V_equations}. 
\end{enumerate} 
In particular, there are unique $v_1$, $v_2 \in \mathcal B$ such that $v_1 >0$ and $v_2>0$ and \eqref{eq:V_equations} holds. 
\end{lemma} 

\begin{proof} 
The proofs of \ref{item:sol_V_implies_sol_M} and \ref{item:sol_M_implies_sol_V} are easy computations. 
Therefore, the existence and uniqueness of the solution to \eqref{eq:V_equations} follows from the existence and uniqueness 
of the solution to \eqref{general MDE}, which is a special case of the general existence and uniqueness result 
\cite[Theorem~2.1]{HeltonRashidiFarSpeicher2007}. 
\end{proof} 
When evaluated on the imaginary axis at $w=\ii \eta$ with $\eta>0$,  the representation \eqref{eq:def_M_semicircular} of the solution to the Dyson equation \eqref{general MDE} with positive definite imaginary part implies the trivial bound  
\begin{equation}\label{eq:trivial bound} 
\norm{M(\zeta,\ii \eta) } \leq \eta^{-1}. 
\end{equation}
Indeed, we note that the argument of $ \mathrm{id} \otimes E$ in \eqref{eq:def_M_semicircular} with $w = \ii \eta$ 
is bounded in norm by $\eta^{-1}$ as the resolvent of a self-adjoint operator. 
Therefore, $\norm{M(\zeta,\ii\eta)} \leq \eta^{-1}$ follows from $\norm{E} = \norm{E(1)} = 1$ 
due to the positivity of $E$ and $E(1) =1$. 
This completes the proof of \eqref{eq:trivial bound}.

Using the solution to \eqref{eq:V_equations} we  characterise the Brown measure in the following proposition  as the Laplacian on $\C$ of the function
 $-\frac{1}{2\pi} W$
defined through 
\begin{equation} \label{eq:def_L} 
  W  (\zeta) := \int_0^\infty \bigg(\avg{v_1(\zeta, \eta)} - \frac{1}{1 + \eta} \bigg) \dd \eta 
\end{equation} 
for each $\zeta \in \C$, where $v_1$ is the first function in the solution pair $(v_1,v_2)$ of the Dyson equation~\eqref{eq:V_equations}. 
If \ref{assum:flatness} holds then the integral in \eqref{eq:def_L} exists in the Lebesgue sense, which follows from 
Lemma~\ref{lmm:Integrating v} below. The  proposition below therefore connects the Brown measure of $a+ \mathfrak c$ to $v_1$.

\begin{proposition}[Characterisation of Brown measure]\label{prp:Brown_measure_construction} 
Let $\sigma$ be the Brown measure of the deformed $\cal{B}$-valued circular element $a + \frak{c}$  with 
$a \in \mathcal B$, \eqref{second order cumulant of c} and 
\eqref{eq:def_S_S_star}. 
Furthermore, let $v_1,v_2>0$ be the solutions of  \eqref{eq:V_equations}. 
If $S$ and $s$ satisfy \eqref{eq:def_operators_S_and_S_star} and \ref{assum:flatness}, respectively, then the following holds.   
\begin{enumerate}[label=(\roman*)]
\item  \label{item:prop_const_i} The Brown measure $\sigma$ is the unique probability measure on $\C$ such that 
\begin{equation} \label{eq:sigma_L_identity} 
 \int_{\C} f(\zeta ) \sigma(\dd \zeta) = -\frac{1}{2\pi} \int_{\C} \Delta f(\zeta)  W  (\zeta ) \dd^2 \zeta 
\end{equation} 
for all $f \in C_0^2 (\C)$, where $\dd^2 \zeta$ denotes the Lebesgue measure on $\C$.  
\item  \label{item:prop_const_ii} The  measure $\sigma$ is also uniquely defined by the identity
\bels{sigma formula in terms of y explicit}{
\int_{\C} f(\zeta ) \sigma(\dd \zeta) = \lim_{\eta \downarrow 0}\frac{1}{ \pi}  \int_{\C} \partial_{{\zeta}}f(\zeta)\avgbb{ \frac{v_1(\zeta, \eta)\2(a- \zeta)}{\eta + S^*v_1(\zeta, \eta)}} \dd^2 \zeta 
}
for all $f \in C_0^1 (\C)$. 
\end{enumerate}
\end{proposition}

The \hyperlink{proof:pro:definition_sigma}{proof of Proposition~\ref{prp:Brown_measure_construction}} 
is presented at the end of Section~\ref{subsec:Representation of Brown measure} below. 
We conclude this section with a short overview of the contents of the remaining paper. The next section provides a detailed analysis of the solution $(v_1,v_2)$ of the Dyson equation as well as a characterisation of $\mathbb{S}$ 
from \eqref{eq:def_mathbb_S}. 
Section~\ref{sec:Properties of the Brown measure} is devoted to the translation of these insights to the 
analysis of the Brown measure. 
We establish the existence of all singularity types in Section~\ref{sec:existence_singularity_types}.

\section{The Dyson equation and its solution} \label{sec:general_setup} 

In this section we analyse the solution $v_1, v_2$ to the Dyson equation \eqref{eq:V_equations}. We begin by establishing bounds on the solution. Then we  determine stability and regularity properties.  Finally we use these insights to  characterise the set $\mathbb{S}$ from \eqref{eq:def_mathbb_S} and expand  $v_1$ and $v_2$ around edge points $\zeta_0 \in \partial \mathbb{S}$. 
 We start with some useful identities and a priori estimates. 

As $S$ and $S^*$ both act on $\mathcal B = L^\infty(\mathfrak X,\mu)$ as integral operators in the form \eqref{eq:def_operators_S_and_S_star}, the identity \eqref{eq:v2_St_v1_equals_v1_S_v2} implies  $\avg{v_2 S^* v_1} = \avg{v_1 Sv_2}$ and thus
\begin{equation} \label{eq:avg_v1_equals_avg_v2}
\avg{v_1} =\avg{v_2}\,.
\end{equation} 
We conclude from \eqref{eq:V_equations} and \eqref{eq:B_relation} that 
\begin{equation} \label{eq:v_simple_bounds} 
 v_1 \leq \eta^{-1}, \qquad v_2 \leq \eta^{-1}, \qquad \abs{y} \leq \abs{a - \zeta} \eta^{-2} 
\end{equation} 
for all $\zeta \in\C$ and $\eta >0$. Furthermore, we have the identity
\begin{equation} \label{eq:y_identities} 
y= \frac{v_1(\bar a-\bar \zeta)}{\eta +S^*v_1} =
\frac{1}{a-\zeta}(1- v_1( \eta +Sv_2))
=\frac{1}{a-\zeta}- \frac{v_1v_2}{\ol{y}}
\end{equation} 
for any $\zeta \in \C\setminus \spec(D_a)$. 
Here, $\spec(D_a)$ denotes the spectrum of $D_a$ considered as multiplication operator $\mathcal B \to \mathcal B$, which coincides with the essential range of $a$.

Throughout the remainder of this section, we assume that $s$ satisfies \ref{assum:flatness}. This implies
\begin{equation} \label{eq:Sprimitiv} 
\bbm{1}_{I_i}S w \sim  \bbm{1}_{I_i} \sum_{j=1}^Kz_{ij}\avg{w\bbm{1}_{I_j}} \,, \qquad  \bbm{1}_{I_i}S^* w \sim \bbm{1}_{I_i} \sum_{j=1}^Kz_{ji}\avg{w\bbm{1}_{I_j}}
\end{equation} 
 for all $w \in \mathcal B$ with $w \geq 0$ and all $i \in \db{K}$.  

\subsection{Bounds on the solution}

This subsection contains  bounds on $v_1$ and $v_2$ under varying assumptions on $s$ and $a$.

\paragraph{Lower bound on $v_1$ and $v_2$} 
We start with some preliminary bounds that require $a \in \mathcal B$ and the boundedness of the function $s$. 

\begin{lemma}[Preliminary bounds on $v_1$ and $v_2$] \label{lem:v_bound_preliminary}
If for some constant $C>0$,  $s(x,y) \leq C$ for all $x$, $y \in \mathfrak X$ then 
\begin{equation} \label{eq:preliminary_bounds_v} 
  \frac{\eta}{\eta^2 + 1 +\abs{\zeta}^2} \lesssim_C v_i \leq \frac{1}{\eta}
\end{equation} 
uniformly for all $i=1$, 2, $\zeta \in \C$ and $\eta >0$. 
\end{lemma} 

\begin{proof} 
The positivity of $v_i$ and $\eta$ in \eqref{eq:V_equations} imply $1/v_i \geq \eta$. Hence, the upper bounds 
in \eqref{eq:preliminary_bounds_v} follow. 

As $s(x,y) \leq C$ for all $x$, $y \in \mathfrak X$, we conclude from \eqref{eq:V_equations} 
and the upper bound $v_2 \leq \eta^{-1}$ that 
\[ \frac{1}{v_1} \leq \eta + C\eta^{-1} + \eta^{-1} 2 ( \norm{a}_\infty^2 + \abs{\zeta}^2), \] 
which implies the missing lower bound in \eqref{eq:preliminary_bounds_v}. 
\end{proof} 

We continue by showing that $v_1$ and $v_2$ are bounded from below by their average. 

\begin{lemma} \label{lem:Saveraging} 
If $s$ satisfies \ref{assum:flatness} then 
\begin{equation} \label{eq:vLowerBound} 
v_1 \gtrsim \avg{v_1} \,, \qquad v_2 \gtrsim \avg{v_2} 
\end{equation} 
uniformly for all $\zeta \in \C$ and $\eta >0$. As a consequence $S^*$ and $S$ act averaging on $v_1$ and $v_2$, respectively, i.e. 
\begin{equation} \label{eq:Saveraging} 
S^*v_1 \sim \avg{v_1} \,, \qquad Sv_2 \sim \avg{v_2} \,.
\end{equation} 
\end{lemma}

The following proof is based on identifying $(v_1,v_2)$ as solution to a variational problem. 
In a similar context, such strategy was also used in \cite[Section~6.2 and Appendix~A.4]{Ajanki_QVE_Memoirs} 
and \cite[Lemmas~3.11 and~3.17]{AltGram}. \

\begin{proof}
For the proof we rephrase \eqref{eq:V_equations} as a variational problem. This will help us to make use of the assumption \ref{assum:flatness}.  
For $\eta >0$, we define  $J_{\eta}\colon \cal{B}_+ \times \cal{B}_+ \to \R \cup \{ + \infty\}$, a functional which is minimised by $(v_1,v_2)$,  by 
\[
J_\eta(x_1,x_2):= \avg{x_1 {S}x_2} +\eta \avg{x_1+x_2}- \avgb{\p{1-u_\eta}\log x_1} -\avgb{\p{1-u_\eta}\log x_2}  \,,
\]
where 
\[
u_\eta:= \frac{|a-\zeta|^2\1v_1}{\eta+S^* v_1} =\frac{|a-\zeta|^2\1v_2}{\eta +S v_2}\,
\]
by \eqref{eq:v2_St_v1_equals_v1_S_v2}.  
Note that $1-u_\eta = v_1(\eta + Sv_2) \in \cal{B}_+$ by \eqref{eq:V_equations}. 

Next, we prove that $(x_1,x_2)=(v_1(\eta),v_2(\eta))$ is a  minimiser for $J_\eta$. Indeed, the functional $J_\eta$  satisfies the lower bounds
\bes{
J_\eta(x_1, x_2)  &\ge  \avg{\eta x_1- \norm{u_\eta}_\infty\bbm{1}(x_1 \ge 1)\log x_1}+ \avg{\eta x_2- \norm{u_\eta}_\infty\bbm{1}(x_2 \ge 1)\log x_2}
\\
&\qquad- (1-\norm{u_\eta}_\infty)(\avg{\log x_1}+\avg{\log x_2})
\\
&\ge
\norm{u_\eta}_\infty\beta_1\pbb{\eta \frac{\alpha_1}{\beta_1}-  \log\frac{\alpha_1}{\beta_1} } + (1-\norm{u_\eta}_\infty)(\eta\avg{x_1} -\avg{\log x_1})
\\
&\qquad 
+\norm{u_\eta}_\infty\beta_2\pbb{\eta \frac{\alpha_2}{\beta_2}-  \log\frac{\alpha_2}{\beta_2} } + (1-\norm{u_\eta}_\infty)(\eta\avg{x_2} -\avg{\log x_2})
 \,,
}
with $\alpha_i:=\avg{x_i\bbm{1}(x_i \ge 1)}$ and $\beta_i:=\avg{\bbm{1}(x_i \ge 1)}$. 
Here, in the second step, we applied Jensen's inequality with the expectation $f \mapsto \beta_i^{-1}  \avg{ \bbm{1}(x_i \geq 1) f }$, i.e.\ $\avg{\bbm{1}(x_i \geq 1) \log x_i} \leq \beta_i\log \alpha_i/\beta_i$. 
Since $\avg{\log x_i} \le \log \avg{x_i}$ by Jensen's inequality, the functional $J_\eta$ is bounded from below for any $\eta >0$. 
Therefore, $\iota_\eta = \inf_{(x_1,x_2) \in \mathcal B_+ \times \mathcal B_+} J_\eta(x_1,x_2) > -\infty$. 
Moreover, $(v_1,v_2)$ satisfies the corresponding Euler-Lagrange equations and $J_\eta$ is strictly convex 
due to $1-u_\eta = v_1(\eta + Sv_2)$ and \eqref{eq:preliminary_bounds_v}. 
Hence, $J_\eta(v_1,v_2) = \iota_\eta$. 

In the regime $\eta \gg 1$, that is there is a constant $\eta_0>0$ depending only on the model parameters such 
that for all $\eta \geq \eta_0$, we deduce from \eqref{eq:V_equations} and the bound $v_i \leq 1/\eta$ in Lemma~\ref{lem:v_bound_preliminary} that  
\[
 v_i \sim  \frac{\eta}{|\zeta|^2 +  \eta^2}, 
\]
which implies $v_i \sim \avg{v_i}$. 
Thus, we restrict to $\eta \lesssim 1$.  In case $\eta \lesssim 1$ and $|\zeta|\gg1 $ we get 
\[
v_i \sim \frac{\eta}{|\zeta|^2},
\]
and therefore $v_i \sim \avg{v_i}$. 
Hence, we restrict to $\eta \lesssim 1$ and $|\zeta|\lesssim  1$  in the following.  
As $(v_1,v_2)$ is a minimiser of $J_\eta$, the value $J_\eta(v_1,v_2)$ is bounded by 
\[
 J_\eta(v_1,v_2)\le J_\eta(1,1) = 2\eta +\avg{{S}1}\,.
\]
With the notations $c_i:= \mu(I_i)$, $\avg{x}_i:= \frac{1}{c_i}\avg{x\bbm{1}_{I_i}}$ and $z_{ij}:= \bbm{1}(s|_{I_i \times I_j}\ne 0)$,  we see by Jensen's inequality that
 \bes{
1 &\gtrsim J_{\eta}(v_1,v_2) 
\\
&=  \avg{v_1 {S}v_2} - \sum_{j}c_j\avgb{\p{1-u}\log v_1}_j -\sum_{j}c_j\avgb{\p{1-u}\log v_2}_j+ \eta \avg{v_1 + v_2}
\\
&\ge c\1\sum_{i,j} z_{ij}\avg{v_1}_i\avg{v_2}_j - \sum_{j}c_j\avg{1-u}_j\log\frac{ \avg{(1-u)v_1}_j\avg{(1-u)v_2}_j}{\avg{1-u}_j^2}
\\
&\ge c\1\sum_{i,j} z_{ij}\avg{v_1}_i\avg{v_2}_j - \sum_{j}c_j\avg{1-u}_j\log \pb{\avg{v_1}_j \avg{v_2}_j}+2\sum_{j}c_j\avg{1-u}_j\log\avg{1-u}_j
\\
&\ge  \sum_{j} \varphi_j(\chi_j)-\frac{2}{\ee}\sum_{j}c_j\,,
}
where in the last step we used $z_{jj} = 1$ for all $j \in \db{K}$ and defined 
\[
\chi_j:=\avg{v_1}_j \avg{v_2}_j\,, \qquad \varphi_j(\chi):=c\1\chi-c_j\1\bbm{1}(\chi \le 1)(1-\avg{u}_j)\log \chi -\bbm{1}(\chi > 1)\log \chi\,
\]
for all $j \in \db{K}$ and $\chi \in (0,\infty)$. 
For all $j \in \db{K}$, we infer $\chi_j \lesssim 1$ and 
\bels{lower bound on nuj}{
\chi_j \ge \exp\pbb{-\frac{C}{1- \avg{u}_j}} =\exp\pbb{-\frac{C}{\avg{v_1(\eta+Sv_2)}_j}}
}
for some positive constant $C \sim 1$. 
We rewrite the Dyson equation \eqref{eq:V_equations} in the form
\bels{rewritten Dyson equation}{
v_1 = \frac{\eta +S^*v_1}{|a-\zeta|^2+(\eta +S^*v_1)(\eta+Sv_2)}\,, \qquad v_2 = \frac{\eta +Sv_2}{|a-\zeta|^2+(\eta+S^*v_1)(\eta+Sv_2)}\,.
}
Now we set $x:=(\eta+S^*v_1)(\eta+Sv_2)$. In particular, we have  
\begin{equation} \label{eq:x_I_i_sim_avg} 
\bbm{1}_{I_i}  x \sim \bbm{1}_{I_i} \pbb{\eta^2 + \sum_{j,k} z_{ij}z_{k i}\avg{v_1}_k \avg{v_2}_j } \sim\bbm{1}_{I_i}  \avg{x}_i \,
\end{equation} 
due to \eqref{eq:Sprimitiv}.
Suppose now that $\avg{x}_i \gg 1$ for an $i \in \db{K}$. Then  by taking the $\avg{\2\cdot\2}_i$-average of \eqref{rewritten Dyson equation}, multiplying the two equations and using \eqref{eq:Sprimitiv} and \eqref{eq:x_I_i_sim_avg} as 
well as $\avg{x}_i \gg 1$ we get 
\[
\chi_i \sim \frac{1}{\avg{x}_i}\ll 1\,.
\]
On the other hand by multiplying the first equation in \eqref{rewritten Dyson equation} with $\eta +Sv_2$ we find 
\[
\avg{v_1(\eta +Sv_2)}_i \sim 1\,
\]
as $\avg{x}_i \gg 1$.  
This contradicts \eqref{lower bound on nuj} and $\chi_i \ll 1$ and we conclude $\avg{x}_i \lesssim 1$ for all 
$i \in \db{K}$. 
With \eqref{rewritten Dyson equation} and \eqref{eq:x_I_i_sim_avg}, this implies 
\[
v_1 \gtrsim S^*v_1\,, \qquad v_2 \gtrsim Sv_2\,.
\]
Iterating these inequalities yields $v_1 \gtrsim \avg{v_1}$ and $v_2 \gtrsim \avg{v_2}$, 
i.e.\ \eqref{eq:vLowerBound}, since $(z_{ij})_{i,j}$ is primitive by \ref{assum:flatness}. Since $z_{ii}=1$ for all $i \in \db{K}$,  we have $S1 \gtrsim 1$ and $S^* 1 \gtrsim 1$. This implies $Sv_2 \gtrsim \avg{v_2} $ and $S^* v_1 \gtrsim \avg{v_1}$, respectively. The upper bounds on $Sv_2$ and $S^*v_1$ in \eqref{eq:Saveraging} follow from the upper bound on $s$ in \ref{assum:flatness}. 
\end{proof}

\paragraph{Bound in $L^2$-norm:} 
Now we show a bound with respect to the norm on $L^2$.

\begin{lemma} \label{lem:hilbert_schmidt} 
If $s$ satisfies \ref{assum:flatness} then 
\begin{equation} \label{eq:hilbert_schmidt_bound} 
 \avg{v_1^2} + \avg{v_2^2} + \avg{\abs{y}^2} \lesssim 1
\end{equation} 
uniformly for all $\zeta \in \C$ and $\eta >0$. 
\end{lemma} 

\begin{proof} 
We multiply the first relation in \eqref{eq:v1} by $v_1^2$ and estimate $v_1 \geq v_1^2 S v_2 \gtrsim v_1^2 \avg{v_2} = v_1^2 \avg{v_1}$ 
due to \eqref{eq:Saveraging} and \eqref{eq:avg_v1_equals_avg_v2}. 
Averaging this  estimate  and using $\avg{v_1} >0$ yields $\avg{v_1^2} \lesssim 1$. 
The bound $\avg{v_2^2} \lesssim 1$ is proved analogously. 
From \eqref{eq:B_relation}, we conclude 
\begin{equation} \label{eq:aux_upper_bound_y} 
\abs{y}^2 = \frac{v_1^2 \abs{a - \zeta}^2}{(\eta + S^* v_1)^2} 
\leq \frac{v_1^2 \abs{a - \zeta}^2}{(\eta + S^* v_1)^2} + \frac{v_1^2(\eta + S v_2)}{\eta + S^* v_1} 
= \frac{v_1}{\eta + S^* v_1}, 
\end{equation} 
where we used \eqref{eq:V_equations} in the last step. 
Hence, \eqref{eq:Saveraging} implies 
\[ \avg{\abs{y}^2} \lesssim \frac{\avg{v_1}}{\eta + \avg{v_1}} \leq 1. \qedhere\] 
\end{proof}

\begin{corollary}\label{cor:bounds on v without regularity}
Let  $a \in \mathcal B$ and $s$ satisfy \ref{assum:flatness}.  Then 
\[
\frac{\eta + \avg{v_i}}{1+\eta^2 + \abs{\zeta}^2 } \lesssim v_i \lesssim \frac{1}{\eta + \avg{v_i}}\,, \qquad i=1,2\,.
\]
\end{corollary}

\begin{proof}
We only consider the case $i = 1$. The case $i=2$ follows analogously. 
For the lower bound, we start from \eqref{eq:V_equations} and get 
\[
v_1 =\frac{\eta + S^*v_1}{(\eta + S^*v_1)(\eta + Sv_2) + \abs{\zeta -a}^2} \gtrsim \frac{\eta + \avg{v_1}}{\eta^2 + \avg{v_1}^2 + \abs{\zeta}^2 + \norm{a}_\infty^2} \gtrsim \frac{\eta + \avg{v_1}}{1+\eta^2 + \abs{\zeta}^2 }\,, 
\]
using $\avg{v_1} =\avg{v_2}$ by \eqref{eq:avg_v1_equals_avg_v2}, \eqref{eq:hilbert_schmidt_bound} and \eqref{eq:Saveraging}.
For the upper bound, \eqref{eq:V_equations}, $\avg{v_1} =\avg{v_2}$ and \eqref{eq:Saveraging} imply
\[
v_1 \le \frac{1}{\eta +Sv_2} \sim \frac{1}{\eta +\avg{v_1}}\,. \qedhere 
\]
\end{proof}

\paragraph{Bound in supremum norm:} 
Under stronger assumptions on $s$ and $a$, we can also get a bound on $v_i$ in the $L^\infty$-norm. 
Let $\Gamma_{a,s}$ be as  defined in \eqref{eq:def_Gamma}.

\begin{lemma}\label{lem:Gamma_implies_bounded_solution} 
Assuming the upper bound $\max\cb{\norm{v_1}_2,\norm{v_2}_2} \le \Lambda$ for some $\Lambda\ge 1$ and $4\Lambda^2 <  \lim_{\tau \to \infty}\Gamma_{a,s}(\tau)$, then we have  
 the bound 
\[
\max\cb{\norm{v_1}_\infty,\norm{v_2}_\infty} \le \frac{\Gamma_{a,s}^{-1}(4\Lambda^2)}{\Lambda}\,.
\] 
\end{lemma}

We remark that if $s$ satisfies \ref{assum:flatness}, then by Lemma~\ref{lem:hilbert_schmidt} a  $\Lambda\ge 1$ with $\Lambda\sim 1$ exists such that $\max\cb{\norm{v_1}_2,\norm{v_2}_2} \le \Lambda$ holds uniformly for all $\eta>0$ and $\zeta \in \C$.

 \begin{proof} 
We specialise the proof of \cite[Lemma~9.1]{AEK_Shape} to our setting.
 \end{proof}

\begin{corollary} \label{cor:boundedness} 
 Let $a \in \mathcal B$.  
If $s$ and $a$ satisfy \ref{assum:flatness} and \ref{assum:Gamma} then 
\[ \norm{v_1}_\infty + \norm{v_2}_\infty \lesssim 1\] 
uniformly for all $\zeta \in \C$ and $\eta >0$. 
\end{corollary}

\paragraph{Scaling relations} 

\begin{lemma} \label{lem:v_scaling} 
Let  $a \in \mathcal B$ and   $s$ satisfy \ref{assum:flatness} and \ref{assum:Gamma}. 
Then the following holds. 
\begin{enumerate}[label=(\roman*)] 
\item \label{item:v_i_sim_avg_v_i} 
Uniformly for all $\zeta \in \C$ and $\eta >0$, we have
\[ v_1 \sim \avg{v_1} = \avg{v_2} \sim v_2, \qquad \abs{y} \lesssim 1. \] 
\item \label{item:zeta_minus_a_near_edge} 
For any sufficiently small positive constant $c\sim 1$ the inequalities $\eta + \avg{v_1(\zeta,\eta)} \leq c$ and $\abs{\zeta} \le  1/c$ imply 
$\abs{\zeta - a} \sim 1$ and $\abs{y} \sim 1$. 
\end{enumerate} 
\end{lemma}

Before going into the proof of Lemma~\ref{lem:v_scaling}, we remark that if  $a \in \mathcal B$ and  $s$ satisfies \ref{assum:flatness} 
 then 
\begin{equation} \label{eq:v_1_large_eta} 
 \norm{v_1(\zeta, \eta) - (1 + \eta)^{-1}} \lesssim (1 + \abs{\zeta})\eta^{-2} 
\end{equation} 
uniformly for $\eta >0$ and $\zeta \in \C$. 
Indeed, for $\eta \in(0,1]$, \eqref{eq:v_1_large_eta} is a trivial consequence of \eqref{eq:v_simple_bounds}. 
For $\eta \geq 1$, \eqref{eq:v_1_large_eta} follows by inverting \eqref{eq:v1}, subtracting $(1 + \eta)^{-1}$ on both sides and 
estimating the right hand side using \eqref{eq:Saveraging} and Lemma~\ref{lem:hilbert_schmidt}.

\begin{proof} 
We first prove that $v_1 \sim \avg{v_1} = \avg{v_2} \sim v_2$. 
As $s$ satisfies \ref{assum:flatness}, equation \eqref{eq:V_equations}, $v_1$, $v_2 > 0$ and \eqref{eq:Saveraging} imply 
\begin{equation} \label{eq:v1_sim_v2} 
 v_1 \sim v_2 \,.
\end{equation} 
Hence, it suffices to show $v_1 \sim \avg{v_1}$ due to \eqref{eq:avg_v1_equals_avg_v2}. 

From \eqref{eq:v_1_large_eta}, we conclude that $v_1 \sim (1+ \eta)^{-1}$ and $\avg{v_1} \sim (1 + \eta)^{-1}$ uniformly for $\eta \gtrsim 1$ and $\abs{\zeta} \lesssim 1$. 
This proves Lemma~\ref{lem:v_scaling} in that regime. 
If, on the other hand,  $\abs{\zeta} \geq \norm{a}_\infty + 1$ then $\abs{\zeta - a} \sim \abs{\zeta}$. Hence, for such $\zeta$, we 
conclude from \eqref{eq:v1} and \eqref{eq:Saveraging} that 
\[ \frac{1}{v_1} \sim \eta + \avg{v_2} + \frac{\abs{\zeta}^2}{\eta + \avg{v_1}}. \] 
As the right hand side is a constant function on $\mathfrak X$, we obtain $v_1 \sim \avg{v_1}$ if $\abs{\zeta} \geq \norm{a}_\infty + 1$. 

Hence, it remains to consider $\abs{\zeta} \lesssim 1$ and $\eta \lesssim 1$. In particular, $\abs{\zeta - a} \lesssim 1$ 
 as $a \in \mathcal B$.  
Thus, \eqref{eq:v1}, \eqref{eq:Saveraging} and \eqref{eq:avg_v1_equals_avg_v2}  imply 
\begin{equation} \label{eq:aux_v1_scaling} 
 \eta + \avg{v_1} \sim v_1 ((\eta + \avg{v_1})^2 + \abs{\zeta - a}^2).
\end{equation} 
Together with Lemma~\ref{lem:hilbert_schmidt}, this yields $\eta + \avg{v_1} \lesssim v_1$.
We conclude $v_1 \gtrsim \avg{v_1}$ and $v_1 \gtrsim \eta$ as well as $\avg{v_1} \gtrsim \eta$. 
If $\abs{\zeta - a} \geq c$ for any $c \sim 1$ then $v_1 \lesssim \eta + \avg{v_1} \sim \avg{v_1}$ 
by \eqref{eq:aux_v1_scaling}. 
Therefore we conclude $v_1 \sim \avg{v_1}$ if $\abs{\zeta - a } \geq c$. 
What remains is the case $\abs{\zeta - a} \leq c$ and $\eta \leq c$ for some constant $c \sim 1$. 
As $v_1\lesssim 1$ by \ref{assum:Gamma} and Lemma~\ref{lem:Gamma_implies_bounded_solution}, we conclude 
from \eqref{eq:aux_v1_scaling} that $1 \gtrsim \eta + \avg{v_1}$ or $1 \gtrsim \frac{ \abs{\zeta - a}^2}{\eta +
\avg{v_1}}$. 
In the second case, \eqref{eq:aux_v1_scaling} implies 
$ \avg{v_1} \lesssim \eta + \avg{v_1} \sim v_1 ( \abs{\zeta - a }^4 + \abs{\zeta - a}^2)$.  
Using $\abs{\zeta- a} \leq c$, choosing $c \sim 1$ sufficiently small and averaging $\avg{v_1} \lesssim v_1 ( \abs{\zeta - a}^4 + \abs{\zeta - a}^2)$ yield a contradiction as $\avg{v_1} > 0$. 
Hence, $\avg{v_1} \gtrsim 1$ and, thus, $\avg{v_1} \sim 1$ by Lemma~\ref{lem:hilbert_schmidt} as well as $1 \gtrsim v_1$ by \eqref{eq:aux_v1_scaling} as $\eta \leq c$ 
for some small enough $c \sim 1$. 
Since $v_1 \lesssim 1$ by Lemma~\ref{lem:Gamma_implies_bounded_solution}, this completes the proof of $v_1 \sim \avg{v_1}$ uniformly for $\zeta \in \C$ and $\eta >0$. This completes the proof of \ref{item:v_i_sim_avg_v_i}. 

For the proof of \ref{item:zeta_minus_a_near_edge},  from $v_1 \sim \avg{v_1}$, \eqref{eq:Saveraging}  and \eqref{eq:aux_upper_bound_y}, we conclude $\abs{y} \lesssim 1$ uniformly for $\zeta \in \C$ and $\eta >0$. 
Owing to \eqref{eq:aux_v1_scaling} and $v_1 \sim \avg{v_1}$, we have $\eta + \avg{v_1} \sim \avg{v_1}(\eta + \avg{v_1})^2 + \avg{v_1} \abs{\zeta - a}^2 $. 
As $\eta + \avg{v_1} \leq c$, by choosing $c\sim 1$ sufficiently small, we can incorporate $\avg{v_1}(\eta + \avg{v_1})^2$ 
into the left-hand side and obtain $\avg{v_1} \lesssim \abs{\zeta - a}^2 \avg{v_1}$. Hence, $1 \lesssim \abs{\zeta - a}$ as $\avg{v_1}>0$  for $\eta >0$.  The bound $\abs{y} \sim 1$ follows from \eqref{eq:y_identities}. 
This proves the additional statement and completes the proof of Lemma~\ref{lem:v_scaling}. 
\end{proof}

\subsection{A relation between the derivatives of $M$} 
In this subsection, we restrict the Dyson equation \eqref{general MDE} to the imaginary axis $w=\ii \eta$ for $\eta>0$, i.e. we consider the solution $M=M(\zeta, \ii \eta)$ of
\begin{equation} \label{eq:mde} 
- M^{-1} =  \begin{pmatrix} \ii \eta & \zeta - a \\ \overline{\zeta- a} & \ii \eta \end{pmatrix} 
+ \Sigma[M]. 
\end{equation}  
We take  derivatives of $M$  with respect to $\eta$, $\zeta$ and $\bar \zeta$. 
We establish a useful relation between these derivatives in the next lemma. Here, we use the notation
\[
\avg{R}: = \frac{1}{2}(\avg{r_{11}} + \avg{r_{22}}) \,, \qquad R= \mtwo{r_{11} & r_{12}}{r_{21} & r_{22} } \in \cal{B}^{2 \times 2}\,.
\]

\begin{lemma} \label{lem:derivatives_M} 
In the setup of Lemma~\ref{lmm:circular_element_Dyson} we have  
\begin{equation} \label{eq:relations_derivatives} 
\avg{\partial_\eta M_{21}(\zeta, \ii \eta) }= 2\ii \avg{\partial_\zeta M(\zeta, \ii\eta)}, \qquad 
\avg{\partial_\eta M_{12}(\zeta, \ii\eta)} = 2 \ii \avg{\partial_{\bar \zeta} M(\zeta, \ii\eta)} 
\end{equation}  
for every $\zeta \in \C$ and $\eta >0$, where we decomposed 
\[ M(\zeta,\ii\eta) = \begin{pmatrix} M_{11}(\zeta,\ii\eta) & M_{12}(\zeta,\ii\eta) \\ M_{21}(\zeta,\ii\eta) & M_{22}(\zeta,\ii\eta) \end{pmatrix}. \] 
\end{lemma}

 Here, the derivatives of $M$ are taken entrywise. 
Before proving Lemma~\ref{lem:derivatives_M}, we note that the Schur complement formula implies 
\begin{equation} \label{eq:Schur} 
\begin{pmatrix} - w & \mathfrak d \\ \mathfrak d^* & - w\end{pmatrix} ^{-1} 
= \begin{pmatrix} w (\mathfrak d\mathfrak d^* - w^2)^{-1} & \mathfrak d (\mathfrak d^* \mathfrak d - w^2)^{-1} \\ (\mathfrak d^* \mathfrak d -w^2)^{-1} \mathfrak d^* & w (\mathfrak d^* \mathfrak d -w^2)^{-1} \end{pmatrix} 
\end{equation} 
for any $\mathfrak d \in \mathcal A$ and any $w \in \C$ with $\Im w>0$. 

\begin{proof}[Proof of Lemma~\ref{lem:derivatives_M}]
Throughout the following, we set $\abs{\mathfrak d} = \sqrt{\mathfrak d^* \mathfrak d}$ for any $\mathfrak d \in \mathcal A$. 
Let $\eta \in (0,\infty)$ and $\zeta \in \C$. 
Owing to \eqref{eq:def_determinant} and the invertibility of $\abs{a + \mathfrak c -\zeta}^2 + \eta^2$, we have 
\begin{equation*} 
 \log D( \abs{a+ \mathfrak c - \zeta}^2 + \eta^2) = \avg{E[ \log ( \abs{a + \mathfrak c -\zeta}^2 + \eta^2)]}.  
\end{equation*} 
In particular, the right-hand side is infinitely often differentiable with respect to $\zeta$, $\overline{\zeta}$ and 
$\eta$ if $\eta >0$. 
Differentiating with respect to $\zeta$, $\overline{\zeta}$ and $\eta$ yield 
\begin{subequations} \label{eq:derivatives_determinant} 
\begin{align} 
\partial_\zeta \log D( \abs{a+ \mathfrak c - \zeta}^2 + \eta^2) & = 
- \avg{E[ ( \abs{a + \mathfrak c - \zeta}^2 + \eta^2)^{-1} (a + \mathfrak c - \zeta)^*]} = - \avg{M_{21}(\zeta,\eta)}
\label{eq:y_from_determinant},  \\ 
 \partial_{\bar \zeta} \log D( \abs{a+ \mathfrak c - \zeta}^2 + \eta^2) & = 
- \avg{E[ ( \abs{a + \mathfrak c - \zeta}^2 + \eta^2)^{-1} (a + \mathfrak c - \zeta)]} = - \avg{M_{12}(\zeta,\eta)}, \\ 
\partial_\eta \log D( \abs{a+ \mathfrak c - \zeta}^2 + \eta^2) & = 2 \eta \avg{E[(\abs{a + \mathfrak c - \zeta}^2 + \eta^2)^{-1}]} = - 2 \ii \avg{M_{22}(\zeta,\eta)}.  
\end{align} 
\end{subequations} 
 Here, we used in the third steps that, for all $\eta>0$ and $\zeta \in \C$, \eqref{eq:def_M_semicircular} and \eqref{eq:Schur} imply 
\[ M(\zeta,\ii\eta) = \begin{pmatrix} M_{11}(\zeta,\ii\eta) & M_{12}(\zeta,\ii\eta) \\ M_{21}(\zeta,\ii\eta) & M_{22}(\zeta,\ii\eta) \end{pmatrix} 
= \begin{pmatrix} \ii \eta E[(\mathfrak d\mathfrak d^* + \eta^2)^{-1}] & E[\mathfrak d (\mathfrak d^* \mathfrak d + \eta^2)^{-1}] \\ E[(\mathfrak d^* \mathfrak d + \eta^2)^{-1} \mathfrak d^*] & \ii \eta E[(\mathfrak d^* \mathfrak d + \eta^2)^{-1}] \end{pmatrix} 
 \] 
with $\mathfrak d = a + \mathfrak c - \zeta$. 
As $\avg{M_{22}} = \avg{v_2} = \avg{v_1} = \avg{M_{11}}$ by \eqref{eq:avg_v1_equals_avg_v2}, Lemma~\ref{lem:derivatives_M} follows from \eqref{eq:derivatives_determinant} 
due to the exchangebility of the derivatives. 
\end{proof}

Let $\cal L$ be the stability operator of \eqref{eq:mde}, defined as
\begin{equation} \label{eq:def_stability_operator} 
\cal L \colon \mathcal B^{2\times 2} \to \mathcal B^{2\times 2}, \quad R \mapsto \cal L[R] := M^{-1} R M^{-1} - \Sigma[R]. 
\end{equation} 
This operator is  invertible for any $\zeta \in \C$ and $\eta >0$ 
due to Lemma~\ref{lem:stability_eta_positive} below.  
Therefore, the implicit function theorem applied to \eqref{eq:mde} and simple computations show that 
\begin{equation} \label{eq:derivatives_M} 
\partial_\zeta M = \cal L^{-1} [E_{12}], \qquad 
 \partial_{\bar \zeta} M = \cal L^{-1} [E_{21}], \qquad 
 \partial_{\eta} M = \ii \cal L^{-1} [E_+]  
\end{equation}  
for all $\eta>0$ and $\zeta \in \C$, 
where we used the notations $E_{12}$, $E_{21}$ and $E_+$ for the elements of $\mathcal B^{2\times 2}$ defined through 
\begin{equation} \label{eq:def_E_12_E_21_E_plus} 
E_{12} := \begin{pmatrix} 0 & 1 \\ 0 & 0 \end{pmatrix}, \qquad
E_{21} := \begin{pmatrix} 0 & 0 \\ 1 & 0 \end{pmatrix}, \qquad E_+ := \begin{pmatrix} 1 & 0 \\ 0 & 1 \end{pmatrix}. 
\end{equation}

\subsection{Stability of Dyson equation and analyticity of its solution} 
In this section we show how the solution $v_1, v_2$ of \eqref{eq:V_equations} can be extended to $\eta =0$. 
The characterization of the Brown measure in Proposition~\ref{prp:Brown_measure_construction} in combination with Lemma~\ref{lem:existence_uniqueness} shows how the Brown measure $\sigma$ relates to the solution of \eqref{eq:V_equations} at the origin, i.e. at  $\eta=0$. Now we introduce a deterministic analog of the $\eps$-pseudospectrum. 
 For this purpose we note that  
the map $w \mapsto \avg{M(\zeta, w)}$ is the Stieltjes transform of a probability measure on $\R$.  Indeed,   \cite[Proposition~2.1 and Definition~2.2]{AEK_Shape} show that 
\begin{equation}\label{eq:Stieltjes transform}
\avg{M(\zeta, w)} = \int_\R\frac{\rho_\zeta(\dd \tau )}{\tau-w}
\end{equation}
holds for all $w \in \C$ with $\im w>0$ and a unique probability measure $\rho_\zeta$ whose support  satisfies $\supp \rho_\zeta \subset \big( \spec(D_{|a-\zeta|})\cup \spec(-D_{|a-\zeta|}) \big) + [-2 \norm{S}_\infty^{1/2},2 \norm{S}_\infty^{1/2}]$.

\begin{definition} \label{def:rho_zeta} 
Let  $\rho_\zeta$ be the unique probability measure on $\R$ for which \eqref{eq:Stieltjes transform} holds.  
 Through $\rho_\zeta$ we define
\bels{eq:def_S_eps}{
 \mathbb S_\eps:= \{\zeta \in \C: \dist(0, \supp \rho_\zeta)\le  \eps\}
}
for any $\eps\geq 0$.
\end{definition}

\begin{remark}\label{rmk:S eps}
The sets $\mathbb S_\eps$ defined in \eqref{eq:def_S_eps} are 
monotonously nondecreasing in $\eps\geq 0$, i.e.\  $\mathbb S_{\eps_1} \subset \mathbb S_{\eps_2}$ if $\eps_1 \le \eps_2$. 
Moreover, they are bounded, in fact, $\mathbb S_{\eps} \subset \{\zeta \in \C: \abs{\zeta} \le \eps + \norm{a}_\infty + 2 (\norm{S}_{\infty})^{1/2}\}$ for all $\eps \geq 0$ as a consequence of \cite[Proposition~2.1]{AEK_Shape}.
 Here, $\norm{S}_\infty$ denotes the operator norm of $S$ viewed as an operator from $\mathcal B$ to $\mathcal B$.  
\end{remark}

 If we stay away from $\mathbb S_\eps$, then the solution is extended to $\eta=0$ by setting $v_i=0$ by the following lemma.

\begin{lemma} \label{lem:v_away_support} 
 Let $\eps >0$. 
Let $\zeta \in  (\C \setminus \mathbb S_\eps)\cap \DD_{1/\eps}$. Then $v_i(\zeta, \eta)\sim_\eps \eta$ for all $\eta \in (0,1]$ and $i=1,2$. In particular, $v_i$ is continuously extended to $ \zeta \in \C\setminus \mathbb S_0$  and $\eta =0$ by setting $v_i(\zeta, 0):=0$. 
\end{lemma}

\begin{proof}
 From \eqref{eq:v1}, we conclude $v_1 (\abs{\zeta- a}^2 + (\eta + S v_2)(\eta + S^*v_1)) = \eta + S^* v_1 \geq \eta$. 
Thus, $\abs{\zeta} \leq \eps^{-1}$, $a \in \mathcal B$, $\eta \leq 1$, \eqref{eq:Saveraging} and Lemma~\ref{lem:hilbert_schmidt} imply $v_1 \gtrsim_\eps \eta$ for all $\eta \in (0,1]$. Similarly, $v_2 \gtrsim_\eps \eta$ for all $\eta \in (0,1]$. 
On the other hand, as $\zeta \in \C\setminus \mathbb S_\eps$, the statement (v) of \cite[Lemma~D.1]{AEK_Shape} holds for $\tau = 0$.  
Hence, \cite[Lemma~D.1 (i)]{AEK_Shape} implies $\max\{v_1(\zeta, \eta), v_2(\zeta, \eta)\} \leq \norm{\Im M(\zeta,\ii \eta))}  \lesssim \eta$   for all $\eta \in (0,c]$ for some sufficiently small $c \sim_\eps 1$.  
If $\eta \in (c,1]$ then the upper bound in Corollary~\ref{cor:bounds on v without regularity} yields $v_i \lesssim \eta^{-1} \sim \eta$ for all $\eta \in (c,1]$. This completes the proof. 
\end{proof}

 The next proposition states that 
if $\avg{v_1}=\avg{v_2}$ remains bounded away from zero as $\eta \downarrow 0$, then the solution has an analytic extension to $\eta=0$. 

\begin{proposition}[Analyticity in the bulk]\label{prp:Analyticity in the bulk}
Let $s$  satisfy \ref{assum:flatness} and 
  $\zeta \in \C$ with $\limsup_{\eta \downarrow 0}\avg{v_1(\zeta, \eta)}> 0$. Then $v_1,v_2: \C \times (0,\infty) \to (0,\infty)$  
 has an extension to a neighbourhood of $(\zeta,0)$ in $\C\times \R$ which is real analytic in all variables.
\end{proposition}

To prove this proposition, we show that the Dyson equation \eqref{eq:V_equations} is stable even for $\eta=0$. However, the equation does not have a unique solution on $\cal{B}_+^2$ for $\eta=0$ without the additional constraint $\avg{v_1}=\avg{v_2}$. Therefore, we have to reformulate the equation to incorporate this constraint.  Proposition~\ref{prp:Analyticity in the bulk} is proved at the end of this subsection.

We recall that $\cal B_+:=\{w \in \cal B:w> 0\}$ and set 
\[
e_- = (1, - 1) \in {\cal B}^2 := \cal B \oplus \cal B\,, \qquad e_-^\perp:= \{h=(h_1,h_2) \in\cal B^2: \avg{h_1} = \avg{h_2} \}.
\]
 For $\eta >0$ and $\zeta \in \C$,  we define $J\equiv J_{\zeta, \eta}\colon e_-^\perp \cap \cal B_+^2 \to e_-^\perp, (w_1,w_2) \mapsto (J_1(w_1,w_2),J_2(w_1,w_2))$ 
through
\bes{
J_1(w_1,w_2)&:=(\eta + Sw_2)\pbb{w_1- \frac{\eta + S^*w_1}{(\eta + S^*w_1)(\eta + Sw_2)+ \abs{a-\zeta}^2}}\,,
\\
J_2(w_1,w_2)&:=(\eta + S^*w_1)\pbb{w_2- \frac{\eta + Sw_2}{(\eta + S^*w_1)(\eta + Sw_2)+ \abs{a-\zeta}^2}}\,. 
}
Then  \eqref{eq:V_equations} takes the form $J(v)=0$ with $v=(v_1,v_2) \in \cal{B}_+^2$.

On $\cal B^2$, we introduce a tracial state and a scalar product defined through 
\bels{scalar product on B2}{
\avgbb{\begin{pmatrix} x_1 \\ x_2 \end{pmatrix}} := \frac{1}{2} \big( \avg{x_1} + \avg{x_2} \big), \qquad \qquad 
\scalarbb{\begin{pmatrix} x_{1} \\ x_{2} \end{pmatrix} }{\begin{pmatrix} y_1 \\ y_2 \end{pmatrix} }:= \frac{1}{2}\pb{\avg{\ol{x}_1y_1}+\avg{\ol{x}_2y_2}}
}
for $x_1$, $x_2$, $y_1$, $y_2 \in \mathcal B$. 
For $x \in\mathcal B^2$, we write $\norm{x}_2 := \sqrt{\scalar{x}{x}}$.  We also interpret $\mathcal B^2$ as an algebra equipped with the componentwise multiplication $(x_1x_2)(y_1,y_2):=(x_1y_1, x_2y_2)$. 

For the rest of this section we will assume that $s$ satisfies \ref{assum:flatness}. 
 Until the proof of Proposition~\ref{prp:Analyticity in the bulk}, we fix $\zeta \in \C$ such that $\limsup_{\eta \downarrow 0}\avg{v_1(\zeta, \eta)}\ge \delta $ for some $\delta>0$.  
Under these conditions, $J$ remains well defined on $e_-^\perp \cap \mathcal B_+^2$  even for $\eta=0$ and we set $J_0:=  J_{\zeta,\eta=0}$.  
We now 
pick candidates for $v_1(\zeta,0)$ and $v_2(\zeta,0)$ by choosing weakly convergent subsequences in the limit $\eta \downarrow 0$.  
By Lemma~\ref{lem:hilbert_schmidt}, there are $v_0 \in (L^2)^2 := L^2 \oplus L^2$ and a monotonically decreasing sequence $\eta_n \downarrow 0$ in $(0,1]$ such that  $v^{(n)} = v(\zeta, \eta_n)$  is weakly convergent to $v_0$ in $(L^2)^2 $, i.e.\ for any $h \in (L^2)^2$, $\scalar{h}{v^{(n)} - v_0} \to 0$ in the limit $n \to \infty$. We recall that $L^2 = L^2(\mathfrak X, \mathcal A, \mu)$.

\begin{lemma} 
\label{lmm:subsequence}
 Let $v_0=\lim_{n \to \infty}v^{(n)}$  be a weak limit as above.  Then $v_0 \in \cal{B}_+^2 \cap e_-^\perp$  and $\delta \lesssim v_0 \lesssim  \frac{1}{\delta}$. Furthermore, $v_0$ satisfies  \eqref{eq:V_equations}  for $\eta=0$, i.e.\ $J_0(v_0)=0$. 
\end{lemma}

For the following arguments, we introduce the operators $S_o$ and $S_d$ on $\mathcal B^2$ defined through 
\begin{equation} \label{eq:def_So_Sd} 
S_o:= \mtwo{0 & S}{S^* & 0}\,, \quad S_d := \mtwo{S^* & 0}{0 & S}\,.
\end{equation} 
Owing to the upper bound in \ref{assum:flatness}, $S_o$ and $S_d$ can be extended naturally to 
operators on $(L^2)^2$.

\begin{proof}
Since $v^{(n)} \to v_0$ weakly and $\avg{e_- v^{(n)}} =0$, we conclude $v_0 \perp e_-$. Furthermore, for any $h \in \cal{B}_+^2$ we get 
$\avg{h v_0} = \lim_{n \to \infty} \avg{hv^{(n)}} \gtrsim \delta \avg{h}$ because  of Corollary~\ref{cor:bounds on v without regularity} and $\limsup_{n \to \infty}\avg{v^{(n)}} \ge \delta$.   From this we conclude $ v_0 \gtrsim \delta$.
 Similarly, Corollary~\ref{cor:bounds on v without regularity} implies $v_0 \lesssim \frac{1}{\delta}$ and thus  $v_0 \in \cal B_+^2$.

The natural extensions of $S$ and $S^*$ to operators on $L^2$ are Hilbert-Schmidt  operators because $s \in L^2(\mathfrak X \times \mathfrak X, \mu \otimes \mu)$ due to the upper bound in \ref{assum:flatness}. 
In particular, $S$ and $S^*$ are compact operators on $L^2$ and, thus, 
$S_ov^{(n)} \to S_ov_0$ and $S_d v^{(n)} \to S_d v_0$ in $(L^2)^2$.  
The bounds $\delta \lesssim v^{(n)} \lesssim \delta^{-1}$  then imply that $J_{\zeta, \eta_n}(v^{(n)}) \to J_{0}(v_0)$ weakly in $(L^2)^2$. 
Consequently, $J_{0}(v_0)=0$. 
\end{proof}

For the formulation of the next lemma, we note that $\norm{T}_\infty$ denotes the operator norm of an 
operator $T \colon \mathcal B^2 \to \mathcal B^2$ and, analogously, $\norm{T}_2$ is the operator norm 
if $T \colon (L^2)^2 \to (L^2)^2$.

\begin{lemma} \label{lem:bound_stability_operator} 
Let $v_0$ be a weak limit of a sequence $v^{(n)}=v(\zeta,\eta_n)$ as above. 
Then 
\[
\norm{(\nabla J_0|_{w =v_0})^{-1}}_2+\norm{(\nabla J_0|_{w =v_0})^{-1}}_\infty \lesssim_\delta 1.
\] 
\end{lemma}

\begin{proof}[Proof of Lemma~\ref{lem:bound_stability_operator}]
Within this proof we will make use of the auxiliary Lemma~\ref{lmm:resolvent bound for L} below  whose proof relies on ideas from  
  \cite{Altcirc}. Therefore we introduce notations that match the ones from \cite{Altcirc}, namely 
\bels{notations from inhomcirc}{
\tau := (\abs{\zeta-a}^2,\abs{\zeta-a}^2)\,
}
and recall the definitions of $S_o$ and $S_d$ from \eqref{eq:def_So_Sd}. 
Using the notations \eqref{eq:def_So_Sd} and \eqref{notations from inhomcirc}, we  write $J$ in the form 
\[
J(w) = (\eta + S_o w)\pbb{w-\frac{1}{\eta + S_o w + \frac{\tau}{\eta + S_d w}}}\,.
\]
Now we take the directional derivative $\nabla_h J$ of $J$ in the direction $h \in \cal B^2$ with $h \perp e_-$, i.e.\ $\avg{h e_-} =0$,   and evaluate at the solution $v=(v_1,v_2)$. Thus, we find 
\begin{equation} \label{eq:nabla_h_J} 
\nabla_h J |_{w=v}
= (\eta + S_o v)\pbb{h +v^2S_o h - \frac{v^2\tau}{(\eta + S_d v)^2}S_d h} =(\eta + S_o v)\scr L h\,,
\end{equation}
where we used $J(v)=0$ and introduced the linear operator $\scr L\equiv \scr L_{\zeta,\eta}(v) \colon \cal B^2 \to \cal B^2$ as
\bels{def of L}{
\scr Lh:=h +v^2S_o h - \frac{v^2\tau}{(\eta + S_d v)^2}S_d h
}
to guarantee the last equality.  Here, $v^2 =(v_1^2, v_2^2)$ since the algebra $\cal B^2$ is naturally equipped with entrywise multiplication. 
We now restrict our analysis to $\eta=0$ and use the following lemma that provides a resolvent estimate for $\scr L_0=\scr L_{\zeta,0}(v_0)$, the operator evaluated on the weak limit $v_0$. 
\begin{lemma}
\label{lmm:resolvent bound for L}
There is $\eps_\ast \sim_\delta 1$ such that for any $\eps \in (0,\eps_\ast)$ we have the bound 
\bels{resolvent bound for L}{
\sup\cB{\norm{(\scr L_0-\omega)^{-1}}_\#: \omega \not \in  \mathbb{D}_\eps \cup  ((1 + \mathbb{D}_{1+\eps}) \setminus \mathbb{D}_{2\eps}) } \lesssim_{\delta, \eps} 1
}
for $\#=2, \infty$. 
Here, $\mathbb{D}_\eps $ contains the single isolated eigenvalue $0$ of $\scr L_0$ with corresponding right and left eigenvectors $v_-:=e_- v_0$ and $S_ov_-$, i.e.\  
\[
\mathbb{D}_\eps\cap \spec(\scr L_0) =\{0\}\,, \qquad   \ker \scr L_0^2 =\rm{Span}(v_-) \,, \qquad \scr L_0v_-=0\,, \qquad \scr L_0^* S_ov_- =0\,.
\] 
Here, $\scr L_0^*$ is the adjoint of $\scr L$ with respect to the $L^2$-scalar product introduced in \eqref{scalar product on B2}. 
\end{lemma}
The \hyperlink{proof:lmm:resolvent bound for L}{proof of Lemma~\ref{lmm:resolvent bound for L}} follows a strategy similar to the one used to prove stability of the Dyson equation in \cite{Altcirc}, where the case $a=0$ was treated. For completeness we present the proof, adjusted to our setup with nontrivial $a$, in 
Appendix~\ref{app:stability_operator} below. Using  Lemma~\ref{lmm:resolvent bound for L}  we now show that 
\bels{L inverse bounded}{
\norm{\scr L_0^{-1}|_{(S_o v_-)^\perp }}_\#\lesssim_{\delta} 1\,, \qquad \#=2, \infty\,,
}
from which the claim of Lemma~\ref{lem:bound_stability_operator}  immediately follows due to \eqref{eq:nabla_h_J}, \ref{assum:flatness} and $v_0 \gtrsim \delta$. 
To see \eqref{L inverse bounded}, we  apply Lemma~\ref{lmm:Twist lemma} to $A:=C\1\scr L_0$ for some appropriately large  positive constant $C\sim_\delta 1$.  We now check the assumptions of the lemma. 
Note that $\scr L_0$ maps  $e_-^\perp$ to  $(S_o v_-)^\perp$.
By Lemma~\ref{lmm:resolvent bound for L} the right and left eigenvectors of $\scr L_0$ corresponding to the eigenvalue  $0$    are $v_-$ and $S_ov_-$, respectively. 
Moreover, $\scalar{v_-}{e_-} \gtrsim_\delta 1$ as $v_0 \gtrsim \delta$, $\abs{\scalar{e_-}{w}} \le \norm{w}_\#$ and that $\norm{\scr L_0w}_\# \gtrsim_\delta \norm{w}_\#$ for any $w \perp S_ov_-$ due to Lemma~\ref{lmm:resolvent bound for L}. By Lemma~\ref{lmm:Twist lemma}  with the choices $\alpha:=0$, $x:= v_-$ and $y:= \frac{1}{2}e_-$ we get $\norm{\scr L_0w}_\# \gtrsim_\delta \norm{w}_\#$ for any $w \perp e_-$. 
Thus, \eqref{L inverse bounded} is shown.  
\end{proof}

Now we use the stability at $\eta =0$ to finish the proof of the main result of this subsection. 

\begin{proof}[Proof of Proposition~\ref{prp:Analyticity in the bulk}]
Let $\zeta_0 \in\C$ be such that $\limsup_{\eta \downarrow 0}\avg{v_1(\zeta_0, \eta)}> 0$. Let $\eta_n \downarrow 0$ such that $v^{(n)}=v(\zeta_0, \eta_n)$ is weakly convergent in $(L^2)^2$. This is possible, because the family $v(\zeta_0, \eta)$ with $\eta \in (0,1]$ is bounded in $(L^2)^2$ due to   Lemma~\ref{lem:hilbert_schmidt}. By Lemma~\ref{lmm:subsequence} the weak limit $v_0=\lim_{n \to \infty}v^{(n)} $ satisfies the Dyson equation, $J_{\zeta,0}(v_0) = 0$,  and by Lemma~\ref{lem:bound_stability_operator} the Dyson equation is stable at $v=v_0$ and $\eta =0$. By the implicit function theorem we find a real analytic function $\wt{v}$, defined on a neighbourhood $U$ of $(\zeta_0,0)$ in $\C\times \R$, such that $\wt{v}(\zeta,\eta)$ solves \eqref{eq:V_equations} and $\wt{v}(\zeta_0,0)=v_0$. Since $v_0 \gtrsim \delta$ according to Lemma~\ref{lmm:subsequence}, $\wt{v}(\zeta,\eta) >0$ on $U$ if the neighbourhood $U$ is chosen sufficiently small. By uniqueness of the solution to the Dyson equation we conclude $\wt{v}(\zeta,\eta)= v(\zeta, \eta)$ for all $(\zeta, \eta) \in U$.  
\end{proof}

\subsection{Characterisation of $\mathbb{S}$} 

Throughout this section we assume that  $a \in \mathcal B$ and   $s$ satisfies \ref{assum:flatness} and \ref{assum:Gamma}. 
 To generalize \eqref{eq:def_B}, \eqref{eq:def_beta} and \eqref{eq:def_mathbb_S} to the setup introduced in 
Section~\ref{sec:general_setup}, we  
 define an operator $B_\zeta\colon \cal{B} \to \cal{B}$, a function $\beta\colon \C \to \R$ and a subset $\mathbb{S} \subset \C$ through
\bels{def of beta and B and S}{
\beta(\zeta) := \inf_{x \in \cal B_+} \sup_{y \in \cal B_+} \frac{\scalar{x}{B_\zeta \2y}}{\scalar{x}{y}}\,, \qquad B_\zeta:= D_{\abs{a-\zeta}^2}-S\,, \qquad  \mathbb{S}:=\{\zeta \in \C: \beta(\zeta) <0\}\,.
}
We also set 
\begin{equation} \label{eq:def_rho_zeta_0} 
\rho_\zeta(0):=\frac{1}{\pi} \lim_{\eta \downarrow 0}\avg{v_1(\zeta, \eta)}\,.
\end{equation} 
This limit exists, because either $\limsup_{\eta \to 0}\avg{v_1(\zeta, \eta)} >0$, in which case $v_1$ can be analytically extended to $\eta =0$ by Proposition~\ref{prp:Analyticity in the bulk}, or $\limsup_{\eta \to 0}\avg{v_1(\zeta, \eta)} =0$ in which case the limit equals zero as well. 
The definition in \eqref{eq:def_rho_zeta_0} is motivated by the fact that the measure 
$\rho_\zeta$ from Definition~\ref{def:rho_zeta} can be shown to have a density on $\R$ and the value of 
this density at zero would be given by the right hand side in \eqref{eq:def_rho_zeta_0}. 

In the following we will denote by $\lambda_{\rm{PF}}(T)$ the spectral radius  of a compact and positivity preserving operator $T$.  By the Krein-Rutman theorem $\lambda_{\rm{PF}}(T)$ is an eigenvalue of $T$ with a positive eigenvector. We also refer to this eigenvalue as the Perron-Frobenius eigenvalue of $T$.  
In particular the operators $S$ and $S^*$ are compact as mentioned in the proof of Lemma~\ref{lmm:subsequence} and therefore so are $D_x S D_y$ and $D_x S^* D_y$ for $x,y \in \cal{B}$. We use this fact in the statement of the following proposition.

\begin{proposition} 
\label{prp:S characterisation}
Let $a \in \mathcal B$ and  $s$ satisfy \ref{assum:flatness} and \ref{assum:Gamma}.  
The following relations between $\beta$, $\mathbb{S}$,  $\mathbb{S}_\eps$ and $\rho_\zeta$ apply. 
\begin{enumerate}[label=(\roman*)] 
\item \label{continuity of beta}
The function $\C \ni \zeta \mapsto \beta(\zeta)$ is continuous and satisfies $\lim_{\zeta \to \infty} 
\beta(\zeta ) = + \infty$.  In particular, $\mathbb{S}$ is bounded.
\item 
\label{item:spec a in S}
The spectrum of $D_a$ lies inside $\mathbb{S}$, i.e.\ 
\bels{spec a in S}{
\spec (D_a) \subset {\mathbb{S}}\,.
}
\item \label{relation beta and lambda}  The sign of $\beta$ satisfies 
\bels{sign beta = sign lambda}{
\sign{\beta(\zeta)} = \rm{sign} \Big(1-\lambda_{\rm{PF}} \Big(SD_{\abs{a-\zeta}}^{-2}\Big)\Big) \,, \qquad \zeta \in \C\,.
}
\item 
\label{characterisation beta>0}
For any $\zeta \in \C$ with $\beta(\zeta)>0$ the operator $B_\zeta$ is invertible. Furthermore, all such $\zeta$ are characterised by
\bels{eq:characterisation beta>0}{
\{\zeta \in \C: \beta(\zeta) >0\} = \{\zeta \in \C: \dist(0, \supp \rho_\zeta)>0\}=\C \setminus \mathbb S_0\,.
}
\item 
\label{characterisation of S}
The set $\mathbb{S}$ is characterised by having a positive singular value density at the origin, i.e.\  
\bels{eq:characterisation of S in terms of rho}{
\mathbb{S}= \{\zeta  \in \C : \rho_\zeta(0)>0\}\,.
}
\end{enumerate}
\end{proposition}

Before proving Proposition~\ref{prp:S characterisation}, we state a corollary, which follows directly from Proposition~\ref{prp:Analyticity in the bulk} and Proposition~\ref{prp:S characterisation} \ref{characterisation of S}. 

\begin{corollary}[Existence and uniqueness for \eqref{eq:V_equations at eta=0}] \label{cor:uniqueness_v_equations_eta=0} 
Let $a \in \mathcal B$ and $s$ satisfy \ref{assum:flatness} and \ref{assum:Gamma}.
Then for each $\zeta \in \mathbb S$, the relations \eqref{eq:V_equations at eta=0} have a unique solution $(v_1, v_2) \in \mathcal B_+ \times \mathcal B_+$. 
\end{corollary}

\begin{proof}
Proof of \ref{continuity of beta}: 
The continuity of $\zeta \mapsto \beta(\zeta)=\beta$ with $B = B_\zeta$ is a consequence of   the bound  
\[
\absbb{\inf_{x \in \cal B_+} \sup_{y \in \cal B_+}\frac{\scalar{x}{(B+D_w)y}}{\scalar{x}{y}} - \beta }
\le \sup_{x \in \cal B_+} \sup_{y \in \cal B_+}\frac{\scalar{x}{D_{\abs{w}}y}}{\scalar{x}{y}} \le \norm{w}_\infty
\]
for any real valued $w \in \cal B$  since $B-B_{\zeta + \omega} = D_w$ with $w = a-\zeta|^2-|a-\zeta + \omega|^2$  and $\norm{w}_\infty \to 0$ for  $\omega \in \C$ with  $|\omega| \to 0$.  This bound follows from 
\[
\frac{\scalar{x}{(B+D_w)y}}{\scalar{x}{y}}  \le \frac{\scalar{x}{By}}{\scalar{x}{y}} +\sup_{\wt x \in \cal B_+} \sup_{\wt y \in \cal B_+}\frac{\scalar{\wt x}{D_{|w|}\wt y}}{\scalar{\wt x}{\wt y}}  \le   \frac{\scalar{x}{By}}{\scalar{x}{y}} + \norm{w}_\infty
\]
for every $x,y \in \cal B_+$, as well as 
\[
\frac{\scalar{x}{(B+D_w)y}}{\scalar{x}{y}}  \ge  \frac{\scalar{x}{By}}{\scalar{x}{y}} - \sup_{\wt x \in \cal B_+} \sup_{\wt y \in \cal B_+}\frac{\scalar{\wt x}{D_{|w|}\wt y}}{\scalar{\wt x}{\wt y}}  \ge    \frac{\scalar{x}{By}}{\scalar{x}{y}} - \norm{w}_\infty,
\]
and then taking the supremum over $y \in \cal B_+$ and the infimum over $x \in \cal B_+$ in both inequalities. 
The statement $\beta(\zeta ) \to + \infty$ as $\zeta \to \infty$ is obvious.

Before we start with the proof of other individual statements of the proposition, we show that $\mathbb{S}$ can  be classified in terms of the Perron-Frobenius eigenvalue of  $SD_{\abs{a-\zeta}}^{-2}$ in the  sense that
\bels{beta lambda same sign}{
\mathbb{S} = \{\zeta \in \C :\lambda(\zeta)>1\} \,,
}
where we introduced $\lambda : \C \to [0,\infty]$ as the limit of a strictly increasing  sequence via
\[
\lambda(\zeta):= \lim_{\eps\downarrow 0} \lambda_\eps(\zeta)\,, \qquad \lambda_\eps(\zeta):= \lambda_{\rm{PF}} \pb{S(\eps+D_{ \abs{a-\zeta}^2})^{-1}}.
\]
 We recall that $\lambda_{\rm{PF}}(T)$ denotes the Perron-Frobenius eigenvalue of $T$ and refer to the comments before the statement of the proposition for its definition. 

To show \eqref{beta lambda same sign} let $\eps \in (0,1)$, $D:=D_{\abs{\zeta-a}^2 } $, $\lambda_\eps=\lambda_\eps(\zeta)$ 
 and $C>0$ such that $1+\abs{\zeta-a}^2 \le C$. 
For $\zeta \in \C$ with $\beta(\zeta) \ge 0$ we get 
\bes{
\beta + \eps \le C\inf_{x\in \mathcal B_+}\sup_{y\in \mathcal B_+} \frac{\scalar{x}{(\eps+B) y}}{\scalar{x}{(\eps+D) y}} = C\pbb{1-  \sup_{x\in \mathcal B_+}\inf_{y\in \mathcal B_+}\frac{\scalar{x}{S(\eps+D)^{-1}y}}{\scalar{x}{y}}}=C\pb{1-\lambda_\eps} ,
}
where we used $\eps + \abs{\zeta-a}^2 \le C$ in the first inequality, 
and conclude 
\bels{beta lambda inequality 1}{
\beta \le C(1-\lambda) \qquad \text{ if } \beta \ge 0\,.
}
For $\zeta \in \C$ with $\beta(\zeta) < 0$ we use 
\[
-\beta-\eps =  \sup_{x\in \mathcal B_+}\inf_{y\in \mathcal B_+} \frac{-\scalar{x}{(\eps +B) y}}{\scalar{x}{ y}}
\]
for sufficiently small $\eps>0$ and  find analogously that 
\bels{beta lambda inequality 2}{
\beta\ge 
C\pb{1-\lambda}  \qquad  \text{ if } \beta < 0\,.
}
From \eqref{beta lambda inequality 1} and \eqref{beta lambda inequality 2} we conclude \eqref{beta lambda same sign}.

We also  show that 
\bels{spec a in closure S}{
\spec (D_a) \subset \{\zeta \in \C: \beta(\zeta) \le 0\}.
} We will improve this to \eqref{spec a in S} below. 
Let $\zeta \in \spec (D_a)$. Then $\essinf \abs{\zeta-a} =0$. Thus, for any $\eps>0$ we find $x \in \mathcal B\setminus \{ 0\}$ with $x \geq 0$ such that $ \abs{\zeta-a}^2 x \le \eps x$. 
 In the definition of $\beta$ from \eqref{def of beta and B and S} we can take the supremum over all $x \in \ol{\cal B}_+$, Thus, we get  
\[
\beta \le \eps-  \inf_{y\in \mathcal B_+} \frac{\scalar{x}{Sy}}{\scalar{x}{y}} \le \eps\,.
\]
Since $\eps>0$ was arbitrarily small, we conclude $\beta \le 0$.  

Proof of \ref{characterisation beta>0}: 
Let $\zeta \in \C$ such that $\beta(\zeta)=\beta >0$. Then \eqref{beta lambda inequality 1} implies $\lambda_{\rm{PF}} (SD^{-1}) <1$ with $D=D_{\abs{\zeta-a}^2 }$. Here, $D$ is invertible because $\essinf \abs{\zeta-a} >0$ due to \eqref{spec a in closure S}.
In particular,  $B= (1- SD^{-1})D$ is invertible.  

Now we show $\dist(0, \supp \rho_\zeta)>0$ to see one inclusion in the characterisation \eqref{eq:characterisation beta>0}. 
The Dyson equation in the matrix representation,  \eqref{general MDE} is solved by 
\bels{def of M0}{
M_0:= \mtwo{0 & (\overline{a - \zeta})^{-1}}{(a-\zeta)^{-1} &0 }
}
at $w=0$. Furthermore, the associated stability operator (cf. \eqref{eq:def_stability_operator})
\bels{MDE stability operator outside}{
\cal{L}_0\colon   \cal B^{2\times 2} \to \cal B^{2\times 2},  \quad R \mapsto M_0^{-1}RM_0^{-1} - \Sigma R = \mtwo{\abs{a-\zeta}^2r_{22}- Sr_{22} & (a-\zeta)^2r_{21}}{(\overline{a-\zeta})^{2}r_{12} & \abs{a-\zeta}^2r_{11}- S^{ *}r_{11} }
}
is invertible because $B_\zeta$ is invertible  and $\essinf \abs{a-\zeta} >0$.  Therefore \eqref{general MDE} can be uniquely solved for sufficiently small $w$ as an analytic function $w \mapsto M(\zeta , w)$ with $M(\zeta,0)=M_0$ and we get
\[
\cal{L}_0[\partial_w M(\zeta,w)|_{w=0} ]= 1\quad \text{and} \quad 
\partial_\eta M(\zeta,  \ii \eta) |_{\eta =0} = \ii\1\cal{L}_0^{-1} [1]
\]
In particular, $\im M(\zeta, \ii \eta)$ is positive definite for sufficiently small $\eta >0$ because $\cal{L}_0^{-1}$ is positivity preserving, as can be seen from a Neumann series expansion using that  $\lambda_{\rm{PF}}(SD^{-1}) < 1$. Therefore $M(\zeta, \ii \eta)$ is the unique solution of \eqref{general MDE} with $m_{11}(\zeta,\ii \eta) =\ii v_1(\zeta, \eta)$ and $m_{22}(\zeta,\ii \eta)=\ii v_2(\zeta, \eta)$. Since $m_{11}|_{\eta=0} =m_{22}|_{\eta=0}=0$ we conclude that $\rho_\zeta(0) = \avg{v_1(\zeta,0)}=0$. The invertibility of $\cal{L}_0$ also implies  analyticity of $M(\zeta, w)$ in $w$ in a small neighbourhood of zero. Thus, $\rho_\zeta ( [-\eps,\eps])=0$ for $\eps>0$ sufficiently small 
 and  $\dist(0, \supp \rho_\zeta )>0$. 

To see the other inclusion in \eqref{eq:characterisation beta>0}, let $\zeta \in \C$ be such that $\dist(0, \supp \rho_\zeta )\ge \delta$ for some $\delta>0$. 
From \cite[Lemma~D.1 (iv)]{AEK_Shape} we know that $M=M(\zeta, \ii \eta)$ is locally a real analytic function of $\eta$ with an expansion $M= M_0 + \ii\1\eta \2M_1 + O(\eta^2)$, where $M_0=M_0^*$.  
Taking the imaginary part of \eqref{general MDE} at $w=\ii \eta$, dividing  both sides by $\eta$  shows that
\bels{Keta equation}{
(M^*)^{-1}K_\eta M^{-1} = 1 + \Sigma[K_\eta]\,,
}
where $K_\eta = \frac{1}{\eta} \im M= \re M_1 + O(\eta)$. In particular, $K_0:= \lim_{\eta \downarrow 0}K_\eta $ exists and since $\im M(\zeta, \ii \eta) \sim_\delta \eta$ by Lemma~\ref{lem:v_away_support}  we get $K_0 \sim_\delta 1$. 
Evaluating \eqref{Keta equation} at $\eta=0$ yields
\bels{K0 equation}{
M_0(\Sigma K_0) M_0 = K_0 -1\,.
}
Taking the scalar product of \eqref{K0 equation} with the left Perron-Frobenius eigenvector of $R \mapsto M_0 (\Sigma R ) M_0$ and using $K_0 \sim_\delta 1$ we see that 
  $\lambda_{\rm{PF}}(R \mapsto M_0 (\Sigma R ) M_0) <1$. This
 is equivalent to $\lambda:=\lambda_{\rm{PF}}(SD^{-1})<1$ with $D:=D_{\abs{a-\zeta}^2}$. Now let $u \in \mathcal B_+$ be the Perron-Frobenius eigenvector of $SD^{-1}$. Since $\eps:=\essinf \abs{a-\zeta}>0$ we get with $y_0:=D^{-1}u$ that
\[
(D -S)y_0=(1-\lambda)Dy_0 >0\,.
\]
Thus, 
\[
\beta = \inf_{x >0} \sup_{y >0} \frac{\scalar{x}{B y}}{\scalar{x}{y}} \ge \inf_{x >0} \frac{\scalar{x}{B y_0}}{\scalar{x}{y_0}} \ge (1-\lambda)\1\eps^2>0\,.
\]
This finishes the proof of \eqref{eq:characterisation beta>0},  i.e.\ of \ref{characterisation beta>0}.   

Proof of \ref{relation beta and lambda}: 
We have now collected enough information to improve \eqref{beta lambda same sign} to \eqref{sign beta = sign lambda}.
Indeed, by \eqref{beta lambda inequality 1} and \eqref{beta lambda inequality 2} it remains to show that $\lambda<1$ implies $\beta >0$. Due to \eqref{beta lambda same sign} we already know $\beta \ge 0$ in  case $\lambda<1$.  
Now let  $\beta=0$ and $\lambda \le 1$. Then we show that $\lambda=1$.
Indeed by the characterisation \eqref{eq:characterisation beta>0}  we have $0 \in \supp \rho_\zeta$. Now we consider the identity 
\[
B_\zeta v_2 
=  \eta-v_2(\eta + S^*v_1)(\eta +Sv_2)
\]
which follows from \eqref{eq:v2}. For some $\eps>0$ we add $\eps\1 v_2$ to both sides and apply  the inverse of $\eps+D$ with $D=D_{\abs{a-\zeta}^2}$. Then we take the scalar product with the right Perron-Frobenius eigenvector $x_\eps \in \mathcal B_+$  of $S^*(\eps+D)^{-1}$ corresponding to its Perron-Frobenius eigenvalue $\lambda_\eps>0$.  Note that the Perron-Frobenius eigenvalues of $S^*(\eps+D)^{-1}$, $(\eps+D)^{-1}S$ and $S(\eps+D)^{-1}$ all coincide. 
Thus we get 
\bels{scalar with x eps}{
(1-\lambda_\eps) \avg{x_\eps v_2} = \eta \1 \avg{(\eps + D)^{-1}x_\eps} -\avg{x_\eps (\eps+D)^{-1}(v_2(\eta + S^*v_1)(\eta +Sv_2)-\eps\1 v_2)}\,.
}
From \cite[Corollary~D.2]{AEK_Shape} and $\avg{v_i} \sim v_i$  by Lemma~\ref{lem:v_scaling}\ref{item:v_i_sim_avg_v_i}  we see that $\eta/\avg{v_i} \to 0$ for $\eta \downarrow 0$. Thus, dividing \eqref{scalar with x eps} by $\avg{v_2}$, taking the limit $\eta\downarrow 0$ and using   \eqref{eq:Saveraging}  reveals
\[
( 1-\lambda_\eps)\avg{x_\eps }\sim ( 1-\lambda_\eps)\avg{x_\eps k} =  \eps \avg{x_\eps(\eps+D)^{-1}k}\sim \eps \avg{x_\eps(\eps+D)^{-1}S1}=
 \eps \avg{S^*(\eps+D)^{-1}x_\eps}=\eps \1 \lambda_\eps\avg{x_\eps }\,,
\]
where $k:= \limsup_{\eta \downarrow 0} \frac{v_2}{\avg{v_2}}\sim 1$. Letting $\eps \downarrow 0$ shows $\lambda=1$.
Thus,  \eqref{sign beta = sign lambda} is proven. 

Proof of \ref{characterisation of S}: By \eqref{eq:characterisation beta>0} we know that $\rho_\zeta(0)>0$ implies $\beta(\zeta) \le 0 $. 
 Thus, it suffices to show that for $\zeta \in \C$ with $\beta(\zeta) \le 0$ we get $\beta(\zeta) =0$ if and only if $ \rho_\zeta(0)=0$. 
 Now let $\beta=\beta(\zeta) \le 0$.
 As above, we consider the identity \eqref{scalar with x eps}. 
First, suppose $\rho_\zeta(0)>0$, i.e.\ we can analytically extend $v$ to $\eta=0$ by Proposition~\ref{prp:Analyticity in the bulk} and have $v|_{\eta=0}>0$.   Then in the limit $\eta \downarrow 0$ we find 
\[
(\lambda_\eps-1) \avg{x_\eps v_2} = \avg{x_\eps (\eps+D)^{-1}(v_2 (S^*v_1Sv_2-\eps))}
\]
Using $v_2 \sim \avg{v_2}\sim \rho_\zeta(0)$ for small enough $\eps>0$ the right hand side satisfies 
\[
\avg{x_\eps (\eps+D)^{-1}(v_2 (S^*v_1Sv_2-\eps))}\sim \rho_\zeta(0)^3 \avg{x_\eps (\eps+D)^{-1}S1} \sim \lambda_\eps\rho_\zeta(0)^3 \avg{x_\eps}.
\]
Since $\avg{x_\eps v_2}  \sim  \rho_\zeta(0) \avg{x_\eps}$ we infer $\lambda_\eps-1 \sim  \lambda_\eps\2\rho_\zeta(0)^2$. Thus, $\lambda>1$ and by \eqref{sign beta = sign lambda} therefore $\beta < 0$.

Conversely, let $v|_{\eta=0}=0$. Then we know from Lemma~\ref{lem:v_scaling} \ref{item:zeta_minus_a_near_edge} that $\delta:=\essinf \abs{a-\zeta}>0$. Since $\beta(\zeta)\le 0$ the characterisation \eqref{eq:characterisation beta>0} implies $0 \in \supp \rho_\zeta$ and by \eqref{sign beta = sign lambda} we have $\lambda \ge 1$. Since  $\eta/\avg{v} \to 0$ for $\eta \downarrow 0$ by \cite[Corollary~D.2]{AEK_Shape} we get, dividing  \eqref{scalar with x eps} by $\avg{v_2}$ and taking the limit $\eta\downarrow 0$,  the scaling behaviour
\[
1-\lambda_\eps \sim \eps \1 \lambda_\eps\,.
\]
This implies $\lambda_\eps \le 1 $, thus $\lambda=1$, and completes the proof of \ref{characterisation of S}.  

Proof of \ref{item:spec a in S}: Let $\zeta \in \spec(D_a)$. By \eqref{spec a in closure S} we know $\beta(\zeta) \le 0$. Suppose $\beta(\zeta)=0$. Then \eqref{eq:characterisation of S in terms of rho} would imply $\rho_\zeta(0)=0$. However, this contradicts $\zeta \in \spec(D_a)$ because of    Lemma~\ref{lem:v_scaling}  \ref{item:zeta_minus_a_near_edge} and the definition of $\rho_\zeta(0)$ in \eqref{eq:def_rho_zeta_0}.   
This finishes the proof of the proposition.
\end{proof}

\subsection{Edge expansion of $v_1$ and $v_2$}

In this section we expand the solution $(v_1,v_2)$ of \eqref{eq:V_equations} around any $\zeta_0 \in \C$ 
with $\beta(\zeta_0) =0$. 
We will see later in Lemma~\ref{lem:beta_zero_and_critical_point} below that the set of these points coincides with $\partial \mathbb{S}$, 
which will turn out to be the regular edge and singular points of the Brown measure $\sigma$. 
Therefore we consider in this section a fixed $\zeta_0 \in \C$ with $\beta(\zeta_0)=0$.
The expansion of $v_1,v_2$ around  $\zeta_0$ is based on  analytic perturbation theory for $\beta$. Throughout this section we will always assume $\abs{\zeta-\zeta_0} + \eta \le c$ for some sufficiently small positive constant $c\sim 1$, i.e.\ we assume that $(\zeta,\eta)$ lies within a small neighbourhood of $(\zeta_0,0)$. 
We assume \ref{assum:flatness} and \ref{assum:Gamma} throughout the remainder of this section. 

To shorten notation, we denote $v_i =v_i(\zeta,\eta)$. 
The identities 
\bels{edge identities}{
&
B_\zeta \1v_2 =  \eta-v_2(\eta + S^*v_1)(\eta +Sv_2),  \qquad  
B_\zeta^* \1v_1 =  \eta-v_1(\eta + S^*v_1)(\eta +Sv_2), 
}
 which follow from \eqref{eq:v1} and \eqref{eq:v2}, respectively,  
are used  to expand $v_1$ and $v_2$ in a neighbourhood of $\zeta_0$. 
We will see in Corollary~\ref{crl:beta as eigenvalue of B} below that   the function $\zeta \mapsto \beta(\zeta)$, used to define $\mathbb{S}$ in \eqref{def of beta and B and S}, coincides locally around   $\zeta_0 $ with the isolated non-degenerate eigenvalue of $B_\zeta$ closest to zero.  
We denote by $b=b_\zeta \in \cal B_+$ and $\ell=\ell_\zeta  \in \cal B_+$   the right and left eigenvectors of $B=B_\zeta$, 
corresponding to the eigenvalue $\beta =\beta(\zeta)$ with normalisation $\avg{b} = \avg{\ell}=1$, i.e.\  
\bels{B eigenvectors}{
 B b = \beta \2b, \qquad B^* \ell = \beta \1\ell\,. 
 }
 The existence and uniqueness of $b$ and $\ell$ is a consequence of analytic perturbation theory and Lemma~\ref{lmm:Properties of B} below. This lemma also implies that $\zeta \mapsto b_\zeta$ and $\zeta \mapsto \ell_\zeta$ are real analytic functions. 
The main result of this section is the following proposition.

\begin{proposition}\label{prp:v edge expansion} Let $s$ and $a$ satisfy \ref{assum:flatness} and \ref{assum:Gamma}. Furthermore, let $\zeta_0 \in \C$ such that $\beta(\zeta_0) = 0$.
Then there is an open neighbourhood 
$U\subset \C\times \R^2$  of $( \zeta_0,0,0)$, an open neighbourhood $V\subset\C\times  \R$ of $(\zeta_0,0)$
and  real analytic functions $\wt{w}_1,\wt{w}_2\colon U \to \cal{B}$ such that 
\[
v_1(\zeta,\eta) = \vartheta(\zeta,\eta)\ell_\zeta+ \wt{w}_i(\zeta, \eta, \vartheta(\zeta,\eta))\,, \qquad 
v_2(\zeta,\eta) = \vartheta(\zeta,\eta)b_\zeta+ \wt{w}_i(\zeta, \eta, \vartheta(\zeta,\eta))
\]
for $(\zeta,\eta) \in V$ and $\eta>0$. Furthermore, 
 $\vartheta=\vartheta(\zeta,\eta)$ satisfies 
\bels{cubic for vartheta}{
\vartheta^3\avgb{\ell b (S^*\ell)(Sb) }+\beta\1\vartheta\1 \avg{\ell \1b}-\eta=  g(\zeta, \eta,\vartheta)\,,
}
for all $(\zeta,\eta) \in V$ 
where  $\ell=\ell_\zeta$, $b=b_\zeta, \beta =\beta(\zeta)$ and $g\colon U \to \R$  is a real analytic function, such that
\[
g(\zeta, \eta,x)=O( \abs{\eta \1x}^2 +\abs{x}^5 )\,, \qquad (\zeta, \eta,x) \in U\,.
\]
\end{proposition}
The proof of  Proposition~\ref{prp:v edge expansion}  is the content of the remainder of this section and will be summarised at its end. 
We remark that as a solution to the cubic equation \eqref{cubic for vartheta} the quantity $\vartheta$ and with it $v_1,v_2$ are not analytic at $\zeta=\zeta_0$ and $\eta=0$.

The following lemma collects spectral properties of $B_{\zeta_0}$. These properties yield corresponding properties of $B_\zeta$ for sufficiently small $\abs{\zeta-\zeta_0}$, using analytic perturbation theory. We will use this idea throughout the remainder of this section after the statement  of Lemma~\ref{lmm:Properties of B}. 
\begin{lemma}[Properties of $B$]
\label{lmm:Properties of B}
Let $\zeta_0 \in \C$ with $\beta(\zeta_0)=0$ and $B_0:=B_{\zeta_0}$. Then there is a constant $\eps>0$ with $\eps\sim 1$ such that  
\bels{B0 resolvent control}{
\sup\cB{\norm{(B_0-\omega)^{-1}}_{\#}+\norm{(B_0^*-\omega)^{-1}}_{\#}: \omega \in \D_{2\eps}\setminus \D_{\eps}} \lesssim 1
}
for $\#=2,\infty$. Here $\D_{\eps}$ contains a single isolated non-degenerate eigenvalue $0$ of $B_0$, i.e,
\bels{Nondegenerate beta}{
\D_\eps \cap \spec(B_0) =\{0\}\,, \qquad \dim \rm{ker}B_0^2 =1\,.
}
Moreover, the  right and left eigenvectors, $b_0 \in \cal{B_+}$ and $\ell_0 \in \cal{B_+}$,  corresponding to this eigenvalue with normalisation $\avg{b_0} =\avg{\ell_0}=1$ satisfy the bounds $\ell_0\sim b_0 \sim 1$. Furthermore, if 
\[
P_0:= \frac{\avg{\ell_0\,\cdot \,}}{\avg{\ell_0\1 b_0}}b_0\,, \qquad Q_0:=1-P_0
\]
denote the associated spectral projections then 
\bels{Q0B0 inverse bound}{
\norm{B_0^{-1}Q_0}_\#+\norm{(B_0^{*})^{-1}Q_0^*}_\# \lesssim 1\,.
}
\end{lemma}

\begin{proof}
Here we present the  proofs of the bounds \eqref{B0 resolvent control} and \eqref{Q0B0 inverse bound} for $B_0$. The corresponding bounds for $B_0^*$ follow analogously. 
From Proposition~\ref{prp:S characterisation} \ref{item:spec a in S} and since $\mathbb{S}$ is bounded  we know that $\abs{a-\zeta_0} \sim 1$.  
Thus,  
$b_0$ is the right eigenvector of $D^{-1}S$ with eigenvalue $1$ and $\ell_0$ is the right eigenvector of $D^{-1}S^*$ with eigenvalue $1$, where $D:=D_{\abs{a-\zeta_0}^2}$. In particular,  $b_0,\ell_0 \in \cal{B}_+$ by the Krein-Rutman theorem and the geometric multiplicity of the eigenvalue $0$ of $B_0$ is $1$. 
Furthermore, the non-degeneracy of the eigenvalue $0$ is a consequence of $b_0,\ell_0 \in \cal{B}_+$. Indeed, suppose we had $\dim \rm{ker}B^2>1$. Then there would be a generalised eigenvector $x$ with $Bx=b_0$ and 
$
\avg{\ell_0\1 b_0} = \avg{\ell_0 \1 B_0x} =0
$ which contradicts $\ell_0>0$ and $b_0>0$.
 This proves \eqref{Nondegenerate beta}, which together with \eqref{B0 resolvent control} implies 
 \eqref{Q0B0 inverse bound}. The relation $b_0 \sim \ell_0 \sim 1$ is a direct consequence of $\abs{a-\zeta_0} \sim 1$ and Assumption~\ref{assum:flatness}. 

We are left with proving \eqref{B0 resolvent control}. 
Instead of controlling the resolvent of $B_0$, it suffices to bound the inverse of $ 1- SD^{-1} - \omega D^{-1}$ 
because
\bels{resolvent of B in terms of S}{
\frac{1}{B_0-\omega} =\frac{1}{D}\pbb{\frac{1}{1-SD^{-1}-\omega \1D^{-1}}\wt{Q}_\omega
+\frac{1}{1-SD^{-1}-\omega \1D^{-1}}\wt{P}_\omega}\,,
}
where $\wt{P}_\omega$ and $\wt{Q}_\omega:=1-\wt{P}_\omega$ are the analytic spectral projections associated with $SD^{-1}-\omega D^{-1}$ such that  
\[
\wt{P}_0 = \frac{\avg{\ell_0\,\cdot \,}}{\avg{\ell_0 D b_0}}Db_0\,.
\]
Analytic perturbation theory can be applied to $SD^{-1}$ because of Lemma~\ref{lmm:resolvent bound for S}, which shows that the resolvent of the operator  $SD^{-1}$ is bounded in annulus around its isolated eigenvalue $1$. 
Consequently, the first summand in \eqref{resolvent of B in terms of S} is bounded for sufficiently small $\abs{\omega}$. The second summand admits the expansion
\[
\frac{1}{1-SD^{-1}-\omega \1D^{-1}}\wt{P}_\omega
= \frac{1}{\wt{\beta}(\omega)}\wt{P}_\omega\,,
\qquad 
\wt{\beta}(\omega) = 
-\omega\2\frac{\avg{\ell_0\1b_0}}{\avg{\ell_0\1 Db_0}} + O(\abs{\omega}^2)
\]
by standard analytic perturbation formulas, see e.g. \cite[Lemma~C.1]{AEK_Shape}.
Therefore the second summand is bounded for  $\omega \in \C \setminus \D_\eps$ for sufficiently small $\eps$.
\end{proof}

\begin{corollary}\label{crl:beta as eigenvalue of B}
Let $\zeta_0 \in \C$ with $\beta(\zeta_0)=0$. Then $0 \in \spec(B_{\zeta_0})$, $\essinf\abs{a-\zeta_0}>0$ and $\lambda_{\rm{PF}} (SD_{\abs{a-\zeta_0}^2}^{-1}) =1$. Furthermore, there is $\eps>0$ such that  $ \beta(\zeta )$ is an isolated non-degenerate eigenvalue of $B_{\zeta }$ for all $\zeta \in \zeta_0+\mathbb{D}_\eps $. In particular $\zeta_0 + \mathbb{D}_\eps \ni \zeta \mapsto \beta(\zeta)$ is real analytic and has the expansion  
\bels{beta expansion at edge}{
\beta(\zeta) &= -2 \rm{Re} \sbb{\frac{\avg{\ell_0\1b_0\1 (a-\zeta_0)}}{\avg{\ell_0\1b_0}}\ol{(\zeta -\zeta_0)}} 
+\pbb{1- 2 \re \sbb{\frac{\avg{\ell_0  \ol{(a-\zeta_0)} B_0^{-1}Q_0[b_0 (a-\zeta_0)]}}{\avg{\ell_0\2b_0}}} } \abs{\zeta-\zeta_0}^2
\\
&\quad \qquad -
2 \re \sbb{\frac{\avg{\ell_0  (a-\zeta_0)B_0^{-1}Q_0[b_{0} (a-\zeta_0)]}}{\avg{\ell_0\2b_0}}\ol{(\zeta-\zeta_0)}^2}+ O(\abs{\zeta-\zeta_0}^3)\,,
}
which implies the formulas
\bels{derivatives of beta}{
\partial_\zeta\beta(\zeta_0)=-\frac{\avg{\ell_0\1b_0\1 \ol{(a-\zeta_0)}}}{\avg{\ell_0\1b_0}} \,, \qquad 
\partial_\zeta \partial_{\ol{\zeta}}\beta(\zeta_0)=
1- 2 \re \sbb{\frac{\avg{\ell_0  \ol{(a-\zeta_0)} B_0^{-1}Q_0[b_0 (a-\zeta_0)]}}{\avg{\ell_0\2b_0}}}
}
for the derivatives of $\beta$ at $\zeta=\zeta_0$. 
\end{corollary}

\begin{proof}
Let $\zeta_0\in \C$ be such that $\beta(\zeta_0)=0$. By Lemma~\ref{lmm:Properties of B} we have $0 \in \spec(B_{\zeta_0})$ and by  Proposition~\ref{prp:S characterisation} \ref{item:spec a in S} we get $\essinf \abs{a-\zeta_0}>0$. The fact that $\lambda_{\rm{PF}} (SD_{\abs{\zeta_0-a}^2}^{-1})=1$ was shown in \eqref{sign beta = sign lambda}. 

Now we show that $\beta(\zeta)$ is an eigenvalue of $ B_{\zeta}$ for sufficiently small $\abs{\zeta-\zeta_0}$. Using analytic perturbation theory, let $b(\zeta)$ and $\ell(\zeta)$ be the right and left eigenvectors of $B_{\zeta}$ corresponding to the isolated non-degenerate eigenvalue $\wt{\beta}(\zeta)$ with $\wt{\beta}(\zeta_0)=0$ that depends real analytically on $\zeta$. 
As $\wt{\beta}(\zeta_0)$ is a real isolated eigenvalue and $B_{\zeta_0}$ as well as $B_{\zeta} - B_{\zeta_0}$ are invariant under complex conjugation, $\wt{\beta}(\zeta)$, $b(\zeta)$ and $\ell(\zeta)$ are also real. 
Since $\ell(\zeta_0) \sim b(\zeta_0) \sim 1$ we have $b(\zeta)$, $\ell(\zeta)\in \cal{B}_+$ for sufficiently small $\abs{\zeta-\zeta_0}$. Therefore 
\[
\wt{\beta}(\zeta)=\inf_{x >0} \frac{\scalar{x}{B_{\zeta }b(\zeta)}}{\scalar{x}{b(\zeta)}}\le \beta(\zeta) \le  
\sup_{y >0} \frac{\scalar{\ell(\zeta)}{B_{\zeta }y}}{\scalar{\ell(\zeta)}{y}}=\wt{\beta}(\zeta)\,,
\]
which proves $\wt{\beta}=\beta$.

The expansion \eqref{beta expansion at edge} is now a direct consequence of analytic perturbation theory, as we see
e.g.\ by using \cite[Lemma~C.1]{AEK_Shape} 
with $B = B_0 + E$ and $E = D_{\abs{a-\zeta}^2} - D_{\abs{a-\zeta_0}^2} = D_{\abs{\zeta- \zeta_0}^2 - 2 \Re( \overline{(a - \zeta_0)} ( \zeta - \zeta_0))}$. 
\end{proof}

Due to analytic perturbation theory with $\zeta$ in a small neighbourhood of $\zeta_0$ and 
by Lemma~\ref{lmm:Properties of B} we have $b \sim \ell \sim 1$. 
We split  $v_1$ and $v_2$  according to the spectral decompositions  of $B^*$ and  $B$, namely 
\bels{vi split}{ 
v_1 = \vartheta_1 \1\ell + \wt{v}_1, \qquad v_2 = \vartheta_2 \1b + \wt{v}_2 
}
with the contributions $\vartheta_i =\vartheta_i(\zeta,\eta)$ to the eigendirections  $\ell$ and $b$ of $B^*$ and $B$ as well as  their complements  $\wt{v}_i =\wt{v}_i(\zeta,\eta)$ given as 
\bels{def:vartheta_i}{
\vartheta_1 :=\frac{ \avg{b\1v_1}}{\avg{\ell\1 b}}\,, \qquad \vartheta_2 :=\frac{ \avg{\ell\1v_2}}{\avg{\ell\1 b}}\,, 
\qquad \wt{v}_1:=Q^*v_1\,, \qquad \wt{v}_2:=Qv_2\,,\qquad Q:=1-\frac{\avg{\ell\,\cdot \,}}{\avg{\ell\1 b}}b\,.
}
To quantify the error terms we introduce 
\[
\alpha:= \norm{v_1}_\infty + \norm{v_2}_\infty\,.
\]
Projecting the identities \eqref{edge identities} with $Q$  and $Q^*$, respectively,  leads to 
\bels{tilde v equations}{
B_\zeta \1\wt{v}_2 
=  O(\eta + \alpha^3), \qquad 
B_\zeta^* \1\wt{v}_1 
=  O(\eta + \alpha^3)\,.
}
Using  $\norm{B^{-1}Q}_\infty\lesssim 1$, a consequence of \eqref{Q0B0 inverse bound}  and analytic perturbation theory , we find 
\bels{bound on tilde alpha}{
\norm{\wt{v}_1}_\infty+\norm{\wt{v}_2}_\infty =O(\eta  +\alpha^3)\,.
}
Because of $\avg{v_1}=\avg{v_2}$,  i.e.\ by \eqref{eq:avg_v1_equals_avg_v2},  \eqref{vi split}  and the normalisation $\avg{b}=\avg{\ell}=1$, \eqref{bound on tilde alpha} implies 
\bels{vartheta identity}{
\vartheta_1 =\vartheta_2+O(\eta  +\alpha^3)\,.
}
Inserting the decomposition \eqref{vi split} into \eqref{edge identities} and using \eqref{bound on tilde alpha}, as well as \eqref{vartheta identity}, leads to 
\bes{
&\beta\1\vartheta_2\1 b+B\wt{v}_2= \eta -\vartheta^3b (Sb)(S^*\ell) + O(\eta \alpha^2 +\alpha^5 )\,,
\\
&\beta\1\vartheta_1\1 \ell+B^*\wt{v}_1= \eta -\vartheta^3\ell (Sb)(S^*\ell) + O(\eta \alpha^2 +\alpha^5 )\,,
}
where we set $\vartheta:=\frac{1}{2}(\vartheta_1 + \vartheta_2)$. 
Now we average the first equation against $\ell$ and the second equation against $b$, use $\avg{b}=\avg{\ell}=1$  and then take the  arithmetic mean of the resulting equations to find
\bels{vartheta cubic}{
\vartheta^3\avgb{\ell b (S^*\ell)(Sb) }+\beta\1\vartheta\1 \avg{\ell \1b}-\eta=  O(\eta \1\alpha^2 +\alpha^5 )\,.
}
From this approximate cubic equation we conclude the scaling behaviours 
\begin{equation} \label{eq:scaling_alpha_in_terms_of_lambda_and_eta} 
\alpha \sim \vartheta \sim \sqrt{\max\{0, -\beta\}}+\frac{\eta}{\eta^{2/3} + \abs{\beta}} \,, \qquad \norm{\wt{v}_1}_\infty+\norm{\wt{v}_2}_\infty\sim \eta + \max\{0, -\beta\}^{3/2},
\end{equation} 
in the regime of sufficiently small $\alpha$, 
where we used $\vartheta \ge 0$ and $\vartheta > 0$ for $\eta >0$ to choose the correct branch of the solution. The corresponding argument is summarised in Lemma~\ref{lmm:scaling of cubic} in the appendix. 

To apply this lemma we absorb the $O(\alpha^5)$-term on the right hand side of \eqref{vartheta cubic} into the cubic term in $\vartheta$ on the left hand side, i.e.\ we write $O(\alpha^5)= \gamma\2 \vartheta^3 $ for some $\gamma= O(\alpha^2)$, which we absorb into the coefficient of the $\vartheta^3$-term. Such rewriting is possible since  $\alpha= O(\vartheta)$ in the regime where $\alpha$ is sufficiently small. This holds because $\vartheta \gtrsim \eta $ and   $\alpha = O(\vartheta + \eta + \alpha^3)$ by \eqref{vi split}, \eqref{vartheta identity} and \eqref{bound on tilde alpha}. 
Now we see that $\alpha$ is indeed small for $(\zeta, \eta)$ in a neighbourhood of $(\zeta_0,0)$. 
Due to the characterisation of $\mathbb{S}$ in \eqref{eq:characterisation of S in terms of rho} we have $\lim_{\eta\downarrow 0} \alpha |_{\zeta = \zeta _0}=0$. With $\beta(\zeta_0)=0$ and because  $ \alpha$ is a continuous function of $\eta$ when $\eta>0$, the scaling  \eqref{eq:scaling_alpha_in_terms_of_lambda_and_eta} implies $\alpha |_{\zeta = \zeta _0} \sim \eta^{1/3}$. Since $\zeta \mapsto \beta(\zeta)$ is continuous by Proposition~\ref{prp:S characterisation} \ref{continuity of beta} and $\alpha$ is a continuous function of $\zeta$ for any $\eta>0$ the behaviour \eqref{eq:scaling_alpha_in_terms_of_lambda_and_eta} holds as long as $\eta+ \abs{\zeta-\zeta_0}$ is sufficiently small. 

We now summarise our insights by finishing the proof of Proposition~\ref{prp:v edge expansion}. 
\begin{proof}[Proof of Proposition~\ref{prp:v edge expansion}]
By following the computation leading to \eqref{tilde v equations}  we easily see that the right hand side of these equations are real analytic functions of $\vartheta$, $\eta$, $\zeta$ and $\wt{v}_i$. By the implicit function theorem  and the invertibility of $B$ on 
the range of $ Q$ the   $\wt{v}_i$ are  real analytic functions of $\vartheta, \eta$ and  $\zeta$. Similarly, the right hand side of \eqref{vartheta cubic} is a real analytic function of   $\vartheta, \eta$ and  $\zeta$. Together, we have proved Proposition~\ref{prp:v edge expansion} with 
\[
\wt{v}_i(\zeta,\eta) = \wt{w}_i(\vartheta(\zeta,\eta), \eta, \zeta)\,. \qedhere
\]
\end{proof}

\section{Properties of the Brown measure $\sigma$} \label{sec:Properties of the Brown measure}

In this section, we show our main results about the existence and properties of the Brown measure $\sigma$, Proposition~\ref{prp:Brown_measure_construction}, Theorem~\ref{thr:properties_sigma_general} and Theorem~\ref{thr:Sing classification}.  We start by proving that the Brown measure has an explicit construction as a distributional derivative of the function $L$ defined in \eqref{eq:def_L} through the solution of the Dyson equation as stated in Proposition~\ref{prp:Brown_measure_construction}.  
The \hyperlink{proof:pro:definition_sigma}{proof of Proposition~\ref{prp:Brown_measure_construction}} is presented in Subsection~\ref{subsec:Representation of Brown measure}, the proofs of Theorem~\ref{thr:properties_sigma_general} and Theorem~\ref{thr:Sing classification} at the end of this section,  Section~\ref{sec:Properties of the Brown measure}.

\subsection{Characterisation of the Brown measure $\sigma$} \label{subsec:Representation of Brown measure}

Here, we present the \hyperlink{proof:pro:definition_sigma}{proof of Proposition~\ref{prp:Brown_measure_construction}}. 
The main idea of this proof is to show that 
 $-L$ from \eqref{eq:def_L} is subharmonic and, therefore, the distribution $- \frac{1}{2\pi} \Delta L$ is induced by a measure. 
Before we present this proof, we establish a few necessary ingredients.  
The next lemma, in particular, implies that $L$ is well-defined.

\begin{lemma}[Integrating $\avg{v_1}$ with respect to $\eta$] \label{lmm:Integrating v}
Let  $a \in \mathcal B$ and  $s$ satisfy \ref{assum:flatness}. 
Then, uniformly for $\zeta \in \C$ 
and $\eta>0$, we have 
\begin{equation} \label{eq:avg_M_bounded} 
 0 \leq \avg{v_1(\zeta,\eta)} \lesssim \frac{1}{1 + \eta}. 
\end{equation} 

Furthermore, uniformly for any $T >0 $ and $\zeta \in \C$, we have 
\begin{equation} \label{eq:integral_bounds} 
\int_0^T \absbb{\avg{v_1 (\zeta, \eta)} - \frac{1}{1 + \eta} } \dd \eta 
\lesssim \min \Big\{ T , \sqrt{1 + \abs{\zeta}} \Big\}, 
\qquad 
 \int_T^\infty \absbb{\avg{v_1(\zeta,\eta)} - \frac{1}{1 + \eta} } \dd \eta 
\lesssim \frac{1+\abs{\zeta}}{T}. 
\end{equation}  
\end{lemma} 

\begin{proof} 
From \eqref{eq:v_simple_bounds} and Lemma~\ref{lem:hilbert_schmidt}, we immediately conclude \eqref{eq:avg_M_bounded}. 
The bounds in \eqref{eq:integral_bounds} follow directly from \eqref{eq:avg_M_bounded} and \eqref{eq:v_1_large_eta}. 
\end{proof} 

\begin{proof}[\linkdest{proof:pro:definition_sigma}{Proof of Proposition~\ref{prp:Brown_measure_construction}}] 
We first note that each of the identities \eqref{eq:sigma_L_identity} and \eqref{sigma formula in terms of y explicit} 
uniquely characterises a probability measure on $\C$. 
We now show that \eqref{eq:sigma_L_identity} and \eqref{sigma formula in terms of y explicit} characterise the same probability measure on $\C$. 
Let $f \in C_0^2(\C)$. By dominated convergence and \eqref{eq:integral_bounds}, we obtain 
\begin{align*} \label{eq:relation L and y}
\int_{\C} \Delta f(\zeta)  W  (\zeta) \dd^2 \zeta 
& = \lim_{\eps \downarrow 0} \int_{\C} \Delta f(\zeta) \int_\eps^\infty \bigg( \avg{v_1(\zeta, \eta)} - \frac{1}{1 + \eta} \bigg) \dd \eta \dd^2 \zeta \\ 
& = 2 \lim_{\eps \downarrow 0} \int_{\C} \partial_{\zeta} f(\zeta) \int_\eps^\infty \partial_\eta \avg{\ol{y(\zeta,\eta)}}\dd \eta \dd^2 \zeta \\ 
& = - 2 \lim_{\eps\downarrow 0} \int_{\C} \partial_{\zeta} f(\zeta) \avg{\ol{y(\zeta,\eps)}} \dd^2 \zeta. 
\end{align*} 
Here, in the second step, we integrated by parts, exchanged differentiation and integration and 
used $\partial_{\bar\zeta} \avg{v_1} = - \partial_\eta \avg{\ol{y}}/2$ due to \eqref{eq:relations_derivatives}, 
\eqref{eq:avg_v1_equals_avg_v2} and \eqref{eq:structure_M}. To see that the $\partial_\zeta$-derivative and the integral in the second step can be exchanged we use the bound $\abs{\partial_{\bar\zeta} \avg{v_1}} \lesssim \eta^{-2}$ that follows from \eqref{eq:derivatives_M} and Lemma~\ref{lem:stability_eta_positive}.
The third step is a consequence of $\lim_{\eta \to \infty} \avg{y(\zeta, \eta)} = 0$ for all $\zeta \in \C$, 
which follows from the definition of $y$ in \eqref{eq:B_relation} and the upper bound in \eqref{eq:preliminary_bounds_v}. 
 Owing to \eqref{eq:B_relation}, this  
 shows that the right-hand sides of \eqref{eq:sigma_L_identity} and \eqref{sigma formula in terms of y explicit} coincide. 

What remains is to prove that \eqref{sigma formula in terms of y explicit} characterises the Brown measure of $a + \mathfrak c$. 
From \eqref{eq:Brown_measure}, we see that the Brown measure $\sigma=\sigma_{a + \mathfrak c}$ of $a + \mathfrak c$ 
coincides with $\frac{1}{2\pi} \Delta \log D(a + \mathfrak c - \zeta)$, where $\Delta$ denotes the distributional 
Laplacian with respect to $\zeta$. 
On the other hand, the properties of the Fuglede-Kadison determinant from \eqref{eq:def_determinant} 
 imply 
\[ D(a + \mathfrak c -\zeta) = D(\abs{a + \mathfrak c - \zeta}) = \lim_{\eta \downarrow 0}  
(D( \abs{a + \mathfrak c - \zeta}^2 + \eta^2))^{1/2}. \] 
Hence, integration by parts, 
\eqref{eq:y_from_determinant} and \eqref{eq:structure_M} yield 
\begin{align*}  
\int_{\C} \Delta f(\zeta) \sigma(\dd^2 \zeta)  
= \lim_{\eta \downarrow 0} \frac{1}{4\pi} \int_{\C} \Delta f (\zeta) \log D( \abs{a + \mathfrak c - \zeta}^2 + \eta^2) \dd^2 \zeta  
= \lim_{\eta \downarrow 0} \frac{1}{\pi} \int_{\C} (\partial_{\bar \zeta} f(\zeta)) \avg{y(\zeta,\eta)} \dd^2\zeta. 
\end{align*}  
Note that the Fuglede-Kadison determinant here is monotone in $\eta$, justifying the exchange of limit and integration.   
Thus,  owing to \eqref{eq:B_relation},  \eqref{sigma formula in terms of y explicit} characterises the Brown measure $\sigma_{a + \mathfrak c}$, 
which completes the proof of Proposition~\ref{prp:Brown_measure_construction}. 
\end{proof}

\subsection{Strict positivity in the bulk} \label{sec:proof_positivity_derivative_y} 

In this subsection we show that the Brown measure has strictly positive density in the bulk, i.e.\  inside $\mathbb S$ as defined in \eqref{def of beta and B and S}. 

\begin{proposition} \label{pro:derivative_y_positive} 
Let  $a \in \mathcal B$ and  $s$ satisfy \ref{assum:flatness}. 
Then there is $C \sim 1$ such that $- \partial_{\bar \zeta} \avg{y(\zeta, \eta)}  \in (0,C]$ 
for all $\zeta \in \C$ and $\eta >0$. 
 \end{proposition}

Differentiating the last identity in \eqref{eq:y_from_determinant} and  
a straightforward computation using \eqref{eq:structure_M} and $Y(Y^*Y +\eta^{2})^{-1} Y^* = 1 - \eta^2 ( YY^* + \eta^2)^{-1}$ with $Y = a + \mathfrak c - \zeta$ yield 
\begin{align*} 
 - \partial_{\ol{\zeta}} \avg{y(\zeta,\eta)} & = \eta^2 \avg{E[ ( \abs{a + \mathfrak c - \zeta}^2 + \eta^2)^{-1}
(\abs{a + \mathfrak c - \zeta}_*^2 + \eta^2)^{-1}]} \\ 
& = \eta^2 \avg{E[ ( \abs{a + \mathfrak c - \zeta}^2 + \eta^2)^{-1/2}( \abs{a + \mathfrak c - \zeta}_*^2 + \eta^2)^{-1} ( \abs{a + \mathfrak c - \zeta}^2 + \eta^2)^{-1/2}]} , 
\end{align*}  
where we used the abbreviations $\abs{Y}^2 = Y^*Y$ and $\abs{Y}^2_* = Y Y^*$. 
This implies 
\begin{equation} \label{eq:bounds_partial_y} 
 0 < - \partial_{\ol{\zeta}} \avg{y(\zeta,\eta)} \leq \eta^{-2}. 
\end{equation} 
The lower bound $- \partial_{\ol{\zeta}}\avg{y(\zeta,\eta)} \ge 0$ is equivalent to  $\zeta \mapsto \log D( \abs{a + \mathfrak c - \zeta}^2 + \eta^2) $ being subharmonic, which was observed in  \cite[Lemma 2.8]{MR2339369}.

\begin{proof}[\linkdest{proof:pro:derivative_y_positive}{Proof of Proposition~\ref{pro:derivative_y_positive}}]
The main work in this proof is to represent $-\partial_{\ol{\zeta}} \avg{y(\zeta,\eta)}$ as the quadratic form of a self-adjoint operator, which we show to be positive  and bounded. 
For $\eta>0$ and $\zeta \in \C$, we start from the second identity in \eqref{eq:derivatives_M} and compute 
\bes{
 \partial_{\ol{\zeta}} \scalar{E_{21}}{M} 
=\scalarB{E_{21}} {\cal{L}^{-1}E_{21}}
=\scalarB{\cal{C}_{M}^*E_{21}} {E_{21}}+\scalarB{\cal{C}_{M}^*E_{21}} {(1-\Sigma\cal{C}_M)^{-1}\Sigma\cal{C}_ME_{21}}\,,
}
where $\cal{C}_{M}R:= MRM $  and $\cal{C}_{M}^* = \cal{C}_{M^*}$. Note that $1-\Sigma\cal{C}_M$ is invertible since $1-\Sigma\cal{C}_M=\cal L\cal{C}_M$,  $\mathcal{C}_M$ is invertible  and $\cal L$ is invertible by Lemma~\ref{lem:stability_eta_positive}.  
With
\[
\Sigma\cal{C}_ME_{21} = 
\mtwo{\ii S( v_2\ol{y})& 0}{0 &\ii S^*( v_1\ol{y})}\,, \qquad \cal{C}_{M^*}E_{21} = 
 \mtwo{-\ii v_1\ol{y}& \ol{y}^2}{-v_1v_2 & -\ii v_2\ol{y}}\,,
\]
and the action of $\Sigma\cal{C}_M$ on diagonal matrices in $\cal{B}^{2 \times 2}$ given by 
\[
\Sigma \cal{C}_M\mtwo{r_1 & 0}{0& r_2} = 
\mtwo{S(\abs{y}^2 r_{1} -v_2^2r_{2}) &0}{0&S^*(- v_1^2 r_{1} + \abs{y}^2r_{2})}
\]
this  simplifies to 
\bels{sigma formula via Y}{
- \partial_{\ol{\zeta}}  \scalar{E_{21}}{M}  =  \frac{1}{2}\avg{v_1v_2}+\scalarbb{\vtwo{v_1\ol{y}}{v_2\ol{y}}}{(1-Y)^{-1}\vtwo{S(v_2\ol{y})}{S^*(v_1\ol{y})}}\,,
}
where the
scalar product on $\cal{B}^2$  is the one from  \eqref{scalar product on B2} and 
  $Y \colon \cal{B}^2 \to \cal{B}^2$ is defined as
\[ Y \begin{pmatrix} r_1 \\ r_2 \end{pmatrix} = \begin{pmatrix} S(\abs{y}^2 r_{1} -v_2^2r_{2})\\ 
S^*(- v_1^2 r_{1} + \abs{y}^2r_{2}) \end{pmatrix}
=\mtwo{SD_{\abs{y}}^2&  -SD_{v_2}^2}{-S^*D_{v_1}^2 & S^* D_{\abs{y}}^2} \begin{pmatrix} r_1 \\ r_2 \end{pmatrix} \,.\]

Now we introduce a symmetrisation of $Y$. 
For this purpose we define $\wh{v} \in \cal{B}$ via
\bels{def: v}{
\wh{v}:= \sqrt{v_1 (\eta +Sv_2)} = \sqrt{v_2 (\eta +S^{*}v_1)}\,,
}
where the second identity is due to \eqref{eq:v2_St_v1_equals_v1_S_v2},
and  $V, F,T  \in \cal{B}^{2 \times 2}$ as
 \bels{def of F,T,V}{
 T:= \mtwo{-\hat{v}^2 & \abs{a-\zeta}^2  \frac{v_1v_2}{\hat{v}^2}}{\abs{a-\zeta}^2 \frac{v_1v_2}{\hat{v}^2} & -\hat{v}^2 }\,,\quad
 V:= \mtwo{\frac{\hat{v}}{v_1}& 0}{0 & \frac{\hat{v}}{v_2}}\,, \quad 
 F := V^{-1}S_oV^{-1}
 }
   analogous to  \cite[(3.27)]{Altcirc}.  Note that   $ |a-\zeta|^2\frac{v_1 v_2}{\hat{v}^2}=\frac{|y|^{2}}{v_1v_2}$ by the definition of $y$ in \eqref{eq:B_relation}. 
Then $VF TV^{-1} = Y$ and represented in terms of $F$ and $T$ the formula \eqref{sigma formula via Y} reads
\bels{rep of sigma via F}{
-2\partial_{\ol{\zeta}} \scalar{E_{21}}{M}  
=
\scalarbb{\vtwo{\wh{v}\2\ol{y}}{\wh{v}\2\ol{y}}}{\pbb{\frac{1}{{X}}+\frac{2}{1-{F}{T}}\2{F}}\vtwo{ \wh{v}\2\ol{y}}{\wh{v}\2\ol{y}}}\,,
}
where we introduced
\bels{def of X}{
{X}:=\mtwo{ D\pb{\frac{\hat{v}^{2}}{v_1v_2}\abs{y}^2} & 0}{0 & D\pb{ \frac{\hat{v}^{2}}{v_1v_2}\abs{y}^2}}\,.
}

In particular, \eqref{sigma formula via Y} is  the quadratic form of a self-adjoint operator, evaluated on a vector in the subspace of $\cal{B}^2$ with identical entries in the first and second component.
With the orthogonal projection onto this subspace represented by 
\[
{E}:= \frac{1}{2}\mtwo{1 & 1}{1 & 1} \in \cal{B}^{2 \times 2}\,,
\]
we have ${T}{E} = {E} {T} $,  ${X}{E} = {E} {X} $ and 
\bels{rep of T}{
{T} = -1 +  2{X}{E}\,.
}
The representation \eqref{rep of T} of $T$ holds because of 
\bels{hat v relation}{
1  = \wh{v}^2+ \frac{\wh{v}^2}{v_1v_2}\abs{y}^2\,,
}
which follows directly from  \eqref{eq:V_equations}, 
\eqref{eq:B_relation} and the definition of $\wh{v}$ in 
\eqref{def: v}. 

Inserting \eqref{rep of T} into \eqref{rep of sigma via F}, we see that proving positivity of the right hand side of \eqref{rep of sigma via F} reduces to proving that the operator 
\bels{expression through K}{
{E}\pbb{\frac{1}{{X}}+\frac{2}{1+{F}-2{F}{X}{E}}\2{F}}{E} &= {E}\frac{1}{\sqrt{{X}}}
\pbb{1+\frac{2}{1+\wt{{F}}\frac{1}{{X}}-2\wt{{F}}{E}}\wt{{F}}E}
\frac{1}{\sqrt{{X}}}{E}
\\
 &=
{E}\frac{1}{\sqrt{{X}}}
\pbb{1+ \frac{1}{1- \frac{2}{1+ \wt F\frac{1}{X} } \wt FE}\frac{2}{1+ \wt F\frac{1}{X} }\wt FE}
\frac{1}{\sqrt{{X}}}{E}
\\
&= {E}\frac{1}{\sqrt{{X}}}
\pbb{1-\frac{2}{1+ \wt{{F}}\frac{1}{{X}}}\wt{{F}}{E}}^{-1}
\frac{1}{\sqrt{{X}}}{E}
= {E}\frac{1}{\sqrt{{X}}}
\frac{1}{1-{K}}
\frac{1}{\sqrt{{X}}}{E}
}
on $\cal{B}^2$ is positive definite on the image of $E$, where we introduced 
\bels{def:F and K}{
\wt{{F}}:= \sqrt{{X}}\2 {F} \sqrt{{X}}\qquad \text{ and } \qquad {K}:= {E}\frac{2}{1+ \wt{{F}}\frac{1}{{X}}}\wt{{F}}{E}
={E}\sqrt{{X}} \frac{2{F}}{1+ {F}}\sqrt{{X}}{E}
}
in the calculation and used $E^2=E$  as well as $EX=XE$. 
Combining \eqref{rep of sigma via F}, \eqref{rep of T}, \eqref{def of X} and \eqref{expression through K})
yields 
\begin{equation} \label{eq:y_with_K} 
-2\partial_{\ol{\zeta}} \scalar{E_{21}}{M}  
= \scalarbb{\vtwo{ \ol{e}_y\sqrt{v_1v_2}}{\ol{e}_y\sqrt{v_1v_2}}}{
\frac{1}{1-{K}}
\vtwo{ \ol{e}_y\sqrt{v_1v_2}}{\ol{e}_y\sqrt{v_1v_2}}}, 
\end{equation} 
where $e_y := \frac{y}{\abs{y}} \in \cal{B}$. 
We now split the self-adjoint operator ${F} = {F}_+- {F}_-$ into its positive and negative parts 
and estimate 
\bels{lower bound on 1-K}{
 \sqrt{{X}}\pbb{\frac{2{F}_+}{1+ {F}_+}-\frac{2{F}_-}{1- {F}_-}} \sqrt{{X}}
\le  \sqrt{{X}}\pbb{\frac{2{F}_+}{1+ {F}_+}} \sqrt{{X}} \leq 1 - D_{\wh{v}}^2 \,,
}
where we used $\norm{{F}}<  1$ for the first inequality.  To see this,
from the definitions of $F$ in \eqref{def of F,T,V}  and $\wh v$ in \eqref{def: v} we read off 
\[
F\vtwo{\wh v}{\wh v} =\vtwo{\wh v}{\wh v} -\eta \vtwo{{v_1}/{\hat{v}}}{{v_2}/{\hat{v}}} \,.
\]
Since $F$ is a positivity preserving compact operator, there is a positive eigenvector $f = (f_1, f_2)$ with $Ff = \norm{F} f$. Taking the scalar product with $f$ yields
\[
\norm{F} =1 -\eta \frac{\avg{f_1 v_1 /\wh v}+\avg{f_2 v_2 /\wh v}}{\avg{ f_1 \wh v}+\avg{ f_2 \wh v}} <1\,.
\]
The final inequality in \eqref{lower bound on 1-K} follows from $\frac{2{F}_+}{1+ {F}_+} < 1$ because $0 \leq F_+ \leq \norm{F} < 1$  
and \eqref{hat v relation}. 
Hence, $1 - K>0$, which implies $- \partial_{\bar \zeta} \avg{y(\zeta, \eta)} >0$. Moreover, 
we plug \eqref{lower bound on 1-K} into \eqref{eq:y_with_K} and use \eqref{def: v} as well as \eqref{eq:Saveraging}
 to obtain 
\[ 
-2\partial_{\ol{\zeta}} \scalar{E_{21}}{M}  \leq \avgbb{\frac{v_2}{\eta + S v_2}} \lesssim 1, 
\] 
which proves the upper bound on $- \partial_{\bar \zeta} \avg{y(\zeta, \eta)}$. 
This completes the proof of Proposition~\ref{pro:derivative_y_positive}. 
\end{proof}

The next corollary follows from Proposition~\ref{prp:Brown_measure_construction} and the upper bound in Proposition~\ref{pro:derivative_y_positive}.  
We recall the definition of $\mathbb S_\eps$ from \eqref{eq:def_S_eps}.

\begin{corollary}\label{cor:supp_sigma_subset_S_0} 
Let  $a \in \mathcal B$,  $s$ satisfy \ref{assum:flatness} and $\sigma$ be the measure from Proposition~\ref{prp:Brown_measure_construction}. 
Then $\supp \sigma \subset \mathbb{S}_0$ 
and the measure $\sigma$ satisfies the identity
\bels{sigma formula in terms of y}{
\sigma(\zeta) = -\lim_{\eta \downarrow 0}\frac{1}{ \pi}  \partial_{\ol{\zeta}}\avg{ y(\zeta, \eta)}
}
in the sense of distributions on $\C$, where  $y  $ is the $(2,1)$ component of $M$ from \eqref{eq:structure_M}. 
Moreover, $\sigma$ is absolutely continuous with respect to the Lebesgue measure on $\C$ and its density is bounded. 
\end{corollary} 

Owing to \eqref{eq:y_with_K} and \eqref{sigma formula in terms of y}, we see 
that the Brown measure admits the representation   
\bels{rep of sigma with K}{
\pi\1 \sigma  = \lim_{ \eta \downarrow 0} 
\scalarbb{\vtwo{ \ol{e}_y\sqrt{v_1v_2}}{\ol{e}_y\sqrt{v_1v_2}}}{
\frac{1}{1-{K}}
\vtwo{ \ol{e}_y\sqrt{v_1v_2}}{\ol{e}_y\sqrt{v_1v_2}}}
}
in a distributional sense, where $K$ is defined in \eqref{def:F and K}.

\begin{proposition}[Strict positivity of Brown measure on $\mathbb S$]   \label{pro:Strict positivity of Brown measure}
Let  $a \in \mathcal B$ and  $s$ satisfy \ref{assum:flatness} and \ref{assum:Gamma}. 
Then the density of the Brown measure $\sigma$ (cf.\ Corollary~\ref{cor:supp_sigma_subset_S_0}) is strictly positive and real analytic on $\mathbb S$.  
\end{proposition} 

For the proof of Proposition~\ref{pro:Strict positivity of Brown measure}
we compute  the  Brown measure   through the formula \eqref{sigma formula in terms of y}, i.e.\ the distributional identity $\pi\sigma=-2\lim_{\eta \downarrow 0}\ol{\partial}_{\zeta} \scalar{E_{21}}{M}$.
First we will see that the right hand side in \eqref{sigma formula in terms of y} is nonnegative and is in fact positive when evaluated at $\eta>0$, i.e.\ we prove Proposition~\ref{pro:derivative_y_positive}. After that we will see that under assumption \ref{assum:Gamma} the right hand side can be continuously extended to $\eta=0$ away from $\partial \mathbb{S}$ and remains a bounded function of $\zeta$, i.e.\ $\sigma$ has a density.

\begin{proof} [Proof of Proposition~\ref{pro:Strict positivity of Brown measure}]
For the proof of analyticity of $\sigma$, we recall the definition of $y$ from \eqref{eq:B_relation}. 
We conclude from Proposition~\ref{prp:S characterisation} \ref{characterisation of S}, \eqref{eq:def_rho_zeta_0} 
and Proposition~\ref{prp:Analyticity in the bulk} that $\mathbb{S} \to \C$, $\zeta \mapsto y(\zeta,0)$ 
is real analytic.  Therefore, \eqref{sigma formula in terms of y} implies that $\sigma$ is real analytic on $\mathbb{S}$. 

To prove a lower bound on $\sigma$, we use \eqref{rep of sigma with K} and see that $1-K$ remains bounded on the image of $E$ as $\eta \downarrow 0$. Indeed by  the identity in  \eqref{lower bound on 1-K} the only contribution to $K$ that may potentially be unbounded is the one associated with $F_-$. However, $EF_- E |_{\eta =0}\le 1-\eps$ for some $\eps>0$ 
 because of the spectral gap of $F$ above $-1$ in  Lemma~\ref{lmm:Spectral properties of F and T}  and the fact that $(\wh{v},-\wh{v})$, 
the eigenvector corresponding to eigenvalue $-1$, is mapped to zero by $E$.
\end{proof}

\subsection{Edge behaviour of the Brown measure} 
\label{sec:proof_properties_sigma_general} 

Here we show that $\sigma$ can be continuously extended to the boundary of $\mathbb{S}$ and compute its boundary values. 
Throughout this subsection we assume \ref{assum:flatness} and \ref{assum:Gamma}. 

\begin{proposition}[Boundary values of $\sigma$]
\label{prp: Boundary values of sigma}
There exists a real analytic extension of $\sigma|_{\mathbb{S}}$ to a neighbourhood of  $\ol{\mathbb{S}}$. 
The extension satisfies 
\bels{sigma edge value}{
\sigma(\zeta_0) =\frac{1}{\pi} \frac{\abs{\avg{(a-\zeta_0)\ell_0\1 b_0}}^2}{\avg{\abs{a-\zeta_0}^4 \ell_0^2\1b_0^2}}
= \frac{\avg{\ell_0\1b_0}^2\abs{\partial_{{\zeta}}\2\beta(\zeta_0)}^2}{\pi\1\avg{\abs{a-\zeta_0}^4 \ell_0^2\1b_0^2}}
}
for any $\zeta_0\in \partial\mathbb{S}$, where $\ell_0:=\ell|_{\zeta=\zeta_0}$ and $b_0:=b|_{\zeta=\zeta_0}$. Furthermore, for $\zeta_0 \in \rm{Sing}$ a singular boundary point the extension satisfies $\partial_\zeta \sigma (\zeta_0)=0$ and 
\bels{Delta sigma}{
\Delta \sigma (\zeta_0) = \frac{ 32|\avg{\ell_0  (a-\zeta_0) B_0^{-1}[b_0 (a-\zeta_0)]}|^2+\avg{\ell_0 b_0}^2(\Delta\beta(\zeta_0))^2}{2\1\pi\1\avg{|a-\zeta_0|^4\ell_0^2 b_0^2 }}\,.
}
\end{proposition}

\begin{proof}
We use the identity \eqref{sigma formula in terms of y}
to compute $\sigma$ at some $\zeta$ in a sufficiently small neighbourhood of $\zeta_0 \in \partial \mathbb{S}$ in terms of $y$.  We expand $y$ in terms of $v_1,v_2$ with the help of \eqref{eq:y_identities} and  expand $v_1,v_2$ in terms of $\beta$.  In this proof we set $\eta=0$ and by shifting the spectrum we assume without loss of generality that $\zeta_0=0$. 
In this case, instead of the decomposition~\eqref{vi split}, 
it is more convenient to write
\bels{vi split with wi}{ 
v_1 = \vartheta_1 \1(\ell -\theta  w_1), \qquad v_2 = \vartheta_2 \1(b -\theta w_2 )
}
where we introduced $\theta:= \vartheta_1 \vartheta_2$ and where due to \eqref{bound on tilde alpha} the vectors $w_1$ and $w_2$ remain bounded, i.e. 
$\norm{w_1}_\infty+\norm{w_2}_\infty\lesssim 1$. We already know $\alpha \sim \vartheta_1 \sim \vartheta_2$ due to \eqref{eq:scaling_alpha_in_terms_of_lambda_and_eta} and \eqref{vartheta identity} and that $\vartheta_i$ are real analytic functions of $\zeta$ for $\zeta \in \mathbb{S}$ by Proposition~\ref{prp:v edge expansion} even for $\eta=0$. Furthermore, $\vartheta=0$ at $\eta=0$ for $\zeta \in \C \setminus \ol{\mathbb{S}}$. 
For $\zeta \in \overline{\mathbb{S}}$ and $\eta=0$, the first equation in \eqref{edge identities} becomes
\bels{beta theta equation}{
\beta \1 \avg{\ell b}+\theta\avg{\ell (b-\theta w_2)  S^*(\ell -\theta w_1)S(b-\theta w_2)}=0,  
}
after projecting  onto the left eigenvector $\ell =\ell_\zeta$  of $B=B_\zeta$ and 
\bels{wi theta equation}{
 w_1&=  (B^{-1})^*Q^*[(\ell -\theta w_1) S^*(\ell -\theta w_1)S(b-\theta w_2)], 
 \\
w_2 &=  B^{-1}Q[(b-\theta w_2)  S^*(\ell -\theta w_1)S(b-\theta w_2)],    
}
after projecting \eqref{edge identities} onto the complement of the left and right  eigendirections, where we recall the projection $Q$ from \eqref{def:vartheta_i}. 
Owing to the implicit function theorem, we  
 see from the structure of the equations \eqref{beta theta equation} and  \eqref{wi theta equation} that $\beta$, $w_1$ and $w_2$ are locally analytic function of $\theta$ for fixed  $B$, and therefore fixed $b, \ell, Q$. 
 Furthermore,  as $\norm{w_1}_\infty + \norm{w_2}_\infty \lesssim 1$, 
\eqref{beta theta equation} shows
\bels{beta theta relation}{
\beta \1 \avg{\ell b}+\theta\avg{\ell b  (S^*\ell) (Sb)}  
 = O(\theta^2)\,. 
}
Combining \eqref{eq:y_identities} with \eqref{sigma formula in terms of y} we get the formula 
\bels{sigma derivative formula}{
\pi\1 \sigma = \avgbb{\frac{1}{a-\zeta} \partial_{\ol{\zeta}} (v_1Sv_2)}  \,.
}
Since  $v_1Sv_2 
 = \theta \ell S b + O(\theta^2)$  is an analytic function of $\theta$ for fixed $B$, by the relation \eqref{beta theta relation} between $\theta$ and $\beta$, as well as the fact that $\beta$ is locally real analytic around $\zeta_0=0$ due to Corollary~\ref{crl:beta as eigenvalue of B}, 
we see that $\sigma$ is the $\ol{\zeta}$-derivative of an analytic function of $\beta$ for fixed $B$. Since $B$ also analytically depends on $\zeta$, we conclude that the right hand side of \eqref{sigma derivative formula} can be analytically extended to $\zeta$ in a neighbourhood of $0$. 
We denote this extension again by $\sigma$. 
Furthermore, we have
\begin{equation} \label{eq:partial_ol_zeta_v_1_S_v_2} 
\partial_{\ol{\zeta}}(v_1Sv_2 ) = (\ell S b)\1 \partial_{\ol{\zeta}}\theta + \theta \1\partial_{\ol{\zeta}}(\ell S b) +O( | \partial_{\ol{\zeta}}\theta|\theta+ \theta^2)\,.
\end{equation}
Evaluating at $\zeta=0$ yields 
\[
\pi\1 \sigma |_{\zeta=0} = \avgbb{\frac{1}{a}(\ell_0 S b_0)\1 \partial_{\ol{\zeta}}\theta |_{\zeta=0}} =-\avg{\ol{a}\ell_0 b_0}\frac{\avg{\ell_0 b_0}}{\avg{\ell_0^2b_0^2 |a|^4}}\partial_{\ol{\zeta}}\beta |_{\zeta=0}
\]
where we used \eqref{beta theta relation} 
and $B_0^* \ell_0=0= B_0b_0$. With \eqref{derivatives of beta} the claim \eqref{sigma edge value} follows. 

Now we assume $\zeta_0=0 \in \rm{Sing}$ is a singular boundary point. 
 Since $\beta$ is locally analytic in a neighbourhood of $0$ and $\partial_{\ol{\zeta}}\beta(0)=0 $, we have $\beta =O(|\zeta|^2)$ and   $\partial_{\ol{\zeta}}\beta = O(|\zeta|)$. Thus, using 
\eqref{eq:partial_ol_zeta_v_1_S_v_2}, 
\eqref{beta theta relation} 
and $Bb =\beta b$ we get 
\bels{derivative vSv}{
\partial_{\ol{\zeta}}(v_1Sv_2 ) =  |\zeta-a|^2 \ell b\2 \partial_{\ol{\zeta}}\theta + \theta \partial_{\ol{\zeta}}(\ell S b) +O( |\zeta|^3)\,.
}
Since $\partial_{\ol{\zeta}}B  = D_{\zeta-a}$, $\beta(\zeta) =0$, $\partial_\zeta \beta(\zeta) = 0$
and $\ell$, $b \in \mathcal B_+$, we have  
\bels{derivatives of ell and b}{
\partial_{\ol{\zeta}} b|_{\zeta=0} &=\ol{\partial_{{\zeta}} b}|_{\zeta=0}=  B_0^{-1}[a b_0]\,, \qquad  \partial_{\ol{\zeta}}\ell |_{\zeta=0}=\ol{\partial_{{\zeta}}\ell} |_{\zeta=0}=  (B_0^*)^{-1}[{a} \ell_0]\,,
}
where all quantities with index $0$ are evaluated at $\zeta=\zeta_0=0$ and we used $ \avg{a \ell_0b_0}=0$ because $\partial_\zeta \beta(0)=0$ (cf. \eqref{derivatives of beta}) to guarantee the inverses of $B_0$ and $B_0^*$ can be applied. 
Furthermore, 
\bels{expand first term}{
\avgb{\ol{(a-\zeta)}\ell b}&=-\ol{\zeta}\avg{\ell_0b_0}  +2 \avg{\ol{a}  \re [\zeta\partial_\zeta (\ell b)|_{\zeta=0}]} +O(|\zeta|^2)
\\
&=-\ol{\zeta}\avg{\ell_0b_0}  +2 \avg{\ol{a}  \re [\zeta\ell_0 B_0^{-1}[\ol{a} b_0]+\zeta b_0 (B_0^*)^{-1}[\ol{a} \ell_0]]} +O(|\zeta|^2)
}
where we used $ \avg{a \ell_0b_0}=0$ again. 
Inserting \eqref{derivative vSv} into  \eqref{sigma derivative formula} and using \eqref{derivatives of ell and b}, \eqref{beta theta relation} and \eqref{expand first term}
we find 
\bels{sigma expansion in terms of beta}{
\pi\1 \sigma &= \avgb{\ol{(a-\zeta)}\ell b} \partial_{\ol{\zeta}}\theta + \avgbb{\frac{1}{a-\zeta} \partial_{\ol{\zeta}}(\ell S b)}\theta+O( |\zeta|^3) 
\\
&= 
\frac{ \avg{\ell_0 b_0}}{\avg{\ell_0^2 b_0^2 |a|^4}}\bigg(
\pb{\ol{\zeta}\avg{\ell_0b_0}  -2 \avg{\ol{a}  \re [\zeta\ell_0\ B_0^{-1}[\ol{a} b_0]+\zeta b_0 (B_0^*)^{-1}[\ol{a}\ell_0]]} } \partial_{\ol{\zeta}}\beta 
\\
& \qquad - \avgbb{\frac{1}{a} \pb{(S b_0) (B_0^*)^{-1}[a \ell_0] + \ell_0 S B_0^{-1}[a b_0]}} \beta\bigg)+O( |\zeta|^3)
}
In particular, $\partial_\zeta \sigma |_{\zeta=0}=0$. 
By \eqref{beta expansion at edge} the derivative of $\beta$ satisfies 
\[
\partial_{\ol{\zeta}}\beta  = \pbb{1- 2 \re \sbb{\frac{\avg{\ell_0  \ol{a} B_0^{-1}[b_0 a]}}{\avg{\ell_0\2b_0}}} }\zeta - 2
\frac{\avg{\ell_0  aB_0^{-1}[b_{0} a]}}{\avg{\ell_0\2b_0}} \ol{\zeta} +O(|\zeta|^2)\,.
\]
Using this, $B_0b_0 = 0$ and the expansion of $\beta$ from \eqref{beta expansion at edge} in \eqref{sigma expansion in terms of beta}, 
we conclude 
\bes{
\frac{\pi\1\avg{\ell_0^2 b_0^2 |a|^4}}{ 8\avg{\ell_0 b_0}^2}\Delta \sigma|_{\zeta=0} 
&= 2|R[a,a]|^2+(1-R[a,\ol{a}]- R[\ol{a},a])^2= 2|R[a,a]|^2+(\partial_\zeta \partial_{\ol{\zeta}}\beta(0))^2,
}
where we introduced
\[
R[x,y]:= \frac{\avg{\ell_0  x B_0^{-1}[b_0 y]}}{\avg{\ell_0\2b_0}}\,, \qquad x,y \in \{a,\ol{a}\}\,.
\]
With \eqref{derivatives of beta} this proves \eqref{Delta sigma}. 
\end{proof}

\begin{lemma} \label{lem:beta_zero_and_critical_point} 
Let $\zeta \in \C$ such that $\beta(\zeta)=0$ and $\partial_\zeta\beta(\zeta)=0$. Then $\Delta\beta(\zeta) < 0$. In particular, 
\bels{boundary of S = zero of beta}{
 \partial\mathbb{S}=\{\zeta \in \C: \beta(\zeta)=0\}.
} 
\end{lemma}

\begin{proof}
Let $\zeta\in \C$ with $\beta(\zeta)=0$ and $\partial_\zeta\beta(\zeta)=0$.
From \eqref{derivatives of beta} we read off 
\bels{derivatives of beta at edge}{
\avg{\ell\1b\1 (a-\zeta)} = 0\,, \qquad 
\partial_\zeta \partial_{\ol{\zeta}}\beta=
1- 2 \re \sbb{\frac{\avg{\ell  \ol{(a-\zeta)} B^{-1}b (a-\zeta)}}{\avg{\ell\2b}}} \,,
}
where we evaluated the expressions in \eqref{derivatives of beta} at $\zeta_0=\zeta$ and omitted the projection $Q_0$ in the formula for $\partial_\zeta \partial_{\ol{\zeta}}\beta$ since $\avg{\ell\1b  (a-\zeta)}=0$ implies $Q_0[b (a-\zeta) ] =b (a-\zeta)$.

We write $\partial_\zeta \partial_{\ol{\zeta}}\beta$ in terms of 
\[
K:=D\pbb{\frac{\sqrt{\ell}}{\abs{a-\zeta}\sqrt{b}}}SD\pbb{\frac{\sqrt{b}}{\abs{a-\zeta}\sqrt{\ell}}}\,, \qquad
x := \sqrt{\ell \1b} \frac{a-\zeta}{\abs{a-\zeta}}
\]
and arrive at 
\[
\partial_\zeta \partial_{\ol{\zeta}}\beta=1- 2 \re \sbb{\frac{\avg{\ol{x} (1-K)^{-1}x}}{\avg{\abs{x}^2}}}
= -\frac{1}{\avg{\abs{x}^2}}\scalarbb{\frac{1}{1-K}x }{(1-K^*K)\frac{1}{1-K}x}\,.
\]
 Here we used that $(1- K)^{-1}x$ is well-defined since $x \perp k$ due to \eqref{derivatives of beta at edge}, 
 where $k:= \abs{a-\zeta}\sqrt{\ell \1 b}$ is the right and left Perron-Frobenius eigenvector of $K$, i.e.\ $(1-K)k=0=(1-K^*)k$ due to \eqref{B eigenvectors} with $\beta =0$. Furthermore, $(1-K^*K)k =0$ implies that $k$ is the Perron-Frobenius eigenvector of $K^*K$ and thus $1-K^* K$ is strictly positive definite on $k^\perp$, implying $\Delta \beta < 0$. 
Since $\beta$ is real analytic in a neighbourhood of $\zeta$ with $\beta(\zeta) = \partial_\zeta \beta(\zeta)=0$ according to Corollary~\ref{crl:beta as eigenvalue of B} and such $\zeta$ cannot be  a local minimum of $\beta$ due to $\Delta\beta(\zeta) <0$ we infer \eqref{boundary of S = zero of beta}.
\end{proof}

As a consequence of Lemma~\ref{lem:beta_zero_and_critical_point}, the definition of $\mathbb{S}$ in \eqref{def of beta and B and S} and Proposition~\ref{prp:S characterisation} \ref{characterisation beta>0} 
 yield 
\begin{equation} \label{eq:overline_S_equal_intersection_S_eps} 
\mathbb S_0= \overline{\mathbb S}. 
\end{equation}

We now have all ingredients to prove Theorems~\ref{thr:properties_sigma_general} and~\ref{thr:Sing classification}.

\begin{proof}[\linkdest{proof_thr:properties_sigma_general}{Proof of Theorem~\ref{thr:properties_sigma_general}}] 
Items \ref{item:prop_sigma_ii} and \ref{item:prop_sigma_iii} are proved in Propositions~\ref{pro:Strict positivity of Brown measure} 
and~\ref{prp: Boundary values of sigma}.
For the proof of \ref{item:prop_sigma_iv}, we conclude $\supp \sigma \subset \mathbb{S}_0 = \overline{\mathbb{S}}$ 
from Corollary~\ref{cor:supp_sigma_subset_S_0} and \eqref{eq:overline_S_equal_intersection_S_eps}. 
Moreover, $\mathbb{S} \subset \supp \sigma$ follows from \ref{item:prop_sigma_iii}, which completes the proof of \ref{item:prop_sigma_iv} due to 
 Proposition~\ref{prp:S characterisation} \ref{item:spec a in S}.
Note that $\partial \mathbb{S}$ is a real analytic variety due to \eqref{boundary of S = zero of beta} and 
Corollary~\ref{crl:beta as eigenvalue of B}.  
The dimension of $\partial \mathbb{S}$ is at most one as $\Delta \beta(\zeta)  \neq 0 $ if $\partial_\zeta \beta(\zeta)  =0$ by Lemma~\ref{lem:beta_zero_and_critical_point}. 
This shows~\ref{item:prop_sigma_v}. 
Part~\ref{item:prop_sigma_vi} follows from Proposition~\ref{prp: Boundary values of sigma} 
and the fact that $\ell_0$ and $b_0$ and, thus, $g$ are real analytic by Corollary~\ref{crl:beta as eigenvalue of B}.  
\end{proof} 

\begin{proof}[\linkdest{proof_thr:Sing classification}{Proof of Theorem~\ref{thr:Sing classification}}] Let $\zeta \in \rm{Sing}$ a singular point of $\sigma$, i.e.\ $\beta(\zeta) =0$ and $\partial_\zeta \beta(\zeta) = 0$. 
By Lemma~\ref{lem:beta_zero_and_critical_point} we have $\Delta \beta(\zeta)  < 0 $ and by \eqref{Delta sigma} also $\Delta \sigma(\zeta)>0 $.  Now we apply the following lemma about critical points of  real analytic functions in two dimensions  with positive Laplacian, which simplifies to the well-known \emph{Morse lemma} 
 (see e.g.\ \cite[Lemma~2.2]{MilnorMorseTheoryBook})  if the critical point is non-degenerate.  
\begin{lemma} [Higher order Morse lemma] \label{lmm:possible singuarity types}
Let $f: U \to \R$ be a real analytic function on an open set $U \subset \C$ and $\zeta_0 \in U$. Suppose that $\partial_\zeta f (\zeta_0)=0$ and $\Delta f (\zeta_0)>0 $. Then there is are a real analytic diffeomorphism $\Phi: V \to U_0$ from an open neighbourhood $V\subset \C$ of zero to an open neighbourhood $U_0 \subset U$ of $\zeta_0=\Phi(0)$, a natural number $K \ge 2$ and $\tau \in \{0,\pm 1\}$,  such that 
\[ f \circ \Phi(z) = f(\zeta_0) + (\Re z)^2 + \tau (\Im z)^K  , \qquad \text{ for all }z \in V. \] 
\end{lemma}

The \hyperlink{proof:lmm:possible singularity types}{proof of Lemma~\ref{lmm:possible singuarity types}} is given 
in Appendix~\ref{app:proof_higher_order_morse_lemma}. 
Since $-\beta$ satisfies the assumptions of Lemma~\ref{lmm:possible singuarity types} at $\zeta$ and $\beta(\zeta)=0$ we conclude  that $-\beta$ is of singularity type $x^2 +\tau y^K$ at $\zeta$, where we used the terminology from Definition~\ref{def:singularity type}. If $\tau=-1$ or $\tau = 1$ and odd $K$, then $\zeta \in \partial\ol{ \mathbb{S}} = \partial \supp \sigma$. 
In this case $\zeta \in \rm{Sing}_{K-1}^{\rm{edge}}$. 

Otherwise, i.e.\ if $\tau = 0$ or $\tau = 1$ and even $K$, $\zeta$ lies in the interior of $\ol{\mathbb{S}}$.  
In this case $\sigma$ satisfies the conditions of Lemma~\ref{lmm:possible singuarity types} at $\zeta$ with $\sigma(\zeta)=0$, as 
$\partial_\zeta \sigma(\zeta) = 0$ because of Proposition~\ref{prp: Boundary values of sigma}. 
Thus, $\sigma$ is either of  singularity type $x^2 +  y^{2n}$ for some $n \in \N$ or of singularity type $x^2$ at $\zeta$. We note that the singularity types $x^2 -  y^{n}$ and $x^2 + y^{2n+1}$ are not possible since $\sigma$ is a real analytic function in a neighbourhood of $\zeta$ by Proposition~\ref{prp: Boundary values of sigma}, which is 
nonnegative at $\zeta$ as $\zeta$ lies in the interior of $\ol{\mathbb{S}}$. 
In case $\sigma$ is of singularity type $x^2 +  y^{2n}$, we have $\zeta \in \rm{Sing}^{\rm{int}}_{2n}$ and 
$\zeta$ is an isolated point of $\C\setminus \mathbb{S}$. 
If $\sigma$ is of singularity type $x^2$, then $\zeta \in \rm{Sing}^{\rm{int}}_{\infty}$. In this case the fact that  $\sigma$ is a restriction of a real analytic function on an open neighbourhood of $\ol{\mathbb{S}}$ by Theorem~\ref{thr:properties_sigma_general} \ref{item:prop_sigma_iii} implies that the connected component of $\rm{Sing}^{\rm{int}}_{\infty}$ containing $\zeta$ is a closed analytic path without self-intersections. 
\end{proof}

\section{Existence of singularity types} \label{sec:existence_singularity_types}

In this section we  provide examples that show that all possible singularity types that are allowed by Theorem~\ref{thr:Sing classification} appear in the Brown measures of suitably chosen deformations of standard circular elements, i.e. we prove Theorem~\ref{thr:all_singularities_appear}. The proof of the theorem is summarised before Subsection~\ref{subsec:Singularity types of even order}. We work with the case $\cal{B} = L^\infty[0,1]$ and such that $\frak{c}$ is a standard circular element, i.e. $E[\frak{c}^* x \frak{c}] = E[\frak{c} x \frak{c}^*] = \avg{x}$ for $x \in \cal{B}$. 
We fix $a \in \mathcal B$. Hence, $a$ and $\frak{c}$ are ordinary $*$-free in $(\mathcal A, \avg{ \cdot})$ and $a$ is automatically normal.  
In this setup  $B_\zeta x= |a-\zeta|^2 x - \avg{x}$ for all $x \in \mathcal B$
and the Dyson equation \eqref{eq:V_equations} simplifies to 
\begin{equation} \label{eq:scalar_dyson_S_equal_avg} 
\kappa= \eta+ \avgbb{ \frac{\kappa} {\kappa^2 +   |\zeta-a|^2}}\,, \qquad 
v_1  =v_2 =  \frac{\kappa} {\kappa^2 +  |\zeta-a|^2}\,, 
\end{equation}
where $\kappa =\kappa(\zeta,\eta)= \eta + \avg{v_1(\zeta,\eta)}\ge 0$. To apply our theory, 
we will choose $a \colon [0,1] \to \C$ such that 
\[
\essinf_{x \in \mathfrak [0,1]} \int_{[0,1]} \frac{\dd y}{ \abs{a(x) - a(y)}^2 } =\infty \,.
\]
For $\zeta \in \mathbb{S}$ we have $\kappa_0=\kappa(\zeta,0)>0$ and get 
\[
\beta(\zeta)=\inf_{x \in \cal B_+} \sup_{y \in \cal B_+}\frac{\scalar{x}{(B_\zeta+\kappa_0^2)y}}{\scalar{x}{y}} -\kappa_0^2\ge -\kappa_0^2
\]
for $\beta$ from \eqref{eq:def_beta}, by choosing for $y$ the function $b := (|a-\zeta|^2 + \kappa_0^2)^{-1}$, for which  $(B_\zeta+\kappa_0^2)b=0$ due to \eqref{eq:scalar_dyson_S_equal_avg} and $\eta =0$. 
Similarly, we get  $\beta(\zeta) \le -\kappa_0^2$ by choosing $x=b$. Thus, $\beta(\zeta) = -\kappa(\zeta,0)^2$ for $\zeta \in \mathbb{S}$ and 
\bels{beta for circular}{
\avgbb{\frac{1}{|a-\zeta|^2-\beta(\zeta)}}=1\,.
}
In particular,
\bels{partial beta for circular}{
\avgbb{\frac{1}{(|a-\zeta|^2-\beta)^2}}\partial_{\ol{\zeta}}\beta = \avgbb{\frac{\zeta-a}{(|a-\zeta|^2-\beta)^2}} \,.
}
Thus, we  express the Brown measure $\sigma$ on $\mathbb{S}$ in terms of $\beta$, using \eqref{sigma derivative formula}, \eqref{beta for circular} and \eqref{partial beta for circular}, through 
\bels{sigma and beta for circular}{
\pi\1\sigma &=-\avgbb{\frac{1}{a-\zeta} \partial_{\ol{\zeta}} \frac{\beta} {|a-\zeta|^2-\beta}}  
\\
&= \avgbb{\frac{-\beta}{({|a-\zeta|^2-\beta})^2}}-
\avgbb{\frac{1}{a-\zeta}\pbb{ \frac{1} {|a-\zeta|^2-\beta} + \frac{\beta} {(|a-\zeta|^2-\beta)^2}   } }\partial_{\ol{\zeta}}\beta
\\
&= \avgbb{\frac{1}{({|a-\zeta|^2-\beta})^2}}\pb{-\beta + |\partial_{\ol{\zeta}}\beta|^2}\,.
}
 The previous formulas have been derived earlier, see e.g.\ \cite{BelinschiYinZhong2024,BordenaveCaputoChafai2014,Khoruzhenko1996,Zhong2021}.  
From \eqref{beta for circular} and \eqref{eq:def_mathbb_S}, we see that $\mathbb{S}$ is given by the equation $\mathbb{S}=\{\zeta\in \C: f(\zeta) >1\}$, where 
\bels{def: f for deformed circular}{
f(\zeta):= \int \frac{\nu(\dd \omega)}{ \abs{\zeta-\omega}^2}  =\avgbb{\frac{1}{|\zeta-a|^2}}
}
and $\nu$ is the spectral measure of $a$, i.e. $\nu$ is the unique measure on $\C$ for which $\int \zeta^k\ol{\zeta}^l \nu(\dd \zeta) = \avg{a^k\ol{a}^l}$. We interpret $f(\zeta)=\infty$ for $\zeta \in \supp \nu$. We note that 
\[
\Delta f(\zeta) = \avgbb{\frac{4}{|a-\zeta|^4}}>0
\] for all $\zeta \not \in \supp \nu$.

In the following we consider deformations $a$ such that $\nu$ is symmetric with respect to the imaginary axis. In this case $f$, $\beta$ and $\sigma$ inherit this symmetry from $\nu$. In particular,  $f(\zeta) = f(- \ol{\zeta})$ for all $\zeta \in \C$. We will make use of the following corollary of Lemma~\ref{lmm:possible singuarity types}, which is proven at the end of Appendix~\ref{app:proof_higher_order_morse_lemma}. 
 \begin{corollary}[Higher order Morse lemma with symmetry]\label{crl:singularity with symmetry}
Let $f: U \to \R$ be a real analytic function on an open set $U \subset \C$ with $0 \in U$. Suppose that $\partial_\zeta f (0)=0$ and $\Delta f (0)>0 $. 
Assume furthermore, that $f$ is symmetric with respect to the imaginary axis,  i.e.\ $f(\zeta) = f(-\ol{\zeta})$ for all $\zeta \in U$ such that $-\ol{\zeta} \in U$. 
If $\partial_{\re \zeta}^2 f(0) >0$, then the diffeomorphism $\Phi$ from Lemma~\ref{lmm:possible singuarity types} 
with $\Phi(0) = 0$  can be chosen such that $D\Phi(0)$ is positive definite and diagonal.  
Furthermore, the exponent $K\ge 2$ and the sign $\tau$ in the statement of Lemma~\ref{lmm:possible singuarity types} are 
\[
K := 1+\sup\{k \in \N: \partial_{\im \zeta}^l f(0) =0\text{ for all } l \le k \}\,, \qquad 
\tau = 
\begin{cases}
\sign \partial_{\im \zeta}^K f(0) & \text{ if } \; K< \infty\,,
\\
0 & \text{ if } \; K= \infty\,.
\end{cases}
\]
\end{corollary}

The following lemma connects the singularity types of $f$, $\beta$ and $\sigma$ in this case.  It shows that in order to prove Theorem~\ref{thr:all_singularities_appear},  it suffices to show that $\nu$ from \eqref{def: f for deformed circular} can be chosen such that $f-1$ has a given singularity type of the form $x^2 + \tau \1 y^K$ with $\tau \in \{0,\pm 1\}$ and $K \in \N$ with $K \ge 2$ at $0$. 

\begin{lemma}\label{lmm:singularity types sigma f and beta}
Let $f$ be defined as in \eqref{def: f for deformed circular}, assume that $f(\zeta) = f(- \ol{\zeta})$ for all $\zeta \in \C$, that $\partial_{\re \zeta}^2 f(0)>0$ and that $f(0)=1$.   Then $0 \in \partial \mathbb{S}$ and the singularity types of $f-1$ and $-\beta$ agree at $0$. If $0 \in \rm{Sing}$ and  $-\beta$ is of singularity type $x^2 + y^{2n}$ at $0$ for some $n \in \N$, then $\sigma$ is also of singularity type $x^2 + y^{2n}$. 
\end{lemma} 

\begin{proof} For  $f(0)=1$ we have   $0 \in \partial \mathbb{S}$ by definition of $\mathbb{S}$. Then \eqref{partial beta for circular} implies 
\bels{partial f and partial beta}{
\avgbb{\frac{1}{|a-\zeta|^4}}\partial_\zeta \beta = - \partial_\zeta f\,.
}
Evaluating at $\zeta=0$ shows that $\partial_\zeta \beta =0$ if and only if $\partial_\zeta f =0$, proving that the singularity types   of $f-1$ and $-\beta$ agree for a regular edge point $0 \in \rm{Reg}$. We now denote  $x=\re \zeta$ and $y=\im \zeta$ and assume $0 \in \rm{Sing}$.  Then $f$ and $-\beta$ satisfy the assumptions in Corollary~\ref{crl:singularity with symmetry}. By applying the derivative $\partial_y$ repeatedly to \eqref{partial beta for circular} and evaluating at $\zeta=0$, we see that the exponent $K$ and the sign $\tau$ from Corollary~\ref{crl:singularity with symmetry} for $f$ and $-\beta$ agree. 

Finally, we show that the singularity types of $\sigma$ and $-\beta$ at $0$ agree.  From 
\eqref{sigma and beta for circular} we see that 
\bels{second partial sigma}{
\partial_\zeta \sigma(0) =0 \,, \qquad 
\partial_w^2 \sigma(0)=\frac{1}{\pi}\pbb{- \partial_w^2 \beta(0)+\frac{1}{4}(\partial_w^2 \beta(0))^2}\,,
}
where $w=x,y$ since $\beta$ is an even function of $x$ and $f(0)=1$. In particular $\Delta\sigma(0)>0$ and  $\partial_x^2 \sigma(0)>0$ because  $-\partial_x^2 \beta(0)=  \partial_x^2 f(0)>0$  by \eqref{partial beta for circular}. 
We conclude that $\sigma$ also satisfies the assumptions of Corollary~\ref{crl:singularity with symmetry}. Suppose $-\beta$ has singularity type $x^2 +  y^2$ at $0$, then $\partial_y^2 \sigma(0)>0$ by \eqref{second partial sigma}. Now suppose $-\beta$ has singularity type $x^2 +  y^{2n}$ at $0$ for some $n \ge 2$. By repeated application of the derivative $\partial_y$ on \eqref{sigma and beta for circular} and evaluating at $\zeta =0$ we see that 
\[
\partial_y^k \sigma (0) =- \frac{1}{\pi} \partial_y^k \beta (0)\,, \qquad k =3, \dots, 2n\,,
\]
which implies that $\sigma$ is of singularity type $x^2 +  y^{2n}$ at $0$. 
\end{proof}

To construct singularity types of $f$ at $0$ of the form $x^2 \pm y^K$ with a finite $K \ge 2$, we will choose $\nu$ with finite support, i.e. $\supp \nu = \{a_1, \ldots, a_N \}$ with distinct nonzero $a_i \in \C$ and $\nu_i=\nu(\{a_i\})$. In this case 
\bels{discrete f}{ f(\zeta) = \sum_{i = 1}^N \frac{\nu_i}{\abs{\zeta- a_i}^2}\,. }
To guarantee symmetry of $f$ with respect to the imaginary axis we write $f$ as a sum of functions of the form
\begin{equation} \label{eq:def_g_a} 
g_a(\zeta) := \frac{1}{2 }\pbb{\frac{1}{\abs{\zeta+1/a}^2}+\frac{1}{\abs{\zeta-1/\ol{a}}^2}}\,,
\end{equation} 
with $a\in \C_{>}:=\{z \in \C: \re z>0, \im z >0\}$ a complex number in the first quadrant. For even $K$ we can even restrict to $f$ being a sum of functions of the form  
\bels{Def of f_a}{
f_a(\zeta) :=\frac{1}{2}\pb{g_a(\zeta)+g_{\ol{a}}(\zeta)}
}
with  $a \in \C_{>}$. With these definitions any convex combination of functions of the form $g_a$ and $f_a$ are of the form \eqref{discrete f}. 

In the following subsections we construct an example for each singularity type that is allowed by Theorem~\ref{thr:Sing classification}. 
The result of these constructions is summarised in the following lemma.

\begin{lemma} 
\label{lmm:Examples exist}
The following examples exist.  
\begin{enumerate}[label=(\alph*)] 
\item \label{lmm:Examples 1} Quadratic singularity types:  With the choice $\nu = \frac{1}{2}(\delta_1+\delta_{-1})$ the function $f-1$, where $f$ is from \eqref{def: f for deformed circular}, has singularity type $x^2-y^2$ at $0$. The function $f-1$ with  
$f(\zeta) :=f_a(\zeta)$ and $a:=\frac{1}{\sqrt{2}}(1+\ii)$
has singularity type $x^2+y^2$ at $0$. 
\item \label{lmm:Examples 2}Singularity types of even $y$-power: Let $n \in \N$ with $n \ge 2$ and $\tau \in \{\pm 1\}$. There are $c_1, \dots, c_n>0$ with $\sum_{i=1}^n c_i=1$ and distinct $a_1, \dots, a_n \in \C_>$, such that the function $f-1$ with 
\[
f(\zeta):= \sum_{i=1}^n c_i f_{a_i}(\zeta)
\]
has singularity type $x^2 +\tau \1y^{2n}$ at $0$. 
\item \label{lmm:Examples 3}Singularity types of odd $y$-power: Let $n \in \N$. There are $c_1, \dots, c_{2n+2}>0$ with $\sum_{i=1}^{2n+2} c_i=1$ and distinct $a_1, \dots, a_{2n+2} \in \C_>$, such that the function $f-1$ with 
\[
f(\zeta):= \sum_{i=1}^{2n+2} c_i g_{a_i}(\zeta)
\]
has singularity type $x^2 +y^{2n+1}$ at $0$. 
\item \label{lmm:Examples 4}Singularity types of infinite order: There is a choice $r>0$ and $\alpha>0$ such that for $\nu=\frac{1}{2} \nu_{\mathrm{c}} + \frac{1}{2} \delta_{\ii \alpha}$, where $\nu_{\mathrm{c}}$ is the uniform distribution on the  circle $\ii\1\alpha + r \2S^1$, and $f$ from \eqref{def: f for deformed circular} the function $f-1$ has singularity type $x^2$ at $0$. 
\end{enumerate}
\end{lemma} 

Lemma~\ref{lmm:Examples exist}~\ref{lmm:Examples 1} is easily verified by a simple calculation. The other parts 
of Lemma~\ref{lmm:Examples exist} are proven in each of the following subsections, i.e. we prove 
Part~\ref{lmm:Examples 2} in Subsection~\ref{subsec:Singularity types of even order}, Part~\ref{lmm:Examples 3} in Subsection~\ref{subsec:Singularity types of odd order}  and Part~\ref{lmm:Examples 4} in Subsection~\ref{subsec:Singularity types of infinite order}. 
With these ingredients we now prove our main result about the existence of all singularity types.

\begin{proof}[\linkdest{proof_thr:all_singularities_appear}{Proof of Theorem~\ref{thr:all_singularities_appear}}] Since all examples from Lemma~\ref{lmm:Examples exist} have $\nu$ which is symmetric with respect to the imaginary axis, we can apply  Lemma~\ref{lmm:singularity types sigma f and beta} in each of these examples. To see the existence of a singularity of the type $\rm{Sing}_{1}^{\rm{int}}$, we take the example with $f(\zeta) :=f_a(\zeta)$ and $a:=\frac{1}{\sqrt{2}}(1+\ii)$ from Lemma~\ref{lmm:Examples exist}~\ref{lmm:Examples 1}. By Lemma~\ref{lmm:singularity types sigma f and beta} the singularity type of $f-1$ coincides with the singularity type of $-\beta$ and also with the singularity type of $\sigma$, showing $0 \in \rm{Sing}_{1}^{\rm{int}}$. To see an example with  $ 0 \in \rm{Sing}_{k}^{\rm{int}}$ and order $k \ge 2$, we consider Lemma~\ref{lmm:Examples exist}~\ref{lmm:Examples 2} with $\tau=+1$. Again by  Lemma~\ref{lmm:singularity types sigma f and beta} the singularity type of $f-1$ coincides with the singularity type of $\sigma$. 

Now we show that  $\rm{Sing}_{k}^{\rm{edge}}$ with order $k \in \N$ is not empty for some example. For the case $k=1$ the example is taken from Lemma~\ref{lmm:Examples exist}~\ref{lmm:Examples 1} with the choice  $\nu = \frac{1}{2}(\delta_1+\delta_{-1})$, for odd $k>1$ from  Lemma~\ref{lmm:Examples exist}~\ref{lmm:Examples 2} with $\tau=-1$ and for even $k>1$ from  Lemma~\ref{lmm:Examples exist}~\ref{lmm:Examples 3}. In each case  the singularity types of $f-1$ and $-\beta$ at $0$ coincide due to Lemma~\ref{lmm:singularity types sigma f and beta}. 

Finally  $0 \in \rm{Sing}_{\infty}^{\rm{int}}$ for the example from Lemma~\ref{lmm:Examples exist}~\ref{lmm:Examples 4} since here the singularity types of $f-1$ and $\sigma$ coincide by Lemma~\ref{lmm:singularity types sigma f and beta}. 
\end{proof}

\subsection{Singularity types of even $y$-power} \label{subsec:Singularity types of even order}
Here we prove Lemma~\ref{lmm:Examples exist}~\ref{lmm:Examples 2}.  Since $f_{c\1 a}(0) = c^2 f_{a}(0)$ for $f_{a}$ defined as in \eqref{Def of f_a} and $c>0$, it suffices to show that $f-f(0)$ has singularity type $x^2+\tau\1 y^{2n}$ at $0$, i.e. we can drop the condition $f(0)=1$. Furthermore, we can consider $f$ of the form  
\[
f (\zeta):= \sum_{i=1}^n c_i \frac{\re a_i}{|a_i|^2}f_{a_i}(\zeta)\,,
\]
with $a_i \in \C_{>}$, $c_i >0$ and show that this $f$ has singularity type $x^2 + \tau \1y^{2n}$ at $0$. This is sufficient since such $f$ can be transformed into a convex combinations of the $f_{a_i}$ by dividing by $\sum_{i=1}^n c_i \frac{\re a_i}{|a_i|^2}$ and without changing the singularity type. Properties of the function $f_a$ are summarised in Lemma~\ref{lem:properties_f_a}. 
The symmetries of $f_a$ (see \eqref{symmetries of f}) in particular imply that its Hessian at $\zeta=0$ is diagonal. Furthermore,  $f$ satisfies $\partial_y^{2k+1} f (0) =0$ for all $k \in \N_0$ due to \eqref{derivatives of fa}, where $\zeta=x+\ii \1y$. It remains to show that $a_i$ and $c_i$ can be chosen such that $\partial_y^{2k}f(0)=0$ for all $k=1, \dots, n-1$ and $\sign \partial_y^{2n}f(0)=\tau$ for a given $\tau \in \{\pm1\}$. Together with the symmetry of $f(\zeta) = f(-\ol{\zeta})$ around the imaginary axis  and since $\Delta f(0)>0$ this is enough to guarantee the local behaviour because of Corollary~\ref{crl:singularity with symmetry}.   By the formula \eqref{derivatives of fa} for $\partial_y^{2k}f_a(0)$ the choices $a_i$ and $c_i$ defining $f$ have to be made such that 
\[
 \sign \pbb{(-1)^n\re \sum_{i=1}^n c_i  a_i^{2n+1}} =\tau\,, \qquad \re\sum_{i=1}^n c_i   a_i^{2k+1} =0 \,, \qquad k=1, \dots, n-1\,.
\]
That such a choice exists is ensured by the following lemma, where we replaced $\tau$ by $(-1)^n\tau$ to avoid tracking the additional sign.

\begin{lemma}[Discrete truncated moment problem for even case]
For each $n \in \N$ and $\tau \in \{\pm 1\}$ there are positive numbers $c_1, \dots, c_n>0$ and distinct complex numbers  $z_1, \dots, z_n \in \C_{>}$  such that 
\[
\re \sum_{i=1}^n c_i z_i^{2n+1} = \tau\,, \qquad \re \sum_{i=1}^n c_i z_i^{2k+1} =0    \,, \qquad  \text{ for all } \; k=1, \dots, n-1\,.
\]
\end{lemma}

\begin{proof}
We perform the proof by induction over $n$. For $n=1$ we choose $z_1 \in \C$ with $\im z_1 >0$, $\re z_1 >0$ and $ \re z_1^3  =\tau $. Then we set $c_1:=1$. Now suppose for $n\in \N$ with $n \ge 2$ we have positive numbers $c_1, \ldots, 
c_n \in (0,\infty)$ and distinct complex numbers $\wh{z}_1, \ldots, \wh{z}_{n-1} \in \C_{>}$ such that 
\bels{hat z system even case}{
\re \sum_{i=1}^{n-1} {c}_i \wh{z}_i^{2n-1} = \pm 1\,, \qquad \re \sum_{i=1}^{n-1} {c}_i \wh{z}_i^{2k+1} =0    \,, \qquad k=1, \dots, n-2\,.
}
The construction works for either choice of sign in the induction hypothesis. We will cover both choices simultaneously.  
Since we can always normalise the coefficients, it suffices to show that there are $c_n >0$ and   $z_i \in \C_{>}$ in the first quadrant such that 
\bels{final choice of c and z}{
\sign \re \sum_{i=1}^n c_i z_i^{2n+1} = \tau\,, \qquad \re \sum_{i=1}^n c_i z_i^{2k+1} =0    \,, \qquad k=1, \dots, n-1\,.
}
To see that this is possible we pick $\omega \in \C$ with $\abs{\omega}=1$, $\re \omega>0$ and $\im \omega >0$ such that $\sign \re \omega^{2n-1} =\mp1 $   and 
$\sign \re \omega^{2n+1} =\tau$. Then we set 
\[
z_n(r) := \frac{ \omega}{r}\,, \qquad c_n(r):=\frac{r^{2n -1}}{ \abs{\re \omega^{2n-1}}}
\]
 for  $r>0$, as well as 
 \[
 F_k(w):= \sum_{i=1}^{n-1} c_i w_i^{2k+1} \,, \qquad G_k(w,r):= F_k(w)+r^{2n-2k-2}\frac{\omega^{2k+1}}{\abs{\re \omega^{2n-1}}}
 \]
 for $w=(w_1, \dots, w_{n-1})$, $k \in \db{n-1} $ and $r\ge 0$. In particular, $G_k(w,r):= F_k(w)+ c_n(r)z_n(r)^{2k+1}$ for $r>0$. For small enough $r>0$ we now solve the system 
 \[
 G_k(w,r) = G_k(\wh{z},0)\,, \qquad k \in \db{n-1}
 \]
 for $w=z(r)$. Since $\pt_{w_i} G_k(w) = \pt_{w_i} F_k(w)$ this system is locally uniquely solvable by the implicit function theorem due to
 \[
\det (\partial_{w_i} F_k(w))_{i,k=1}^{n-1}= \det( (2k+1)c_i w_i^{2k})_{i,k=1}^{n-1} = \prod_{k=1}^{n-1}(2k+1) \prod_{i=1}^{n-1}c_i w_i^2 \prod_{i<j}(w_j^2-w_i^2),
\] 
which does not vanish for $w=\wh{z}$, because by assumption the $\wh{z}_i$ are distinct. With $z_i =z_i(r)$ the system of equations  \eqref{final choice of c and z} is now satisfied. Indeed,  $\re G_k(w,r) = \re F_k(\wh{z},0)=0$ for $k \in \db{n-2}$ by \eqref{hat z system even case}, as well as $\re G_{n-1}(w,r) = \re G_{n-1}(\wh{z},0)=0$ by \eqref{hat z system even case} and the choice of the sign of $\re \omega^{2n-1}$. Finally 
\[
\sign \re \sum_{i=1}^{n} c_i z_i(r)^{2n+1}  = \sign \re \pbb{O(1) + \frac{1}{r^2}\frac{\re \omega^{2n+1}}{\abs{\re \omega^{2n-1}}} } =\tau
\]
by the choice of the sign of $\re \omega^{2n+1}$ and for sufficiently small $r>0$. For small $r>0$ the $z_i(r)$ with $i \in \db{n}$ are also distinct. 
\end{proof}

\subsection{Singularity types of odd $y$-power} \label{subsec:Singularity types of odd order}

Now we prove prove Lemma~\ref{lmm:Examples exist}~\ref{lmm:Examples 3}. As in the treatment of singularity types of even order, it suffices to show that $f-f(0)$ with $f$ of the form
\bels{f ansatz odd case}{
f(\zeta) := \sum_{i=1}^{2n+2} c_i  \frac{\re a_i}{|a_i|^2} g_{a_i}(\zeta) + \frac{1}{\abs{\zeta+ \ii }^2}
}
where $a_1, \dots, a_{2n+2} \in \C_{>}$  and $c_1, \dots, c_{2n+2}>0$ has singularity type $x^2 -y^{2n+1}$ at $0$. The function $g_a$ is defined in \eqref{eq:def_g_a} and its properties are summarised in Lemma~\ref{lem:properties_f_a}. 

To construct such singularity type we choose $f$ such that it satisfies
\bels{conditions on f odd case}{
\partial_y^{2n+1} f(0) < 0\,, \qquad \partial_y^{k} f(0) = 0\,, \qquad k=1, \dots, 2n
}
and apply Corollary~\ref{crl:singularity with symmetry}. The symmetry $f(\zeta) = f(-\ol{\zeta})$ follows from our ansatz \eqref{f ansatz odd case} and $\Delta f(0) >0$ because $\Delta g_a(0) = 4 \abs{a}^4 $ by Lemma~\ref{lem:properties_f_a} and $\Delta \abs{z+\ii}^{-2}|_{z = 0} = 4$. With
\[
\partial_y^{2k-1} \frac{1}{\abs{z+ \ii }^2}\bigg|_{z=0}=- (2k)!\,, \qquad \partial_y^{2k} \frac{1}{\abs{z+ \ii}^2}\bigg|_{z=0}= (2k+1)!\,,
\]
as well as the formulas for the derivatives of $g_a$ from \eqref{derivatives of ga} the conditions \eqref{conditions on f odd case} on our ansatz \eqref{f ansatz odd case} translate to 
\[
(-1)^{n}\im  \sum_{i=1}^n c_i a_i^{2(n+1)} <  2(n +1)\,, \quad(-1)^{k+1}\im  \sum_{i=1}^n c_i a_i^{2k} = 2\1k \,, \quad(-1)^{k+1} \re \sum_{i=1}^n c_i a_i^{2k+1} =2k+1    
\]
 for all $k \in \db{n}$.  
Such a choice of $c_i$ and $a_i$ exists due to the following lemma. 

\begin{lemma}[Discrete truncated moment problem for odd case]
For each $n \in \N$ there are positive numbers $c_1, \dots, c_{2n+2}>0$ and distinct complex numbers  $z_1, \dots, z_{2n+2} \in \C_{>}$ such that
\bels{odd case identities}{
 (-1)^{k+1}\im  \sum_{i=1}^{2n+2} c_i z_i^{2k} = 2\1k \,, \qquad(-1)^{k+1} \re \sum_{i=1}^{2n+2} c_i z_i^{2k+1} =2k+1    \,, \qquad k=1, \dots, n\,,
}
as well as  
\bels{odd case inequality}{
(-1)^{n}\im  \sum_{i=1}^{2n+2} c_i z_i^{2(n+1)} <  2(n +1)\,.
}
\end{lemma}

\begin{proof} We perform an induction over $n$. For the case $n=1$ we first find $z \in \C_{>}$  and $c>0$ such that 
\[
c\1 \im z^2 = 2\,, \qquad c\1\re z^3 = 3\,, \qquad -c\1 \im z^4 < 4\,.
\]
The choice $z = \frac{3}{4}\sqrt{2} \ee^{\ii \pi/12}$ and $c = \frac{32}{9}$ satisfies these conditions. 
Now we set $c_1:=c$ and fix three additional complex numbers $z_2,z_3,z_4 \in \C_{>}$ in the first quadrant that are all distinct and distinct from $z_1$.  We also set $c_i:=\delta$ for $i=2,3,4$, where $\delta>0$ is chosen sufficiently small. Then the relations  
\[
 \im  \sum_{i=1}^{4}  c_iz_i^{2} = 2 \,, \qquad  \re \sum_{i=1}^{4} c_iz_i^{3} =3    \,, \qquad -\im  \sum_{i=1}^{4}c_i  z_i^{4} <4
\]
with $z_1=z_1(\delta)$ are satisfied for $\delta=0$ with $z_1(0)=z$ and by the implicit function theorem can also be satisfied for small $\delta>0$. In fact the two identities have the form $c\im z_1^2 =2 +O(\delta)$ and $c\re z_1^3 =3 +O(\delta)$ and the map $z \mapsto (\im z^2, \re z^3)=(2xy, x^3 - xy^2)$ with $z=x +\ii\1 y$ has the non-vanishing Jacobian
\[ 
\det \begin{pmatrix} 2y & 2x \\ 3x^2 - 3y^2 & - 6xy\end{pmatrix} = -6x(x^2 + y^2) \,.
\]

For the  induction step we assume that $c_1, \ldots, c_{2n}>0$  and $\wh{z}_1, \ldots, \wh{z}_{2n} \in \C_{>}$ are given such that 
\bels{IH for odd moment problem}{
 (-1)^{k+1}\im \sum_{i=1}^{2n}c_i\1\wh{z}_i^{2k} =2 k\,, \quad (-1)^{k+1}\re \sum_{i=1}^{2n}c_i\1\wh{z}_i^{2k+1} = 2k+1\,, \quad 
 (-1)^{n-1}\im \sum_{i=1}^{2n}c_i \wh{z}_i^{2n}<2n
}
for all $k= 2, \dots, n-1$. 
With the short hand notation $w=(w_1, \dots, w_{2n})$ we define 
\[
F_k(w):= \sum_{i=1}^{2n} c_i w_i^{k} \,,\qquad k=2, \dots, 2n+2\,.
\]
We now  construct an additional $c_{2n+1}(r)= \alpha(r)\1 r^{2n}>0$  and $z_{2n+1}(r) = \frac{1}{r} \ee^{\ii \varphi(r)}$ with functions $\varphi=\varphi(r) \in (0,\pi/2)$  and $\alpha=\alpha(r)>0$ for $r>0$ that we choose appropriately such that 
\bels{I-Step for odd moment problem}{
{(-1)^{n+1}\pb{\im F_{2n}(\wh{z}) +  \alpha\1 \sin(2n\1\varphi)}}\!=\!2 n\,, \quad  {(-1)^{n+1}\pB{\re F_{2n+1}(\wh{z}) +  \frac{\alpha}{r} \cos((2n+1)\varphi)}}\!=\!2 n+1 \,,
}
as well as 
\bels{I-Step for odd moment problem inequality}{
{(-1)^{n}\pB{\im F_{2n+2}(\wh{z}) +  \frac{\alpha}{r^2} \sin((2n+2)\1\varphi)}}<2 n+1
}
are satisfied for sufficiently small $r>0$. 
To do so we define $\varphi_\ast:=\frac{\pi}{2} - \frac{1}{2n+1}\pi $. Then $\cos((2n+1)\varphi_*)=0 $
 and $\sin(2n \varphi_*) =\sin((2n+2) \varphi_*)$ because $2n \varphi_* + (2n+2) \varphi_* = \pi(2n-1) $ is an odd multiple of $\pi$. Furthermore,  we have the inequality  $(-1)^{n+1}\sin(2n \varphi_*)=\sin\pb{ \frac{2n}{2n+1}\pi }>0$. 
Now let 
\[
\sigma:= \sign \pb{(-1)^{n+1}{\re F_{2n+1}(\wh{z}) }-(2 n+1)}\,.
\]
We set 
$\varphi(r):= \varphi_\ast + \sigma \1 \beta(r)\1 r$ for some appropriately chosen $\beta(r)>0$ that is bounded and bounded away from zero.  In particular, we have the expansion $(-1)^{n+1}\cos((2n+1)\varphi(r)) =- \sigma (2n+1) r + O(r^2)$. The identities \eqref{I-Step for odd moment problem} become 
\bels{alpha beta equation 1}{
\alpha(r)= \frac{2n - (-1)^{n-1}\im F_{2n}(\wh{z}) }{(-1)^{n+1}\sin(2n\1\varphi(r))}
}
and 
\bels{alpha beta equation 2}{
(-1)^{n+1}\re F_{2n+1}(\wh{z})-(2 n+1) - \sigma (2n+1) \beta(r) \alpha(r)=
\begin{cases}
0 & \text{ if } \sigma=0\,,
\\
O(r)& \text{ if } \sigma \ne 0\,.
\end{cases}
}
For $\sigma=0$ there is no need to choose $\beta$ since $\varphi(r)= \varphi_*$ and we set $\alpha(r):= \alpha_*>0$, satisfying 
\bels{alpha star}{
\alpha_*= \frac{2n - (-1)^{n-1}\im F_{2n}(\wh{z}) }{(-1)^{n+1}\sin(2n\1\varphi_*)}\,,
}
in this case. 
That $\alpha_*$ is positive follows from the  choice of $\varphi_*$ and \eqref{IH for odd moment problem}. 
For $\sigma \ne 0$ we solve the system \eqref{alpha beta equation 1} and \eqref{alpha beta equation 2}. By the definition of $\sigma$, $\varphi_*$ and \eqref{IH for odd moment problem} there is a solution $\alpha(0) = \alpha_*>0$ and $\beta(0) = \beta_*>0$ for $r=0$. 
In fact, $\alpha_*$ is determined by \eqref{alpha star} and $\beta_*$ then by evaluating \eqref{alpha beta equation 2} on $r=0$. 
By the implicit function theorem this solution can be extended to sufficiently small $r>0$. 
Thus, \eqref{I-Step for odd moment problem} is satisfied for our choice of $\alpha(r)$ and $\varphi(r)$. Since $\alpha(0)=\alpha_*>0$ and $(-1)^n\sin((2n+2)\varphi(0)) = (-1)^n\sin (2n\varphi_*)<0$, by the choice of $\varphi_*$, the inequality \eqref{I-Step for odd moment problem inequality} is  also satisfied  for sufficiently small $r>0$.

Now we define 
\[
G_k(w,r):= F_k(w) +c_{2n+1}(r) z_{2n+1}(r)^k\,,\qquad k=2, \dots, 2n+2
\]
for sufficiently small $r>0$. By the choices of the functions $\varphi$ and $\alpha$ above, we can analytically extend  $G_k(w,r)$ to a neighbourhood of $r=0$ for $k =2, \dots, 2n+1$. 
By construction we have 
\[
(-1)^{n+1}\im G_{2n}(\wh{z},r) =2n \,, \qquad (-1)^{n+1}\re G_{2n+1}(\wh{z},r) = 2n+1\,, \qquad (-1)^n\im G_{2n+2}(\wh{z},r)< 2n+1
\]
and $G_{k+1}(w,0)= F_{k+1}(w)$ for $k\in \db{2n-2}$. Now we solve the system of equations $G_{k+1}({z}(r),r)= F_{k+1}(\wh{z})$ for $k \in \db{2n-2}$ together with $G_{2n}({z}(r),r)= G_{2n}(\wh{z},0)$ and $G_{2n+1}({z}(r),r)= G_{2n+1}(\wh{z},0)$ for a solution ${z}=({z}_1, \dots, {z}_{2n})$ with ${z}(0)=\wh{z}$. 
This is possible in a neighbourhood of $r=0$ by the implicit function theorem since 
\bels{system stability odd case}{
\det (\partial_{w_i} F_{k+1}(w))_{i,k=1}^{2n}= \det( (k+1)c_i w_i^{k})_{i,k=1}^{2n} = \prod_{k=1}^{2n}(k+1) \prod_{i=1}^{2n}c_i w_i \prod_{i<j}(w_j-w_i)\,,
}
which does not vanish when evaluated at $w = \wh{z}$. 

Altogether we have found distinct ${z}_1, \dots , z_{2n+1} \in \C_{>}$ and $c_{2n+1}>0$ such that \eqref{odd case identities} and \eqref{odd case inequality} are satisfied with $c_{2n+2}=0$. Finally we add $z_{2n+2} \in \C_{>}$ distinct from all ${z}_1, \dots , z_{2n+1}$. We set $z_*:=({z}_1, \dots , z_{2n})$  and  $c_{2n+2}=\delta$ and find a choice $z(\delta) \in \C_{>}^{2n}$  for $\delta\ge 0$ with $z(0)= z_*$ such that tis choice also satisfies \eqref{odd case identities} and \eqref{odd case inequality}. The choice of $c_{2n+1} $ and $z_{2n +1}$ remains independent of $\delta$. That such choice exists follows again from \eqref{system stability odd case} and the implicit function theorem.
\end{proof}

\subsection{Singularity types of infinite order} \label{subsec:Singularity types of infinite order}
In this subsection we  prove Lemma~\ref{lmm:Examples exist}~\ref{lmm:Examples 4}. As before it suffices to show that $f-f(0)$ is of singularity type $x^2$, i.e. we drop the condition $f(0)=1$. Furthermore, by shifting $\nu$ we can place the singularity at $\ii \in \rm{Sing}_\infty^{\rm{int}}$. We follow Example~\ref{example:x2} and choose $a$ as in that example. 
 A short computation using \eqref{beta for circular} reveals that $\mathbb{S}$ satisfies \eqref{eq:example_x2_mathbb_S}. 
Thus, since $\partial\mathbb{D}_1 \cap \mathbb{S} = \emptyset$ while points on both sides of $\partial \mathbb{D}_1$ 
belong to $\mathbb{S}$, we conclude from the analyticity of $\beta$ that $\partial_\zeta \beta(\zeta) = \partial_{\bar \zeta} \beta(\zeta) = 0$ for all $\zeta \in \partial \mathbb{D}_1$. 
Hence,  $\partial\mathbb{D}_1 \subset \rm{Sing}_\infty^{\rm{int}}$. 
Furthermore, using \eqref{eq:V_equations} and \eqref{sigma formula in terms of y}, we find that 
$\sigma$ is given by \eqref{eq:sigma_for_x2_example} in this example. 
 We refer to Example~\ref{example:x2} for more explanations and to Figure~\ref{fig:two} for a plot of the 
boundary of $\mathbb{S}$ and some sampled eigenvalues.

\section{Circular element with general deformation} 
\label{sec:circular_element_general_deformation}

In this section, we consider the Brown measure $\sigma$ of $a + \mathfrak c$, where $\mathfrak c$ 
is a standard circular element and $a$ is a general operator in a tracial von Neumann algebra that is $\ast$-free from $\mathfrak c$. We prove the classification of singular points  stated in Remark~\ref{rmk:general_deformation}.
Here we use the notation $|x|$ and $|x|_*$ for the positive semidefinite self-adjoint elements defined by 
$|x|^2 =xx^*$ and $|x|_*^2 = x^*x$, where $x$ is a general operator.
The Brown measure in this setup has a density which satisfies the  identity 
\bels{sigma kappa equation}{
\pi\2\sigma(\zeta)=\avgbb{\frac{1}{|a-\zeta|^2+\kappa^2}\frac{\kappa^2}{|a-\zeta|_*^2+\kappa^2}}+ \avgbb{\frac{1}{(|a-\zeta|^2+\kappa^2)^2}}^{-1}\absbb{\avgbb{\frac{1}{(|a-\zeta|^2+\kappa^2)^2}(a-\zeta)}}^2,
}
for all $\zeta \in \mathbb{S}$, which was derived earlier, see e.g.\ \cite[equation (16)]{Khoruzhenko1996}, \cite[Section 4.4]{BordenaveCaputoChafai2014}, \cite[Theorem~7.10]{BelinschiYinZhong2024} 
 and \cite[equation (4.25)]{ErdosJi2023}. 
Here $\kappa =\kappa(\zeta)\ge 0$ is determined by the well-known equation 
\bels{kappa equation}{
\avgbb{\frac{1}{|a-\zeta|^2+\kappa^2}}=1
}
for $\zeta \in \mathbb{S}$, see e.g.\ \cite[equation (17)]{Khoruzhenko1996}, \cite[Lemma~3.6]{Zhong2021}  and \cite[equation (4.19)]{ErdosJi2023}. 
 We note that $|\zeta- a|^2 \geq \norm{(a-\zeta)^{-1}}^{-2}$ in the sense of quadratic forms 
for all $\zeta \in \mathbb{C} \setminus \spec (a)$. 
 
 The boundary of $\supp \sigma$ is contained in $\partial \mathbb{S} =  \{ \zeta \in \C \colon f(\zeta) = 1\}$ with $f(\zeta) =\avg{|a - \zeta|^{-2}}$, i.e.\ in this setup $f(\zeta)-1$ plays the same role as $-\beta(\zeta)$ in our setup above. Since $\Delta f(\zeta) >0$ for any $\zeta \not \in \spec(a)$, the classification of edge points $\zeta_0 \in \partial \ol{\mathbb{S}}\cap \rm{Sing} = \bigcup_{n=2}^\infty\rm{Sing}^{\rm{edge}}_n$ from Remark~\ref{rmk:general_deformation} follows from
Lemma~\ref{lmm:possible singuarity types}. In particular, for such $\zeta_0$ the function  $f-1$ is of   singularity type $x^2 - y^n $ at $\zeta_0$.

To see that a point $\zeta_0 \in \partial\mathbb{S}$ in the interior of $\ol{\mathbb{S}}$ satisfies $\zeta_0 \in \rm{Sing}_{k}^{\rm{int}} \cup  \rm{Sing}_{\infty}^{\rm{int}}$, i.e.\ $\sigma$ is of singularity type $x^2 +y^{2k}$ or $x^2$ at $\zeta_0$, we realise that $\Delta \sigma (\zeta_0)>0$ at such $\zeta_0$ and use Lemma~\ref{lmm:possible singuarity types} again. 
Indeed, 
as $\zeta \mapsto \kappa(\zeta)$ is real analytic on $\mathbb{S}$ by \cite[Lemma~3.6]{Zhong2021}, 
differentiating \eqref{kappa equation} shows 
\bels{kappa derivative}{
\avgbb{\frac{1}{(|a-\zeta|^2+\kappa^2)^2}}\partial_\zeta \kappa^2=\avgbb{\frac{1}{(|a-\zeta|^2+\kappa^2)^2}(a-\zeta)^*}\,
}
for $\zeta \in \mathbb{S}$. 

By the implicit function theorem, we conclude from \eqref{kappa equation} and \eqref{eq:assumption_spec_a_subset_mathbb_S} 
that $\zeta\mapsto \kappa(\zeta)^2$ can be extended to a real analytic function in a neighbourhood of $\ol{\mathbb S}$. 
The nonnegativity of $\kappa^2$ on $\ol{\mathbb S}$ and $\kappa(\zeta_0) = 0$ imply that $\kappa^2$ has a local 
minimum at $\zeta_0$ and, thus, $\partial_\zeta \kappa^2(\zeta_0) = 0$. Hence, evaluating \eqref{kappa derivative} 
at $\zeta_0 \in \partial \mathbb{S}\setminus \partial \ol{\mathbb{S}}$ yields that the right hand side of \eqref{kappa derivative} vanishes. Therefore, differentiating \eqref{kappa derivative} with respect to $\ol{\zeta}$ shows   
\bes{
\avgbb{\frac{1}{|a-\zeta_0|^4}}\partial_\zeta \partial_{\ol{\zeta}}\kappa^2(\zeta_0)=\avgbb{\frac{1}{|a-\zeta_0|^2}\frac{1}{|a-\zeta_0|_*^2}}
}
and, thus, $\Delta\kappa^2(\zeta_0) >0$. 
Moreover, differentiating \eqref{sigma kappa equation} with respect to $\zeta$ and $\ol{\zeta}$ and using $\partial_\zeta \kappa^2(\zeta_0)=0=\kappa^2(\zeta_0)$, as well as \eqref{kappa derivative} shows 
\[
\pi\partial_\zeta \partial_{\ol{\zeta}}\sigma (\zeta_0)
=\avgbb{\frac{1}{|a-\zeta_0|^2}\frac{1}{|a-\zeta_0|_*^2}}\partial_\zeta \partial_{\ol{\zeta}}\kappa^2 (\zeta_0)
+\avgbb{\frac{1}{|a-\zeta|^4}}^{-1}\absbb{\partial_\zeta\avgbb{\frac{1}{(|a-\zeta|^2+\kappa^2)^2}(a-\zeta)}}^2_{\zeta=\zeta_0}. 
\] 
Hence, $\Delta \sigma(\zeta_0) = 4 \partial_\zeta \partial_{\bar \zeta} \sigma(\zeta_0) > 0$ 
as $\Delta \kappa^2(\zeta_0) >0$.

\appendix 

\section{Proof of higher order Morse lemma} 
\label{app:proof_higher_order_morse_lemma} 

\begin{proof}[\linkdest{proof:lmm:possible singularity types}{Proof of Lemma~\ref{lmm:possible singuarity types}} ]

Without loss of generality we assume  $\zeta_0=0$ and $f(0)=0$. We denote $x=\re \zeta$ and $y =\im \zeta$. 
Furthermore, by rotating the coordinates $(x,y)$ so that they coincide with  the eigendirections  of the Hessian $\rm{Hess}(f)$ of $f$, we may assume $f(x,y) = \alpha \1x^2+ \beta \1 y^2 + O(\abs{x}^3 + \abs{y}^3) $, where $2\alpha$ and $2\beta$ are the eigenvalues of $\rm{Hess}(f)$. Here we used $f(0)=0$, $\partial f (0)=0$. Since $\Delta f(0)>0$, we may assume $\alpha>0$ by potentially exchanging the roles of $x$ and $y$. By rescaling we can also assume $\alpha =1$.
 In case $\beta\ne 0$ we use the Morse lemma (see e.g.\ \cite[Lemma~2.2]{MilnorMorseTheoryBook})  
to bring $\wt{f}:= f \circ \Phi$  into the form $\wt{f} = x^2 \pm y^2$. 

For the remainder of the proof, we assume $\beta = 0$ and tacitly shrink the neighbourhood of $0$ whenever necessary.
In the case when $\beta =0$ we see that 
\bels{f representation}{
f(x,y) = x^2(1 +x g(x,y) + y h(y)) + x \1y^2 j(y) + y^3 k(y)\,,
}
where $g,h,j,k$ are real analytic functions. Now we define a local real analytic  diffeomorphism through the ansatz $\Phi_1(x,y):=(x + y^2 \psi_1(y),y )$, where the real analytic function $\psi_1$ will be chosen appropriately below. 
Then $f_1:= f \circ \Phi_1$ has the form
\[
f_1(x,y) = x^2(1 +x g_1(x,y) + y h_1(y)) + x \1y^2 j_1(y) + y^3 k_1(y)
\] 
with 
\begin{align*} 
g_1(x,y) & = g(x+y^2 \psi_1(y),y) +  3y^2 \psi_1(y)  (\pt_1 g)(y^2 \psi_1(y),y) + 3y^2( \psi_1(y) x + y^2 \psi_1(y)^2)(\pt_1^2g)(y^2 \psi_1(y),y) /2 
\\ 
& \phantom{=} + ( 3y^2 \psi_1(y) x^2 + 3y^4 \psi_1(y)^2 x + y^6 \psi_1(y)^3 ) \wt{g}(x,y),  
\\ 
h_1(y) & = h(y) + 3y \psi_1(y) g(y^2 \psi_1(y),y) + 3y^3 \psi_1(y)^2 (\pt_1 g)(y^2 \psi_1(y),y) + y^5\psi_1(y)^3 (\pt_1^2g)(y^2 \psi_1(y),y)/2, 
\\
 j_1(y)&=  j(y) +\psi_1(y) (2 + 2 y h(y) + 3 y^2 g(y^2 \psi_1(y),y)\psi_1(y)+ y^4 \psi_1(y)^2\partial_1 g(y^2 \psi_1(y),y))\,, 
 \\
  k_1(y)&=k(y)+ y\psi_1(y)(j(y)+\psi_1(y)(1+yh(y)+y^2g(y^2\psi_1(y),y)\psi_1(y)))\,,
\end{align*} 
where $\partial_1 $ denotes the partial derivative  with respect to the first argument  and $\wt{g}$ is defined through the expansion
\[ g(x+y^2\psi_1(y),y) = g(y^2\psi_1(y),y) + x (\pt_1 g)(y^2\psi_1(y),y) 
+ x^2 (\pt_1^2g)(y^2 \psi_1(y),y)/2 + x^3 \wt{g}(x,y). \]
By the implicit function theorem, we find a real analytic function $\psi_1(y)$ such that $j_1(y) = 0$ for all small enough $y$. 
In particular, $\psi_1(y) =  (-\frac{1}{2}+O(y))j(y)$. 

Hence, $f_1$ simplifies to 
\[
f_1(x,y) = x^2(1 +x g_1(x,y) + y h_1(y)) + y^3 k_1(y)\,,
\] 
where $k_1(y)= \tau y^{K-3}\psi_2(y)$ for some $\tau \in \{0, \pm1\}$, $K \ge 3$ and a real analytic function $\psi_2$ with $\psi_2(0)>0$. 
We define the local real analytic diffeomorphism $\Phi_2(x,y ):=(x, y\psi_2(y)^{1/K})$.
Then $f_2:= f \circ \Phi_2^{-1}$ 
is of the form
\[
f_2(x,y) = x^2(1 +x g_2(x,y) + y h_2(y)) + \tau \1 y^K
\] 
 for some real analytic functions $g_2$ and $h_2$.  
Finally, we choose the  local real analytic diffeomorphism $\Phi_3(x,y ):=(x(1 +x g_2(x,y) + y h_2(y))^{1/2} , y)$ such that $\wt{f}:= f_2 \circ \Phi_3^{-1}$ 
 is of the form claimed in the lemma. The diffeomorphism from the statement of the lemma is  $\Phi := \Phi_1 \circ \Phi_2^{-1} \circ \Phi_3^{-1}$. 
\end{proof}

\begin{proof}[Proof of Corollary~\ref{crl:singularity with symmetry}]
As in the proof of Lemma~\ref{lmm:possible singuarity types} the case $\beta \ne 0$ and $\alpha \ne 0$ is clear, where $2\alpha$ and $2\beta$ are the two eigenvalues of $\rm{Hess}(f)$. Thus, by rescaling it suffices to consider $\beta = 0$ and $\alpha = 1$. As $f$ is symmetric with respect to the imaginary 
axis,  we have the representation
\[
f(x,y) = f(0)+x^2 p(x^2,y) + y^3k(y)\,,
\]
for real analytic functions $p$ and $k$ with  $p(0,0)=1$. 
In particular, $j(y)=0$ in the representation \eqref{f representation}. Hence, from the proof of Lemma~\ref{lmm:possible singuarity types} we know that 
$\psi_1 \equiv 0$ and $k(y)=k_1(y) =\tau y^{K-3}\psi_2(y)$ with some  real analytic $\psi_2$ such that $\psi_2(0)>0$. Choosing the diffeomorphism 
$
\Phi^{-1}(x,y):= (x \sqrt{p(x^2,y)}, y \2\psi_2(y)^{1/K})
$ shows the claim. 
\end{proof}

\section{Stability of Dyson equation}

\label{app:stability_operator}

Let $M$ be the solution of \eqref{eq:mde}. The stability operator of \eqref{eq:mde} is given by 
\eqref{eq:def_stability_operator}. 
\begin{lemma} \label{lem:stability_eta_positive}
Let $a\in \mathcal B$, $s \colon \mathfrak X \times \mathfrak X \to [0,\infty)$  a bounded measurable function and $\cal L$ defined as in  \eqref{eq:def_stability_operator}. Then $\cal L$ is invertible for any $\eta >0$ and satisfies 
\begin{equation} \label{eq:stability_operator_inverse_bound} 
\norm{\cal L^{-1}}_{\infty}\le \frac{1}{\eta^2}\,.
\end{equation} 
\end{lemma} 

\begin{proof} 
We write $\cal L$ as a directional derivative of $M:=M_R(\zeta,\ii \eta)$, where 
\[
M_R(\zeta,\ii \eta):= E\bigg[ \bigg(\mtwo{-\ii \eta & Y}{Y^* & -\ii \eta}-R\bigg)^{-1} \bigg], \qquad Y:= a + \frak c -\zeta\,.
\]
We note that for $\eta >0$ the map $R \mapsto M_R(\zeta,\ii \eta)$ is differentiable in a neighborhood of $R=0$. 
 Let $\wt{R} \in \mathcal B^{2\times 2}$ be arbitrary.  
Taking 
  the directional derivative $\nabla_{\wt R}$ with respect to $R$ in the direction $\wt R$ yields
 \begin{equation} \label{eq:Linv representation}
\nabla_{\wt R} M = E \bigg[ \mtwo{-\ii \eta & Y}{Y^* & -\ii \eta}^{-1}\wt R \mtwo{-\ii \eta & Y}{Y^* & -\ii \eta}^{-1}\bigg]\,.
\end{equation}
On the other hand 
 $M$ satisfies a modified version of the Dyson equation \eqref{general MDE}, namely 
\[
-\frac{1}{M} = R+\begin{pmatrix} \ii \eta & \zeta - a \\  \ol{\zeta - a} & \ii \eta \end{pmatrix} + \Sigma[M ]\,.
\]
Taking the directional derivative $\nabla_{\wt R}$ of this equation shows $\cal L \nabla_{\wt R} M =\wt R$. 
Since $\wt{R}$ was arbitrary, we conclude that $\cal L$ is invertible and $\cal L^{-1}\wt R = \nabla_{\wt R} M$. 
The bound \eqref{eq:stability_operator_inverse_bound} now follows from the representation \eqref{eq:Linv representation}. 
\end{proof}

\begin{proof}[\linkdest{proof:lmm:resolvent bound for L}{Proof of Lemma~\ref{lmm:resolvent bound for L}}]
We recall that $\zeta \in \C$ is fixed such that $\limsup_{\eta \downarrow 0} \avg{v_1(\zeta,\eta)} \ge \delta  $ for some $\delta>0$ and that $v^{(n)}=v(\zeta, \eta_n) \to v_0$ weakly in $(L^2)^2$, where $v_0 \sim_\delta 1$.  
First we use the identities
\[
\scr Lv_- = -\eta  \frac{\tau\2v^2}{(\eta + S_d v)^2}\,, \qquad \scr L^*(e_-(\eta + S_ov)) = \eta \1 e_-
\]
with $v_-=v e_-$ and $v=v^{(n)}$. 
In the limit $\eta \downarrow  0$ we see $\scr L_0v_-=0$  and $\scr L_0^*S_ov_-=0$.  
Here we used that $v = v_0$ satisfies the Dyson equation, \eqref{eq:V_equations} at $\eta =0$.   
For the rest of this proof we drop the $0$-index from our notation. We introduce $T$, $V$, $F: \cal B^2 \to \cal B^2$ as in \eqref{def of F,T,V}, evaluated at $v=v_0$ and $\eta =0$.
In terms of $T,F$ and $V$ we obtain
\begin{equation} \label{eq:L_in_terms_of_V_T_F} 
\scr L = V^{-1}(1-TF)V\,.
\end{equation} 
We consider the natural extensions $F\colon (L^2)^2 \to (L^2)^2$ and $T\colon (L^2)^2 \to (L^2)^2$. These operators are self-adjoint  
and we use their spectral properties presented in Lemma~\ref{lmm:Spectral properties of F and T} below. The proof of this lemma is postponed to after the proof of Lemma~\ref{lmm:resolvent bound for L}.

\begin{lemma}[Spectral properties of $F$ and $T$] \label{lmm:Spectral properties of F and T}
The Hermitian operator $F\colon (L^2)^2 \to (L^2)^2$ satisfies the following properties:
\begin{enumerate}[label=(\roman*)] 
\item $F$ has non-degenerate isolated eigenvalues at $\pm 1$ and a spectral gap $\eps\sim_\delta1$, i.e.\  
\begin{equation}\label{eq:F spectral gap}
\spec(F) \subset \{-1\} \cup [-1+\eps, 1-\eps]\cup \{1\}\,.
\end{equation}
\item 
The eigenvectors corresponding to the eigenvalues $\pm 1$ are  
\begin{equation}\label{eq:F eigenvectors}
F Vv = Vv \,, \qquad  F Vv_- = -V v_-\,,
\end{equation}
where $v=(v_1, v_2)$ and $v_- =(v_1, -v_2)$. 
\end{enumerate}
The Hermitian operator $T\colon (L^2)^2 \to (L^2)^2$ satisfies the following properties:
\begin{enumerate}[label=(\roman*)] 
\setcounter{enumi}{2} 
\item The spectrum is bounded away from $1$ by a gap of size $\eps \sim_\delta 1$, i.e.\  
\[
\spec(T) \subset [-1,1-\eps]\,.
\]
\item \label{item 4} For $x \in \{(y,-y): y \in L^2\} $ we have $Tx=-x$. 
\end{enumerate}
\end{lemma}

 We now prove Lemma~\ref{lmm:resolvent bound for L} by tacitly using the properties of $F$ and $T$ from Lemma~\ref{lmm:Spectral properties of F and T}.  
Since $f_-:=Vv_- \in \{(y,-y): y \in \cal B\}$ we see that $F$ and $T$ both leave the subspace $(f_-)^\perp$ invariant.  
Now,  using \eqref{eq:L_in_terms_of_V_T_F},  we rewrite the resolvent of $\scr L$ as
\[
\frac{1}{\scr L-\omega} = V^{-1}\frac{1}{1 - TF-\omega} V\,.
\]
The operator $V$ and its inverse satisfy the bound $\norm{V}_\#+\norm{V^{-1}}_\# \lesssim_\delta 1$ for $\# = 2, \infty$, where we used $\delta \lesssim v \lesssim \frac{1}{\delta}$ from Lemma~\ref{lmm:subsequence}.
Thus, it suffices to show \eqref{resolvent bound for L} with $\scr L$ replaced by $1-TF$. 
Furthermore, we can restrict to the case $\# =2$, since \cite[Lemma~4.5]{AK_Corr_circ} is applicable  because  
\[
\norm{TF}_\infty + \norm{TF}_{\infty \to 2} \norm{TF}_{2 \to \infty}  \lesssim_\delta 1\,.
\]
 Here, $\norm{TF}_{\infty \to 2}$ and $\norm{TF}_{2\to \infty}$ denote the operator norms of 
$TF$ viewed as operator $\mathcal B^2 \to (L^2)^2$ and $(L^2)^2 \to \mathcal B^2$, respectively.  

Since $\norm{TF}_2 \le 1$ we have the bound $\norm{(1-TF-\omega)^{-1}}_2\lesssim_{\eps, \delta} 1$  for any $\omega \not \in 1+\mathbb{D}_{1+\eps}$  and any $\eps >0$. Now we can decompose 
\bels{split of TF resolvent}{
\frac{1}{1-TF-\omega} = (1-P)\frac{1}{1-TF-\omega}(1-P) - \frac{1}{\omega}P\,,
}
where $P$ is the orthogonal projection onto the span of $f_-$. Provided $\eps>0$ is chosen sufficiently small, the first summand in \eqref{split of TF resolvent} is bounded for $\omega \in \mathbb{D}_{2\eps}$. Indeed $1-TF$ has a bounded inverse on $f_-^\perp$ by applying 
Lemma~\ref{lem:rotation_inversion} with $H = (L^2)^2$. 
Finally, the second summand  in \eqref{split of TF resolvent} is bounded on $\C\setminus \mathbb{D}_\eps$.   Altogether the bound \eqref{resolvent bound for L} is proven. 
The non-degeneracy of the eigenvector $v_-$ of $\scr L$ also follows from the decomposition \eqref{split of TF resolvent}. 
\end{proof}

\begin{proof}[Proof of Lemma~\ref{lmm:Spectral properties of F and T}] 
The identities \eqref{eq:F eigenvectors} are verified by using the definitions of $V$ and $F$ from  \eqref{def of F,T,V}. That the eigenvalues $\pm 1$ are non-degenerate and the existence of the spectral gap in \eqref{eq:F spectral gap} are seen by studying 
\[
F^2 = \mtwo{ RR^*& 0}{0 & R^*R}\,, \qquad  R:= D_{\frac{v_1}{\hat{v}}}SD_{\frac{v_2}{\hat{v}}} \,.
\]
As we now show, both $RR^*$ and $R^*R$ have spectrum contained in $[0,1-\eps]\cup \{1\}$ for some $\eps\sim_\delta 1$, where $1$ is a non-degenerate eigenvalue. We show this for $RR^*$. The argument for $R^*R$ is analogous. The operator $RR^*: \cal{B}\to \cal{B}$ has a kernel representation
\[
(RR^*u)(x)= \int_{\mathfrak X} t_1(x,y) u(y) \mu(\dd y)
\]
with a symmetric non-negative kernel 
 \[
t_1(x,y) = \frac{v_1(x)}{\wh v (x)}\int_{\mathfrak X} s(x,q)\frac{v_2(q)^2}{\wh v (q)^2}s(y,q)\mu (\dd q)\frac{v_1(y)}{\wh v (y)}
\]
 because of \eqref{eq:def_operators_S_and_S_star}. Since $v = v_0 \sim_\delta 1$ in our regime due to Lemma~\ref{lmm:subsequence}, assumption \ref{assum:flatness} implies
\begin{equation}\label{eq:scaling t1}
t_1(x,y )\sim_\delta \sum_{i,j,l=1}^K z_{il}z_{jl} \bbm{1}_{I_i \times I_j}(x,y) \,.
\end{equation}
Now let $t_k$ be the kernel of $(RR^*)^k$ for any $k \in \N$. 
Since the matrix $Z=(z_{ij})_{i,j=1}^K$ is primitive with positive diagonal there is a $k \in \N$ such that $t_k (x,y)\sim_\delta 1$ for  $\mu$-almost all $x$, $y \in \mathfrak X$. 
Indeed, by \eqref{eq:scaling t1} we have
\[
t_k(x,y )\sim_\delta \sum_{i,j}^K ((ZZ^*)^k)_{ij} \bbm{1}_{I_i \times I_j}(x,y) \,,
\]
where the entries of $(ZZ^*)^k$ are all positive for large enough $k \in \N$ due to the primitivity of $Z$ and its positive diagonal. 
 By \cite[Lemma~5.7]{Ajanki_QVE_Memoirs}  with the necessary eigenfunction from \eqref{eq:F eigenvectors},  we conclude that $(RR^*)^k$ has a spectral gap  $\sim_\delta 1$  and therefore so does $RR^*$. A similar argument to the one presented here has been used in \cite[Lemma 3.3]{AltGram}. 

Now we show the spectral properties of $T$. The proof follows the strategy from  \cite[Lemma~3.6]{Altcirc}.
Since the argument is short we present it here for the convenience of the reader. The property \ref{item 4} is obvious from the definition of $T$ in \eqref{def of F,T,V}.
It remains to see that $\spec(T)\subset [-1,1-\eps]$. Since restricted to vectors $(y,-y)$, the operator $T$ is $-1$ it suffices to consider the restriction of $T$ to vectors $(y,y)$. Here we find $T(y,y)=(D_p y,D_py)$  with $p:=|a-\zeta|^2 \frac{v_1v_2}{\wh{v}^2}-\wh{v}^2$. Recalling the definition of $\wh{v}$ in \eqref{def: v} we see that 
\[
2\frac{|a-\zeta|^2v_2}{Sv_2}-1=p = 1- 2v_2S^{*}v_1 ,
\]
where we used the Dyson equation \eqref{eq:V_equations} at $\eta=0$. Since $v_i \sim_\delta 1 $ this shows the claim about the spectrum of $T$. 
\end{proof}

\begin{lemma}[Contraction-Inversion Lemma] \label{lem:rotation_inversion}
Let $\eps >0$ 
and $T, F $ be two bounded self-adjoint operators on a Hilbert space $H$ such that $ \norm{T} \leq 1$ and $ \norm{F} \leq 1$.
Suppose that there are normalized vectors $f_\pm \in H$ satisfying 
\begin{equation} \label{eq:properties_F}
F f_+ = \norm{F} f_+, \quad F f_- = - \norm{F} f_-, \quad \norm{F x} \leq (1 - \eps) \norm{x}
\end{equation}
for all $x \in H$ such that $x \perp \linspan\{f_+, f_-\}$. 

Furthermore, assume that 
\begin{equation} \label{eq:norm_Tf_plus_Tf_minus}
 \scalar{f_+}{T f_+} \leq 1-\eps, \quad Tf_-=-f_-. 
\end{equation}
Then $1-TF$ maps $f_-^\perp$ into itself and its restriction to $f_-^\perp$ is invertible such that   there is a constant $c >0$, depending only on  $\eps$, 
 with 
\begin{equation} 
 \norm{(1-TF)x}  \ge c \norm{x}. \label{eq:estimate_stability_lemma}
\end{equation}
for all $x \perp f_-$. 
\end{lemma}

\begin{proof} First we see that $TF$ maps $f_-^\perp$ into itself. Indeed, let $h \in H$ with $\norm{h}=1$ and $h \perp f_-$. Then $\scalar{f_-}{TF h} =-\scalar{f_-}{F h}=0$ because of \eqref{eq:norm_Tf_plus_Tf_minus}. Now we write 
$h= \beta f_+ + \gamma x $ for some normalized $x \perp \linspan\{f_+, f_-\}$ and $ \beta, \gamma \in \C$ with $ |\beta |^2 +|\gamma|^2 =1$. 
From here on we follow the proof of \cite[Lemma~3.7]{Altcirc} with the notational identifications ${\bf A} \leftrightarrow T$ and ${\bf B} \leftrightarrow F$.
\end{proof}

\begin{lemma}[Twist lemma]
\label{lmm:Twist lemma}
Let $H$ be a Hilbert space with a scalar product $\scalar{\2\cdot\2}{\2\cdot\2}$ and equipped with a norm $\norm{\2\cdot\2}_{\#}$ (not necessarily induced by the scalar product), $\eps \in (0,1)$ and $A $ a bounded linear operator on $(H,\scalar{\2\cdot\2}{\2\cdot\2})$ such that $\ol{\D}_\eps \cap \spec A = \{ \alpha \}$. 
We assume that $\alpha$ is a non-degenerate eigenvalue of $A$ and $A x = \alpha x$ for 
some $x \in H$ with $\norm{x}_\# = 1$. 
Let
\bels{Definition of projection P}{
P\,:=\,- \frac{1}{2\pi \ii}\oint_{\partial \D_{\eps}} \frac{\dd \zeta}{A-\zeta}\,=\, \scalar{p}{\2\cdot\2}\2x\,,
}
with some $p \in H$
be the corresponding spectral projection and $y \in H$ a vector such that
\bels{assumptions on orthogonality for a and b}{
\abs{\scalar{x}{y}}\,\ge\, 2\2 \eps\,,\qquad
\abs{\scalar{y}{w}}\,\le\, \norm{w}_{\#}\,,\qquad \forall\;w \in H\,.
}
Suppose that $A$ has a bounded inverse on the range of $1-P$ with respect to the norm $\norm{\2\cdot\2}_{\#}$, i.e.\
\bels{assumptions on lower bound of A}{
\norm{Aw}_{\#}\,\ge\, \norm{w}_{\#}\,,\qquad \forall\; w \perp p\,.
}
Then $A$ has a bounded inverse when restricted to $y^\perp$, namely
\bels{Bound on A inverse on b perp}{
\norm{Aw}_{\#}\,\ge\, \frac{\eps}{3}\2\norm{w}_{\#}\,,\qquad \forall\; w \perp y\,.
}
\end{lemma}
\begin{proof} The proof follows the one of \cite[Lemma~4.6]{AK_Corr_circ}  line by line, replacing $\C^d$ by $H$. 
\end{proof}

\section{Auxiliary results}

\begin{lemma} \label{lmm:scaling of cubic}
Let $y> 0$ be the unique solution to the equation $y^3 + \beta \1 y = x$ for  $x> 0$ and $\beta \in \R$. Then 
\[
y \sim \sqrt{\max\{0,-\beta\}} + \frac{x}{x^{2/3}+ \abs{\beta}}\,.
\]
\end{lemma}

\begin{proof}
First we consider the case $\beta \ge 0$. 
Then clearly 
\[
y \sim \frac{x}{x^{2/3} + \abs{\beta}}\,.
\]
Now let  $\beta <0 $,   then we must have $y = \sqrt{-\beta}(1 + \eps)$ for some $\eps > 0$ since $y^2 + \beta \1 y >0$.  For this $\eps$ we get the equation
$
\eps^3 + 3 \eps^2 + 2 \eps = x\1 \abs{\beta}^{-3/2}\,.
$
Thus, we have the scaling 
\[
\eps \sim \min\cbb{\frac{x}{\abs{\beta}^{3/2}}, \frac{x^{1/3}}{\abs{\beta}^{1/2}}}\, . 
\]
Therefore we conclude 
\[
y = \sqrt{\abs{\beta}} +  \min\cbb{\frac{x}{\abs{\beta}}, x^{1/3}} \sim \sqrt{\abs{\beta}}+\frac{x}{x^{2/3} + \abs{\beta}}\,,
\]
which is the claim of the lemma. 
\end{proof}

\begin{lemma}\label{lmm:resolvent bound for S}
Let $S\colon \cal{B} \to \cal{B}$ be an integral operator as in \eqref{eq:def_operators_S_and_S_star} with a kernel $s: \mathfrak X^2 \to (0,\infty)$ that satisfies the bounds $\eps \le s(x,y)\le \frac{1}{\eps}$ for 
all $x$, $y \in \mathfrak{X}$ and some  constant $\eps>0$ and such that the spectral radius of $S$ is normalised to $\varrho(S)=1$. 
Then there are constants $\delta,C >0$, depending only on $\eps$, such that 
\[
\sup \Big\{ \norm{(S-z)^{-1}}_2 \colon z \not \in D_{1-\delta}(0) \cup D_\delta(1)\Big\} \le C
\]
and $\spec(S) \cap D_\delta(1) =\{1\} $ is non-degenerate. 
\end{lemma}

\begin{proof}
We follow line by line the arguments from the proof of \cite[Lemma~A.1]{doi:10.1137/17M1143125}, where the finite dimensional case with $\abs{\frak{X}} < \infty$ and $\mu$ the counting measure is carried out.  
\end{proof}

\begin{lemma}[Properties of $f_a$ and $g_a$] \label{lem:properties_f_a} 
Let $a \in \C_{>}$, as well as $f_a \colon \C \setminus\{ 1/a, 1/\ol{a}, -1/a, -1/\ol{a} \} \to \R$ and $g_a\colon \C \setminus\{ -1/a, 1/\ol{a} \} \to \R$ be
 defined as in  \eqref{Def of f_a} and \eqref{eq:def_g_a}, respectively.  
Then $f_a$ and $g_a$ are real analytic.  
Moreover, the Laplacian and derivatives of $f_a$ along the imaginary axis are 
\bels{derivatives of fa}{
 \Delta f_a(0) = 4 \abs{a}^4, \qquad \partial_y^{2k+1} f_a(0) = 0, \qquad \partial_y^{2k}f_a(0)=(-1)^k(2k)! |a|^2\frac{\re a^{2k+1} }{\re a}
 }
for all $k \in \N_0$, where $\zeta=x+\ii\1 y$. 
Furthermore, $f_a$ has the symmetries 
\bels{symmetries of f}{ 
f_a(\zeta) = f_a(-\zeta)= f_a(\ol{\zeta}) = f_a(-\ol{\zeta}) \,.
} 
The Laplacian and derivatives of $g_a$ along the imaginary axis are 
\bels{derivatives of ga}{ \Delta g_a(0) \!=\! 4 \abs{a}^4, \; \partial_y^{2k-1} g_a(0) \!=\! (-1)^{k+1}(2k-1)! |a|^2\frac{\im a^{2k}}{\re a}\,, \; \partial_y^{2k} g_a(0) \!=\!(-1)^k(2k)! |a|^2\frac{\re a^{2k+1}}{\re a}
}
and $g_a$ has the symmetry $g_a(\zeta) = g_a(-\ol{\zeta})$. 
\end{lemma} 

\begin{proof} 
To verify the formulas for the derivatives of  $f_a$ and $g_a$  we first compute
\[
\partial_\zeta^k \frac{1}{\abs{\zeta}^2} = \frac{(-1)^k k! }{\ol{\zeta} \zeta^{k+1}} \,.
\]
From this we conclude 
\bes{
\partial_y^k \frac{1}{\abs{\zeta}^2} &= \pb{\ii (\partial_\zeta-\partial_{\ol{\zeta}})}^k \frac{1}{\abs{\zeta}^2} 
= \ii^k\sum_{l=0}^k \binom{k}{l}(-1)^l\partial_\zeta^{k-l}\partial_{\ol{\zeta}}^l\frac{1}{\abs{\zeta}^2}
\\
&= (-\ii)^kk!\sum_{l=0}^k\frac{(-1)^l}{\ol{\zeta}^{l+1}\zeta^{k-l+1}}
= \frac{(-\ii)^kk!}{\abs{\zeta}^2\zeta^k}\frac{1-(-\frac{\zeta}{\ol{\zeta}})^{k+1}}{1+\frac{\zeta}{\ol{\zeta}}}
= \frac{(-\ii)^kk!}{\abs{\zeta}^{2k+2}}\frac{\ol{\zeta}^{k+1}+(-1)^k \zeta^{k+1}}{\zeta+\ol{\zeta}}\,,
}
where $\zeta= x+\ii y$. 
For the cases of odd and even numbers of derivatives we get 
\bes{
\partial_y^{2k+1} \frac{1}{\abs{\zeta}^2} = (-1)^{k+1}\frac{(2k+1)!}{\abs{\zeta}^{4k+4}}\frac{\im \zeta^{2k+2}}{\re \zeta}\,, \qquad  \partial_y^{2k} \frac{1}{\abs{\zeta}^2} =
\frac{(-1)^k(2k)!}{\abs{\zeta}^{4k+2}}\frac{ \re \zeta^{2k+1}}{\re \zeta}
}
for $k \in \N_0$. Now we use the definition of $f_a$ in \eqref{Def of f_a} and of $g_a$ in \eqref{eq:def_g_a}, take the derivative and evaluate at $\zeta=0$ to get \eqref{derivatives of fa} and \eqref{derivatives of ga}, respectively. 
\end{proof}

\paragraph{Acknowledgments}
We thank the anonymous referee for additional references and many helpful comments
that simplified several proofs.

\providecommand{\bysame}{\leavevmode\hbox to3em{\hrulefill}\thinspace}
\providecommand{\MR}{\relax\ifhmode\unskip\space\fi MR }
\providecommand{\MRhref}[2]{%
  \href{http://www.ams.org/mathscinet-getitem?mr=#1}{#2}
}
\providecommand{\href}[2]{#2}


\begin{thebibliography}{10}

\bibitem{AM}
M.~Adler and P.~van Moerbeke, \emph{{PDEs for the Gaussian ensemble with
  external source and the Pearcey distribution}}, Comm. Pure Appl. Math.
  \textbf{60} (2007), no.~9, 1261--1292.

\bibitem{Ajanki_QVE_Memoirs}
O.~H. Ajanki, L.~Erd\H{o}s, and T.~Kr\"{u}ger, \emph{Quadratic vector equations
  on complex upper half-plane}, Mem. Amer. Math. Soc. \textbf{261} (2019),
  no.~1261, v+133. \MR{4031100}

\bibitem{AjankiCorrelated}
O.~H. Ajanki, L.~Erd\H{o}s, and T.~Kr\"{u}ger, \emph{Stability of the matrix
  {D}yson equation and random matrices with correlations}, Probab. Theory
  Related Fields \textbf{173} (2019), no.~1-2, 293--373.

\bibitem{AltGram}
J.~Alt, L.~Erd\H{o}s, and T.~Kr{\" u}ger, \emph{Local law for random {G}ram
  matrices}, Electron. J. Probab. \textbf{22} (2017), no.~25, 41 pp.

\bibitem{Altcirc}
J.~Alt, L.~Erd\H{o}s, and T.~Kr\"{u}ger, \emph{Local inhomogeneous circular
  law}, Ann. Appl. Probab. \textbf{28} (2018), no.~1, 148--203.

\bibitem{AEK_Shape}
J.~Alt, L.~Erd\H{o}s, and T.~Kr\"{u}ger, \emph{The {D}yson equation with linear
  self-energy: spectral bands, edges and cusps}, Doc. Math. \textbf{25} (2020),
  1421--1539. \MR{4164728}

\bibitem{AEKN_Kronecker}
J.~Alt, L.~Erd\H{o}s, T.~Kr\"{u}ger, and Yu. Nemish, \emph{Location of the
  spectrum of {K}ronecker random matrices}, Ann. Inst. Henri Poincar\'{e}
  Probab. Stat. \textbf{55} (2019), no.~2, 661--696. \MR{3949949}

\bibitem{AEKS2018}
J.~Alt, L.~Erd{\H o}s, T.~Krüger, and D.~Schröder, \emph{Correlated random
  matrices: Band rigidity and edge universality}, Ann. Probab. \textbf{48}
  (2020), no.~2, 963--1001.

\bibitem{AK_Corr_circ}
J.~Alt and T.~Kr\"{u}ger, \emph{Inhomogeneous circular law for correlated
  matrices}, J.\ Funct.\ Anal. \textbf{281} (2021), no.~7, Paper No. 109120,
  73. \MR{4271784}

\bibitem{AK_pseudo}
J.~Alt and T.~Kr\"{u}ger, \emph{Spectrum occupies pseudospectrum for random
  matrices with diagonal deformation and variance profile}, preprint (2024),
  \href{https://arxiv.org/abs/2404.17573}{arXiv:2404.17573}.

\bibitem{ArmitageGardiner}
D.~H. Armitage and S.~J. Gardiner, \emph{Classical potential theory}, Springer
  Monographs in Mathematics, Springer-Verlag London, Ltd., London, 2001.

\bibitem{bai1997}
Z.~D. Bai, \emph{Circular law}, Ann. Probab. \textbf{25} (1997), no.~1,
  494--529.

\bibitem{BelinschiYinZhong2024}
S.~Belinschi, Z.~Yin, and P.~Zhong, \emph{The {B}rown measure of a sum of two
  free random variables, one of which is triangular elliptic}, Adv. Math.
  \textbf{441} (2024), Paper No. 109562. \MR{4710866}

\bibitem{BelinschiSniadySpeicher2018}
S.~T. Belinschi, P.~\'{S}niady, and R.~Speicher, \emph{Eigenvalues of
  non-{H}ermitian random matrices and {B}rown measure of non-normal operators:
  {H}ermitian reduction and linearization method}, Linear Algebra Appl.
  \textbf{537} (2018), 48--83. \MR{3716236}

\bibitem{BianeLehner2001}
P.~Biane and F.~Lehner, \emph{Computation of some examples of {B}rown's
  spectral measure in free probability}, Colloq. Math. \textbf{90} (2001),
  no.~2, 181--211. \MR{1876844}

\bibitem{Birkhoff_Varga1958}
G.~Birkhoff and R.~S. Varga, \emph{Reactor criticality and nonnegative
  matrices}, J. Soc. Indust. Appl. Math. \textbf{6} (1958), 354--377.
  \MR{100984}

\bibitem{BordenaveCapitaine2016}
C.~Bordenave and M.~Capitaine, \emph{Outlier eigenvalues for deformed i.i.d.
  random matrices}, Comm. Pure Appl. Math. \textbf{69} (2016), no.~11,
  2131--2194. \MR{3552011}

\bibitem{BordenaveCaputoChafai2014}
C.~Bordenave, P.~Caputo, and D.~Chafa\"{\i}, \emph{Spectrum of {M}arkov
  generators on sparse random graphs}, Comm. Pure Appl. Math. \textbf{67}
  (2014), no.~4, 621--669. \MR{3168123}

\bibitem{BordenaveChafai2012}
C.~Bordenave and D.~Chafa\"{\i}, \emph{Around the circular law}, Probab. Surv.
  \textbf{9} (2012), 1--89. \MR{2908617}

\bibitem{MR1662382}
E.~Br\'ezin and S.~Hikami, \emph{Level spacing of random matrices in an
  external source}, Phys. Rev. E (3) \textbf{58} (1998), no.~6, 7176--7185.
  \MR{1662382}

\bibitem{BH}
\'E. Br\'ezin and S.~Hikami, \emph{{Universal singularity at the closure of a
  gap in a random matrix theory}}, Phys. Rev. E \textbf{57} (1998), no.~4,
  4140--4149.

\bibitem{Brown1986}
L.~G. Brown, \emph{Lidskiui's theorem in the type {II} case}, Geometric methods
  in operator algebras ({K}yoto, 1983), Pitman Res. Notes Math. Ser., vol. 123,
  Longman Sci. Tech., Harlow, 1986, pp.~1--35.

\bibitem{CampbellCipolloniErdosJi2024}
A.~Campbell, G.~Cipolloni, L.~Erd\H{o}s, and H.~C. Ji, \emph{On the spectral
  edge of non-{H}ermitian random matrices}, preprint (2024),
  \href{https://arxiv.org/abs/2404.17512}{arXiv:2404.17512}.

\bibitem{MR3500269}
M.~Capitaine and S.~P\'{e}ch\'{e}, \emph{Fluctuations at the edges of the
  spectrum of the full rank deformed {GUE}}, Probab. Theory Related Fields
  \textbf{165} (2016), no.~1-2, 117--161. \MR{3500269}

\bibitem{Cusp2EKS}
G.~Cipolloni, L.~{Erd{\H{o}}s}, T.~{Kr{\"u}ger}, and D.~Schr{\" o}der,
  \emph{Cusp universality for random matrices {II}: the real symmetric case},
  Pure and Applied Analysis \textbf{1} (2019), no.~4, 615--707.

\bibitem{Cipolloni_2020}
G.~Cipolloni, L.~Erd{\H o}s, and D.~Schröder, \emph{Edge universality for
  non-{H}ermitian random matrices}, Probability Theory and Related Fields
  \textbf{179} (2020), no.~1–2, 1–28.

\bibitem{Cook2018}
N.~Cook, W.~Hachem, J.~Najim, and D.~Renfrew, \emph{Non-{H}ermitian random
  matrices with a variance profile ({I}): deterministic equivalents and
  limiting {ESD}s}, Electron. J. Probab. \textbf{23} (2018), Paper No. 110, 61.

\bibitem{dubova2024bulkuniversalitycomplexeigenvalues}
S.~Dubova and K.~Yang, \emph{Bulk universality for complex eigenvalues of real
  non-symmetric random matrices with i.i.d. entries}, preprint (2024),
  \href{https://arxiv.org/abs/2402.10197}{arXiv:2402.10197}.

\bibitem{Dykema1993}
K.~Dykema, \emph{On certain free product factors via an extended matrix model},
  J. Funct. Anal. \textbf{112} (1993), no.~1, 31--60. \MR{1207936}

\bibitem{MR2198797}
K.~Dykema, \emph{Hyperinvariant subspaces for some {$B$}-circular operators},
  Math. Ann. \textbf{333} (2005), no.~3, 485--523, With an appendix by Gabriel
  Tucci. \MR{2198797}

\bibitem{MR2044226}
K.~Dykema and U.~Haagerup, \emph{Invariant subspaces of the quasinilpotent
  {DT}-operator}, J. Funct. Anal. \textbf{209} (2004), no.~2, 332--366.
  \MR{2044226}

\bibitem{ErdosJi2023}
L.~Erd\H{o}s and H.~C. Ji, \emph{Density of {B}rown measure of free circular
  {B}rownian motion}, Doc. Math. \textbf{30} (2025), no.~2, 417--453.
  \MR{4881011}

\bibitem{doi:10.1137/17M1143125}
L.~Erd\H{o}s, T.~Kr\"{u}ger, and D.~Renfrew, \emph{Power law decay for systems
  of randomly coupled differential equations}, SIAM J. Math. Anal. \textbf{50}
  (2018), no.~3, 3271--3290. \MR{3816180}

\bibitem{Cusp1EKS}
L.~{Erd{\H{o}}s}, T.~{Kr{\"u}ger}, and D.~Schr{\" o}der, \emph{Cusp
  universality for random matrices {I}: Local law and the complex {H}ermitian
  case}, Commun. Math. Phys. \textbf{378} (2020), 1203--1278.

\bibitem{Girko1984}
V.~L. Girko, \emph{The circular law}, Teor. Veroyatnost. i Primenen.
  \textbf{29} (1984), no.~4, 669--679. \MR{773436}

\bibitem{HaagerupLarsen2000}
U.~Haagerup and F.~Larsen, \emph{Brown's spectral distribution measure for
  {$R$}-diagonal elements in finite von {N}eumann algebras}, J. Funct. Anal.
  \textbf{176} (2000), no.~2, 331--367.

\bibitem{MR2339369}
U.~Haagerup and H.~Schultz, \emph{Brown measures of unbounded operators
  affiliated with a finite von {N}eumann algebra}, Math. Scand. \textbf{100}
  (2007), no.~2, 209--263. \MR{2339369}

\bibitem{HeltonRashidiFarSpeicher2007}
J.~W. Helton, R.~Rashidi~Far, and R.~Speicher, \emph{Operator-valued
  semicircular elements: solving a quadratic matrix equation with positivity
  constraints}, Int. Math. Res. Not. IMRN (2007), no.~22, Art. ID rnm086, 15.
  \MR{2376207}

\bibitem{HoZhong2023}
C.-W. Ho and P.~Zhong, \emph{Brown measures of free circular and multiplicative
  {B}rownian motions with self-adjoint and unitary initial conditions}, J. Eur.
  Math. Soc. (JEMS) \textbf{25} (2023), no.~6, 2163--2227. \MR{4592867}

\bibitem{Jain2021_non_Hermitian_random_band}
V.~Jain, I.~Jana, K.~Luh, and S.~O'Rourke, \emph{Circular law for random block
  band matrices with genuinely sublinear bandwidth}, J. Math. Phys. \textbf{62}
  (2021), no.~8, Paper No. 083306, 27. \MR{4300220}

\bibitem{Jana2022_nonHermitian_random_band}
I.~Jana, \emph{C{LT} for non-{H}ermitian random band matrices with variance
  profiles}, J. Stat. Phys. \textbf{187} (2022), no.~2, Paper No. 13, 25.
  \MR{4393055}

\bibitem{Khoruzhenko1996}
B.~Khoruzhenko, \emph{Large-{$N$} eigenvalue distribution of randomly perturbed
  asymmetric matrices}, J. Phys. A \textbf{29} (1996), no.~7, L165--L169.
  \MR{1395506}

\bibitem{MR3405746}
J.~O. Lee and K.~Schnelli, \emph{Edge universality for deformed {W}igner
  matrices}, Rev. Math. Phys. \textbf{27} (2015), no.~8, 1550018, 94.
  \MR{3405746}

\bibitem{liu2024repeatederfcstatisticsdeformed}
D.-Z. Liu and L.~Zhang, \emph{Repeated erfc statistics for deformed
  {G}in{UE}s}, preprint (2024),
  \href{https://arxiv.org/abs/2402.14362}{arXiv:2402.14362}.

\bibitem{maltsev2023bulkuniversalitycomplexnonhermitian}
A.~Maltsev and M.~Osman, \emph{Bulk universality for complex non-{H}ermitian
  matrices with independent and identically distributed entries}, preprint
  (2023), \href{https://arxiv.org/abs/2310.11429}{arXiv:2310.11429}.

\bibitem{MilnorMorseTheoryBook}
J.~Milnor, \emph{Morse theory}, Annals of Mathematics Studies, vol. No. 51,
  Princeton University Press, Princeton, NJ, 1963, Based on lecture notes by M.
  Spivak and R. Wells. \MR{163331}

\bibitem{MingoSpeicherBook}
J.~A. Mingo and R.~Speicher, \emph{Free probability and random matrices},
  Fields Institute Monographs, vol.~35, Springer, New York; Fields Institute
  for Research in Mathematical Sciences, Toronto, ON, 2017. \MR{3585560}

\bibitem{Osman_real}
M.~Osman, \emph{Bulk universality for real matrices with independent and
  identically distributed entries}, Electron. J. Probab. \textbf{30} (2025),
  Paper No. 5, 66. \MR{4848680}

\bibitem{PaulsenBook}
V.~Paulsen, \emph{Completely bounded maps and operator algebras}, Cambridge
  Studies in Advanced Mathematics, vol.~78, Cambridge University Press,
  Cambridge, 2002. \MR{1976867}

\bibitem{SaffTotik1997}
E.~B. Saff and V.~Totik, \emph{Logarithmic potentials with external fields},
  Grundlehren der mathematischen Wissenschaften [Fundamental Principles of
  Mathematical Sciences], vol. 316, Springer-Verlag, Berlin, 1997, Appendix B
  by Thomas Bloom. \MR{1485778}

\bibitem{Shlyakhtenko1996}
D.~Shlyakhtenko, \emph{Random {G}aussian band matrices and freeness with
  amalgamation}, Internat. Math. Res. Notices (1996), no.~20, 1013--1025.
  \MR{1422374}

\bibitem{Sniady2002}
P.~\'{S}niady, \emph{Random regularization of {B}rown spectral measure}, J.
  Funct. Anal. \textbf{193} (2002), no.~2, 291--313. \MR{1929504}

\bibitem{Sniday2003_Multinomial}
P.~\'Sniady, \emph{Multinomial identities arising from free probability
  theory}, J. Combin. Theory Ser. A \textbf{101} (2003), no.~1, 1--19.
  \MR{1953277}

\bibitem{Speicher1998}
R.~Speicher, \emph{Combinatorial theory of the free product with amalgamation
  and operator-valued free probability theory}, Mem. Amer. Math. Soc.
  \textbf{132} (1998), no.~627, x+88. \MR{1407898}

\bibitem{TakesakiBook}
M.~Takesaki, \emph{Theory of operator algebras. {I}}, Encyclopaedia of
  Mathematical Sciences, vol. 124, Springer-Verlag, Berlin, 2002, Reprint of
  the first (1979) edition, Operator Algebras and Non-commutative Geometry, 5.
  \MR{1873025}

\bibitem{Tao_2015}
T.~Tao and V.~Vu, \emph{Random matrices: Universality of local spectral
  statistics of non-hermitian matrices}, The Annals of Probability \textbf{43}
  (2015), no.~2.

\bibitem{tao2010}
T.~Tao, V.~Vu, and M.~Krishnapur, \emph{Random matrices: Universality of {ESD}s
  and the circular law}, Ann. Probab. \textbf{38} (2010), no.~5, 2023--2065.

\bibitem{MR1257246}
C.A. Tracy and H.~Widom, \emph{Level-spacing distributions and the {A}iry
  kernel}, Comm. Math. Phys. \textbf{159} (1994), no.~1, 151--174. \MR{1257246}

\bibitem{MR1385083}
C.A. Tracy and H.~Widom, \emph{On orthogonal and symplectic matrix ensembles},
  Comm. Math. Phys. \textbf{177} (1996), no.~3, 727--754. \MR{1385083}

\bibitem{Voiculescu1986}
D.~Voiculescu, \emph{Addition of certain noncommuting random variables}, J.
  Funct. Anal. \textbf{66} (1986), no.~3, 323--346. \MR{839105}

\bibitem{Voiculescu1991}
D.~Voiculescu, \emph{Limit laws for random matrices and free products}, Invent.
  Math. \textbf{104} (1991), no.~1, 201--220. \MR{1094052}

\bibitem{Voiculescu1995}
D.~Voiculescu, \emph{Operations on certain non-commutative operator-valued
  random variables}, no. 232, 1995, Recent advances in operator algebras
  (Orl\'eans, 1992), pp.~243--275. \MR{1372537}

\bibitem{zhang2024bulkuniversalitydeformedginues}
L.~Zhang, \emph{Bulk universality for deformed {G}in{UE}s}, preprint (2024),
  \href{https://arxiv.org/abs/2403.16120}{arXiv:2403.16120}.

\bibitem{Zhong2021}
P.~Zhong, \emph{{B}rown measure of the sum of an elliptic operator and a free
  random variable in a finite von {N}eumann algebra}, preprint (2021),
  \href{https://arxiv.org/abs/2108.09844}{arXiv:2108.09844}.

\end{thebibliography}
\end{document}